\documentclass[]{amsart}

\usepackage{amssymb}
\usepackage{amsfonts}
\usepackage{amsmath}
\usepackage{amsthm}
\usepackage{graphicx}
\usepackage{mathrsfs}
\usepackage{dsfont}
\usepackage{amscd}
\usepackage{multirow}
\usepackage[all]{xy}
\normalfont
\usepackage[T1]{fontenc}
\usepackage{calligra}
\usepackage{verbatim}
\usepackage[usenames,dvipsnames]{color}
\usepackage[colorlinks=true,linkcolor=Blue,citecolor=Violet]{hyperref}
\usepackage{enumerate}

\newcommand{\N}{{\mathds{N}}}

\newcommand{\Q}{{\mathds{Q}}}
\newcommand{\R}{{\mathds{R}}}
\newcommand{\C}{{\mathds{C}}}
\newcommand{\T}{{\mathds{T}}}

\newcommand{\D}{{\mathfrak{D}}}
\newcommand{\A}{{\mathfrak{A}}}
\newcommand{\B}{{\mathfrak{B}}}

\newcommand{\bigslant}[2]{{\raisebox{.2em}{$#1$}\left/\raisebox{-.2em}{$#2$}\right.}}

\newcommand{\Lip}{{\mathsf{L}}}
\newcommand{\Hilbert}{{\mathscr{H}}}

\newcommand{\dpropinquity}[1]{{\mathsf{\Lambda}^\ast_{#1}}}

\newcommand{\modpropinquity}[1]{{\mathsf{\Lambda}^{\mathsf{mod}}_{#1}}}
\newcommand{\dmodpropinquity}[1]{{\mathsf{\Lambda}^{\ast\mathsf{mod}}_{#1}}}
\newcommand{\dmetpropinquity}[1]{{\mathsf{\Lambda}^{\ast\mathsf{met}}_{#1}}}

\newcommand{\Kantorovich}[1]{{\mathsf{mk}_{#1}}}
\newcommand{\KantorovichMod}[1]{{\mathsf{k}_{#1}}}

\newcommand{\Haus}[1]{{\mathsf{Haus}_{#1}}}

\newcommand{\StateSpace}{{\mathscr{S}}}

\newcommand{\MongeKant}{{Mon\-ge-Kan\-to\-ro\-vich metric}}

\newcommand{\Qqcms}[1]{{$#1$}--\gQqcms}
\newcommand{\gQqcms}{quasi-Leibniz quantum compact metric space}

\newcommand{\gQVB}{metrized quantum vector bundle}
\newcommand{\QVB}[2]{$\left({#1},{#2}\right)$--\gQVB}

\newcommand{\qvba}[1]{{\mathrm{qvb}\left({#1}\right)}}
\newcommand{\gMVB}{metrical quantum vector bundle}
\newcommand{\MVB}[3]{$\left({#1},{#2},{#3}\right)$--\gMVB}
\newcommand{\qcms}{quantum compact metric space}

\newcommand{\unit}{1}

\newcommand{\sa}[1]{{\mathfrak{sa}\left({#1}\right)}}

\newcommand{\inner}[3]{{\left<{#1},{#2}\right>_{#3}}}

\newcommand{\bridge}[1]{#1} 
\newcommand{\dom}[1]{{\operatorname*{dom}\left({#1}\right)}}
\newcommand{\codom}[1]{{\operatorname*{codom}\left({#1}\right)}}
\newcommand{\diam}[2]{{\mathrm{diam}\left({#1},{#2}\right)}}

\newcommand{\norm}[2]{\left\|{#1}\right\|_{#2}}
\newcommand{\tunnelset}[3]{{\text{\calligra Tunnels}\,\left[{#1}\stackrel{#3}{\longrightarrow}{#2} \right]}}

\newcommand{\bridgereach}[2]{{\varrho\left(\bridge{#1}\middle\vert {#2}\right)}}

\newcommand{\bridgemodularreach}[1]{{\varrho^\sharp\left({#1}\right)}}
\newcommand{\bridgeheight}[2]{{\varsigma\left(\bridge{#1}\middle\vert{#2}\right)}}

\newcommand{\bridgeimprint}[1]{{\varpi\left(\bridge{#1}\right)}}
\newcommand{\bridgelength}[2]{{\lambda\left(\bridge{#1}\middle\vert{#2}\right)}}
\newcommand{\bridgemodularlength}[1]{{\lambda^\sharp\left(\bridge{#1}\right)}}

\newcommand{\bridgenorm}[2]{{\mathsf{bn}_{ \bridge{#1}  }\left({#2}\right)}}
\newcommand{\decknorm}[2]{{\mathsf{dn}_{ \bridge{#1}  }\left({#2}\right)}}

\newcommand{\Lie}[2]{{\left\{{#1},{#2}\right\}}} 

\newcommand{\targetsettunnel}[3]{{\mathfrak{t}_{#1}\left({#2}\middle\vert{#3}\right)}}

\newcommand{\CDN}{{\mathsf{D}}}
\newcommand{\CDNa}{{\mathsf{W}}}

\newcommand{\worknote}[1]{} 
\newcommand{\opnorm}[3]{{\left|\mkern-1.5mu\left|\mkern-1.5mu\left| {#1} \right|\mkern-1.5mu\right|\mkern-1.5mu\right|_{#3}^{#2}}}

\newcommand{\tunnellength}[1]{{\lambda\left({#1}\right)}}

\newcommand{\tunnelextent}[1]{{\chi\left({#1}\right)}}

\newcommand{\coh}[1]{{\mathrm{co}\left(#1\right)}}
\newcommand{\alg}[1]{{\mathfrak{#1}}}

\newcommand{\module}[1]{{\mathscr{#1}}}

\theoremstyle{plain}
\newtheorem{theorem}{Theorem}[section]

\newtheorem{corollary}[theorem]{Corollary}

\newtheorem{lemma}[theorem]{Lemma}
\newtheorem{proposition}[theorem]{Proposition}

\newtheorem{theorem-definition}[theorem]{Theorem-Definition}

\theoremstyle{definition}
\newtheorem{definition}[theorem]{Definition}

\newtheorem{example}[theorem]{Example}

\newtheorem{convention}[theorem]{Convention}

\theoremstyle{remark}

\newtheorem{remark}[theorem]{Remark}

\newtheorem{notation}[theorem]{Notation}

\renewcommand{\geq}{\geqslant}
\renewcommand{\leq}{\leqslant}

\numberwithin{equation}{section}

\allowdisplaybreaks[4]

\begin{document}

\title{The Dual Modular Gromov-Hausdorff propinquity and completeness}
\author{Fr\'{e}d\'{e}ric Latr\'{e}moli\`{e}re}
\email{frederic@math.du.edu}
\urladdr{http://www.math.du.edu/\symbol{126}frederic}
\address{Department of Mathematics \\ University of Denver \\ Denver CO 80208}

\date{\today}
\thanks{This work is part of the project supported by the grant H2020-MSCA-RISE-2015-691246-QUANTUM DYNAMICS and grant \#3542/H2020/2016/2 of the Polish Ministry of Science and Higher Education.}

\begin{abstract}
  We introduce in this paper the dual modular propinquity, a complete metric, up to full modular quantum isometry, on the class of {\gQVB s}, i.e. of Hilbert modules endowed with a type of densely defined norm, called a D-norm, which generalize the operator norm given by a connection on a Riemannian manifold. The dual modular propinquity is weaker than the modular propinquity yet it is complete, which is the main purpose of its introduction. Moreover, we show that the modular propinquity can be extended to a larger class of objects which involve {\qcms s} acting on {\gQVB s}.
\end{abstract}
\maketitle



\section{Introduction}

The dual modular Gromov-Hausdorff propinquity is a noncommutative analogue of the Gromov-Hausdorff distance for the class of {\gMVB s}, which are modules over {\qcms s}, i.e. noncommutative analogues of compact metric spaces. The purpose for the introduction of this metric is at least twofold. First, our project in noncommutative geometry aims at providing an analytical framework for the study of theories in mathematical physics as objects in a larger space, endowed with at least a topology induced by metrics of the kind introduced here and along our work in \cite{Latremoliere13,Latremoliere13b,Latremoliere14,Latremoliere15,Latremoliere16c,Latremoliere18b}. Indeed, some of the literature \cite{Grosse97,Connes97,Seiberg99,Schwarz98,Madore91,Madore,Wallet12,dandrea13} in quantum field theory strongly suggest the importance of metrics in the construction of finite-dimensional approximations. Second, and more practically, the metric in this manuscript is a complete metric up to full modular quantum isometry, dominated by the modular propinquity of \cite{Latremoliere16c} which we do not believe to be complete, and the dual modular propinquity is the foundation for our work on a Gromov-Hausdorff distance on spectral triples and other noncommutative differential structures, as discussed in \cite{Latremoliere18g}.

The construction of a Gromov-Hausdorff-like metric in noncommutative geometry presents several delicate issues. The general pattern \cite{Rieffel00,kerr02,li03,Wu06b,kerr09,Latremoliere13,Latremoliere13c} consists in finding an appropriate notion of isometric embeddings for certain structures endowed with some metric, and then take the infimum over all such embeddings of a number meant to measure how apart the structures are. Finding the right notions to make the above scheme function is however not obvious in noncommutative geometry, where quantum spaces are described only via noncommutative algebras seen as analogues of algebras of functions with appropriate geometric or analytical properties. In particular, even if an abstract notion of isometric embedding seems suitable, it also is essential that some prescription on how to build such embeddings be a foundational component of our work. A pattern which has served us well is to first introduce rather specific means to tell how far two structures may be --- we often call these means ``bridges'' --- and use them to define a metric for which most of our examples are established. Bridges are not per say given by isometric embedding, one the surface at least. Then, once this first set of tools is proven to at least be usable, we proceed to define a  more abstract metric which enjoys more pleasant properties, such as completeness, where explicit isometric embeddings are defined, usually under the terminology of ``tunnels''. The new, tunnel-based metric is more theoretically pleasing while always dominated by the bridge-based metric --- because we prove that bridges do in fact provide explicit tunnels, and in fact it is how we usually prove convergence for either metrics. Both are essential to our work so far.

The resulting family of metrics, which we named the Gromov-Hausdorff propinquities, has been showed to enjoy various nice properties and applications. We know of several convergence and finite dimensional approximations results, such as the continuity of quantum tori and approximations of quantum tori by fuzzy tori \cite{Latremoliere05,Latremoliere13c} (later on approached using our propinquity and different techniques in \cite{Junge16}, with potential connection to quantum information theory), continuity for families of AF algebras \cite{Latremoliere15}, full matrix approximations of coadjoint orbits for semi-simple Lie groups \cite{Rieffel15},  continuity for conformal and other types of deformations \cite{Latremoliere16}. We have a noncommutative analogue of Gromov's compactness theorem \cite{Latremoliere15c} and conditions for preservation of symmetries \cite{Latremoliere17c}. Our constructions have proven to be flexible enough to incorporate additional structures besides metric. Notably, we extended our propinquity to a covariant version \cite{Latremoliere18b,Latremoliere18c} defined on classes of quantum dynamical systems (i.e. proper metric semigroups and groups actions on {\qcms s}), and, central to this work, to modules \cite{Latremoliere16c,Latremoliere17a,Latremoliere18a, Rieffel17a}.

We followed our pattern when working with {\qcms s}, and then again with modules over {\qcms s}. Defining convergence for modules over quantum metric spaces, or even classical metric spaces, is not a well explored issue, and thus progress in this direction is of evident interest in sight of the important role that modules play in noncommutative geometry --- and vector bundles play in classical geometry. Lacking a commutative model led us to first use our bridges between {\qcms s} from \cite{Latremoliere13} to define a sort of modular bridges between new structures called {\gQVB s}, which are Hilbert modules over {\qcms s} endowed with a particular norm \cite{Latremoliere16c}. We then were able to prove that such modular bridges allow to define a metric, the modular propinquity, for which we then were able to prove the continuity of the family of Heisenberg modules over quantum tori, which is a rather involved process. The modular propinquity, just as the quantum propinquity, is not known to be complete, and has a technical description involving finite paths made of bridges between modules, but it provides an explicit, core toolkit to prove convergence of modules.

Yet, as a likely not complete metric based on paths of arbitrary lengths, the modular propinquity is less suited as a general analytical tool. Instead, as we have done before, we introduce a weaker metric on {\gQVB s}, which is complete, and still a metric up to the same appropriate notion of isomorphism as the modular propinquity. Being complete, this new metric is a good analytical tool, where for instance, the study of compactness of classes of {\gQVB s} becomes much more amenable, and we may hope to use of our metric to eventually prove the existence of certain spaces with desirable properties as limits of, say, finite dimensional physical models. Thus it is an important step for our program. This constitutes the initial theoretical value of the dual modular propinquity.

Several important objectives for our program also rely on the construction in this paper. First, the structure of a metrical quantum vector bundle is needed to define the notion of convergence for metric spectral triples. A spectral triple $(\A,\Hilbert,D)$ is given by a unital C*-algebra $\A$ acting on a Hilbert space $\Hilbert$ and an unbounded self-adjoint operator $D$ subject to some conditions, chief among them that $[D,a]$ is an operator with bounded closure for all $a$ in a dense *-subalgebra of $\A$. A spectral triple is metric when the seminorm $\Lip : a\in\A \mapsto \opnorm{[D,a]}{}{\Hilbert}$ (allowing for $\infty$) is an L-seminorm. Now, a metric spectral triple defines a metrical vector bundle, with $\Hilbert$ seen as a module over $\A$ --- but \emph{not} a Hilbert modules over $\A$, but rather a Hilbert module over $\C$. This flexibility is crucial to obtain a nontrivial metric between spectral triples. Moreover, the construction in this work does not involve the notion of ``modular treks'', which actually creates some difficulties when working with group actions, which we also need for our work with spectral triples. Thus the more direct approach taken here is very helpful moving forward. 

We hope to be able to use completeness of some classes of metric spectral triples to construct new spectral triples from Cauchy sequences of spectral triples. To this end, as the spectral propinquity is built on top of the metric on {\gMVB s} in this paper, we certainly hope that the metrics in this paper are complete. This issue is important, for instance, to try and obtain spectral triples on objects for which natural approximations exist, on which constructing a spectral triple is relatively easier than on the limit. A first step in this direction can be found in our work on the convergence of sums of spectral triples on curve to the spectral triples on certain fractals such as the Sierpinsky triangle and the Kigami triangle \cite{Latremoliere20a}. Completeness may help extend these ideas to more complicated fractals, by starting with spectral triples on simpler approximating spaces.

In a different direction, another source of {\gQVB s} is given by Dirichlet forms and their noncommutative analogues. A Dirichlet form is a particular quadratic form densely defined on some Hilbert space. Many examples of Dirichlet forms can be used to defined {\gQVB s}. Completeness of the modular propinquity can thus be seen as a tool to construct new Dirichlet forms by a Cauchy sequence argument. An example of potential application of this idea is to fractals, where Dirichlet forms are the starting point for the construction of differential calculi. More generally, noncommutative analogues of Dirichlet forms can be used to define differential calculi and provide another road to noncommutative geometry \cite{Sauvageot91}. Once more, this opens a new way to construct Dirichlet forms on fractals, for instance, and to elucidate some of the known constructions. In summary, the completeness of our metric makes it possible to use our work to explore the construction of new differential structures in noncommutative geometry.

On a more immediate, practical level, we were motivated to define the modular propinquity to avoid the use of modular treks in \cite{Latremoliere16c}, because they would raise some technical difficulties in our definition of the Gromov-Hausdorff distance for spectral triples in \cite{Latremoliere18g}. We furthermore are not aware that a Gromov-Hausdorff distance on vector bundles of the kind presented here and in \cite{Latremoliere16c}, has been introduced even in the commutative setting, and thus we intend to pursue applications of the present work to the metric approach to Riemannian geometry in a future work. Last, we would note that the fact we are able to use patterns which have emerged through our research to build the dual modular propinquity which then enjoys the basic properties we wish for is a good sign. Together with analogue results concerning the covariant propinquity, this paper reflects that our general approach to noncommutative metric geometry seems to bear fruits.

We begin this paper with the necessary framework of noncommutative metric geometry, starting with the notion of a {\qcms}, and moving quickly to {\gQVB s}. In working on this paper, we noticed that we could actually relax one condition from our original definition of {\gQVB s} in \cite{Latremoliere16c}. This turns out to be a valuable observation, as elements of the proof in \cite{Latremoliere16c} which relied on this now extra, redundant condition when working with {\gQVB s} can be recycled to extend the dual-modular propinquity to a much larger class of objects which we call {\gMVB s}, which play the key role in our work on the convergence of spectral triples in \cite{Latremoliere18g}. Informally, {\gMVB s} consist of a {\qcms} acting on a left module which is also a {\gQVB} over a \emph{different} {\qcms}.

We then define the dual modular propinquity for {\gQVB s}. As we discussed previously, we introduce a notion of modular tunnels, and prove that modular bridges in the sense of \cite{Latremoliere16c} provide such modular tunnels. Interestingly, the {\gQVB s} constructed from modular bridges do not meet our old definition in \cite{Latremoliere16c} though they do meet our relaxed notion in this paper. We also prove that modular tunnels can be almost composed, in the spirit of \cite{Latremoliere14}, which allows us to show that the dual-modular propinquity is at least a pseudo-metric on the class of {\gQVB s}. We then proceed to prove that our new metric is indeed a metric up to full modular quantum isometry, as desired. The proof follows a scheme similar to \cite{Latremoliere16c}, to which we refer whenever appropriate. However, we take care to include enough information to show how we can circumvent the now missing condition from our updated definition of {\gQVB s}.

We then extend our work to {\gMVB s} --- since these structures extend {\gQVB s}. As will be the pattern in this paper, we prefer to provide detailed proofs for {\gQVB s} and then only indicate what needs to be added in order to work with {\gMVB s}, so that notations remain simple. In general, this extension process is actually the easier part of this work, which is actually a rather reassuring observation that our definition is sensible.

The last section of this paper discusses completeness for the dual-modular propinquity. We start from the conclusion of \cite{Latremoliere13b} and proceed to show that the dual-modular propinquity is a complete metric, and remains so when extended to {\gMVB s}. This is the property which largely justifies the introduction of the dual-modular propinquity in general.

\section{Metrized and Metrical Quantum Vector Bundles}

Noncommutative metric geometry is the study of noncommutative analogues of the algebras of Lipschitz functions over metric spaces. We call these analogues {\qcms s}. This term appeared in \cite{Connes89}, and our current definition is the result of an evolution \cite{Rieffel98a,Rieffel99,Rieffel10c,Latremoliere13,Latremoliere15}, led by the needs of our work to both include as many valuable examples as possible while being a solid foundation for the rest of our work. Quantum compact metric spaces are pairs of a unital C*-algebra $\A$ and a noncommutative Lipschitz seminorm, which means a seminorm which induces by duality a metric for the weak* topology of the state space $\StateSpace(\A)$, while satisfying some form of a Leibniz inequality, parametrized by so-called permissible functions. The fundamental examples are given by the pairs $(C(X),\Lip)$ of the C*-algebra $\C$-valued continuous over a compact metric space $(X,d)$, with the usual Lipschitz seminorm $\Lip$ (allowing for $\infty$) induced by $d$ on $C(X)$. While the Lipschitz seminorms satisfy the usual form of the Leibniz inequality, other examples, such as the {\qcms s} $(\A,\Lip)$ where $\A$ are unital AF C*-algebras with a faithful tracial state, constructed in \cite{Latremoliere15} and including such examples as $UHF$ algebras and Effros-Shen algebras, satisfy a weakened form of the Leibniz inequality. Our work is designed to accommodate such examples, by fixing in advance a particular form of a Leibniz inequality by means of what we call permissible functions. Some form of Leibniz inequality is absolutely necessary as it is the tool we use to prove that distance zero for the propinquity implies the existence of a full quantum isometry.

\begin{notation}
  We denote the norm of any vector space $E$ by $\norm{\cdot}{E}$ unless otherwise specified.

  If $\A$ is a unital C*-algebra, then we denote its state space as $\StateSpace(\A)$, its space $\{a\in\A: a^\ast = a\}$ of self-adjoint elements by $\sa{\A}$, and its unit by $\unit_\A$. 

  Moreover, if $a\in\A$, we denote its real part $\frac{1}{2}(a+a^\ast)$ as $\Re a$ and its imaginary part $\frac{1}{2i}(a-a^\ast)$ as $\Im a$. A quick computation shows that for all $a\in\A$:
  \begin{equation*}
    \max\left\{ \norm{\Re a}{\A},\norm{\Im a}{\A} \right\} \leq \norm{a}{\A} \leq \sqrt{2}\max\left\{ \norm{\Re a}{\A}, \norm{\Im a}{\A} \right\} \text{.}
  \end{equation*}
\end{notation}

\begin{definition}
  A function $F : [0,\infty)^4\rightarrow [0,\infty)$ is \emph{permissible} when it is weakly increasing from $[0,\infty)^4$ endowed with the product order, and such that:
  \begin{equation*}
    \forall x,y,l_x,l_y \geq 0 \quad F(x,y,l_x,l_y) \geq x l_y + y l_x \text{.}
  \end{equation*}
\end{definition}

\begin{definition}[{\cite{Rieffel98a,Latremoliere13,Latremoliere15}}]\label{qcms-def}
  A \emph{\Qqcms{F}} $(\A,\Lip)$, where $F$ is a permissible function, is a unital C*-algebra $\A$, a seminorm $\Lip$ defined on a dense Jordan-Lie subalgebra $\dom{\Lip}$ of $\sa{\A}$ and a function $F : [0,\infty)^4\rightarrow [0,\infty)$ such that:
  \begin{enumerate}
  \item $\left\{ a \in \dom{\Lip} : \Lip(a) = 0 \right\} = \R\unit_\A$,
  \item the {\MongeKant} $\Kantorovich{\Lip}$ defined between any two states $\varphi,\psi \in \StateSpace(\A)$ by:
    \begin{equation*}
      \Kantorovich{\Lip}(\varphi,\psi) = \sup\left\{ \left|\varphi(a) - \psi(a)\right| : \Lip(a) \leq 1 \right\}
    \end{equation*}
    metrizes the weak* topology on $\StateSpace(\A)$,
  \item $\Lip$ is lower semi-continuous with respect to $\norm{\cdot}{\A}$,
  \item for all $a,b \in \dom{\Lip}$, we have
    \begin{equation*}
      \max\left\{ \Lip\left(\Re(a b)\right), \Lip\left(\Im(a b)\right) \right\} \leq F(\norm{a}{\A},\norm{b}{\A},\Lip(a), \Lip(b))\text{.}
    \end{equation*}
  \end{enumerate}
  If $(\A,\Lip)$ is a {\Qqcms{F}}, then $\Lip$ is called a \emph{$F$-quasi-Leibniz L-seminorm}.
\end{definition}

There are many examples of {\qcms s}: besides all compact metric spaces, noncommutative examples include quantum tori \cite{Rieffel98a,Rieffel02}, curved quantum tori \cite{Latremoliere16}, hyperbolic group C*-algebras \cite{Ozawa05}, nilpotent group C*-algebras \cite{Rieffel15b}, AF algebras \cite{Latremoliere15}, noncommutative solenoids \cite{Latremoliere16b}, Podl{\`e}s spheres \cite{Kaad18}, and more. Some of these examples, such as AF algebras \cite{Latremoliere15c}, are quasi-Leibniz but not Leibniz. A locally compact theory has emerged \cite{Latremoliere05b,Latremoliere12b}, though this paper focuses on the compact theory.

In this paper, we are concerned with the study of appropriate classes of modules over {\qcms s}. Appropriate here will mean that they are endowed with some metric structure, in a manner inspired by the definition of {\qcms s}. As discussed in \cite{Latremoliere16c}, our model is given by hermitian $\C$-vector bundles over Riemannian manifolds, endowed with a choice of a metric connection. The module of continuous sections of a hermitian vector bundle over a compact manifold is naturally a Hilbert module over the C*-algebra of continuous functions over its base, where Hilbert modules are defined as follows.

\begin{definition}[{\cite{Rieffel74}}]
  A \emph{Hilbert module} $\left(\module{M},\inner{\cdot}{\cdot}{\module{M}}\right)$ over a C*-algebra $\A$ is a left $\A$-module, and a $\C$-bilinear map:
    \begin{equation*}
      \inner{\cdot}{\cdot}{\module{M}} : \module{M}\times\module{M} \rightarrow \A
    \end{equation*}
    such that:
    \begin{enumerate}
      \item $\forall \omega,\eta \in \module{M}, a \in \A \quad \inner{a\omega}{\eta}{\module{M}} = a\inner{\omega}{\eta}{\module{M}}$,
      \item $\forall \omega,\eta \in \module{M} \quad \inner{\omega}{\eta}{\module{M}} = \inner{\eta}{\omega}{\module{M}}^\ast$,
      \item $\forall \omega\in \A \quad \inner{\omega}{\omega}{\module{M}} \geq 0$,
      \item $\forall \omega\in \A \quad \inner{\omega}{\omega}{\module{M}} = 0 \iff \omega = 0$,
      \item $\module{M}$ is a Banach space when endowed with the norm $\omega\in\module{M} \mapsto \norm{\omega}{\module{M}} = \sqrt{\norm{\inner{\omega}{\omega}{\module{M}}}{\A}}$.
    \end{enumerate}
\end{definition}

\begin{convention}
  When $(\module{M},\inner{\cdot}{\cdot}{\module{M}})$ is a Hilbert module, its Hilbert norm $\sqrt{\norm{\inner{\cdot}{\cdot}{\module{M}}}{\A}}$ is always denoted by $\norm{\cdot}{\module{M}}$.
\end{convention}

We will use the following natural notion of a morphism of module, allowing for base algebra changes.

\begin{definition}
  Let $\A$ and $\B$ be two unital C*-algebra.  A \emph{module morphism} $(\theta,\Theta)$ from a left $\A$-module $\module{M}$ to a left $\B$-module $\module{N}$ is a unital *-morphism $\theta : \A \rightarrow \B$ and a linear map $\Theta: \module{M} \rightarrow \module{N}$ such that:
  \begin{equation*}
    \forall a \in \A, \omega \in \module{M} \quad \Theta(a \omega) = \theta(a) \Theta(\omega)\text{.}
  \end{equation*}

  A \emph{Hilbert module morphism} $(\theta,\Theta)$ from a Hilbert $\A$-module $(\module{M},\inner{\cdot}{\cdot}{\module{M}})$ to a Hilbert $\B$-module $(\module{N},\inner{\cdot}{\cdot}{\module{N}})$ is a modular morphism $(\theta,\Theta)$ from $\module{M}$ to $\module{N}$ such that:
    \begin{equation*}
      \forall \omega,\eta \in \module{M} \quad \theta\left(\inner{\omega}{\eta}{\module{M}}\right) = \inner{\Theta(\omega)}{\Theta(\eta)}{\module{N}}\text{.}
    \end{equation*}

    A \emph{Hilbert module isomorphism} $(\theta,\Theta)$ is a Hilbert module morphism where $\theta$ and $\Theta$ are both bijections.
\end{definition}

We note that by a standard argument, if $(\theta,\Theta)$ is a Hilbert module isomorphism, then $(\theta^{-1},\Theta^{-1})$ is also a Hilbert module isomorphism.

A {\gQVB} is a Hilbert module over a {\qcms} which is, in addition, equipped with a kind of norm which captures some of the metric information encoded, for instance, in a connection over a Hermitian bundle, as explained  in \cite[Example 3.10]{Latremoliere16c}.

\begin{definition}
  A pair $(F,H)$ of functions is \emph{permissible} when $F$ is permissible and $H : [0,\infty)^2 \rightarrow  [0,\infty)$ which is weakly increasing from $[0,\infty)^2$ endowed with the product order, and such that:
  \begin{equation*}
    \forall x,y \geq  0 \quad H(x,y) \geq 2 x y \text{.}
  \end{equation*}
\end{definition}

\begin{definition}\label{metrized-bundle-def}
  A \emph{\QVB{F}{H}} $(\module{M},\inner{\cdot}{\cdot}{\module{M}},\CDN,\A,\Lip_\A)$, where $(F,H)$ is a permissible pair of functions, is given by:
  \begin{enumerate}
    \item a {\Qqcms{F}} $(\A,\Lip_\A)$ called the \emph{base quantum space},
    \item a Hilbert $\A$-module $(\module{M},\inner{\cdot}{\cdot}{\module{M}})$,
    \item a norm $\CDN$ on a dense, $\C$-linear space $\dom{\CDN}$ of $\module{M}$ such that:
      \begin{enumerate}
      \item $\forall \omega \in \dom{\CDN} \quad \norm{\omega}{\module{M}} \leq \CDN(\omega)$,
      \item $\left\{ \omega \in \dom{\CDN} : \CDN(\omega) \leq 1 \right\}$ is compact for $\norm{\cdot}{\module{M}}$,
      \item $\forall \omega,\eta\in \module{M} \quad \max\left\{ \Lip_\A\left(\Re \inner{\omega}{\eta}{\module{M}}\right), \Lip_\A\left(\Im \inner{\omega}{\eta}{\module{M}}\right) \right\} \leq H (\CDN(\omega) , \CDN(\eta))$ --- this inequality is referred to as the \emph{inner quasi-Leibniz inequality}.
      \end{enumerate}
  \end{enumerate}
  If $(\module{M},\inner{\cdot}{\cdot}{\module{M}},\CDN,\A,\Lip_\A)$ is a {\gQVB}, then $\CDN$ is called an \emph{$H$-quasi-Leibniz $D$-norm}.
\end{definition}

Definition (\ref{metrized-bundle-def}) is a relaxed version of \cite[Definition 3.8]{Latremoliere16c}, where we removed the so-called modular quasi-Leibniz requirement on D-norms. There are two reasons to do so. First, we will see that we do not require it for our construction --- or for the constructions in \cite{Latremoliere16c} for that matter. Second, modular tunnels obtained from modular bridges will not satisfy this extra condition. Nonetheless, we will return to this matter when working with {\gMVB s} below. 

For our purpose, the proper notion of morphisms, and isomorphisms, for {\gQVB s} is given as follows:

\begin{definition}\label{modular-isometry-def}
  Let $\mathds{A} = (\module{M},\inner{\cdot}{\cdot}{\module{M}},\CDN_{\module{M}},\A,\Lip_\A)$ and $\mathds{B} = (\module{N},\inner{\cdot}{\cdot}{\module{N}},\CDN_{\module{N}},\B,\Lip_\B)$ be two {\gQVB s}. An \emph{modular quantum isometry} $(\theta,\Theta)$ from $\mathds{A}$ to $\mathds{B}$ is given by a Hilbert module morphism $(\theta,\Theta) : (\module{M},\inner{\cdot}{\cdot}{\module{M}}) \rightarrow (\module{N},\inner{\cdot}{\cdot}{\module{N}})$ such that:
  \begin{enumerate}
    \item $\theta : (\A,\Lip_\A) \twoheadrightarrow (\B,\Lip_\B)$ is a quantum isometry,
    \item $\Theta : \module{M} \twoheadrightarrow \module{N}$ is a surjective map,
    \item for all $\omega \in \dom{\module{N}}$, we have:
      \begin{equation*}
        \CDN_{\module{N}}(\omega) = \inf\left\{ \CDN_{\module{M}}(\eta) : \Theta(\eta) = \omega \right\} \text{.}
      \end{equation*}
    \end{enumerate}

    A \emph{full modular quantum isometry} $(\theta,\Theta)$ from $\mathds{A}$ to $\mathds{B}$ is a Hilbert module isomorphism from $(\module{M},\inner{\cdot}{\cdot}{\module{M}})$ to $(\module{N},\inner{\cdot}{\cdot}{\module{N}})$ such that $\theta$ is a full quantum isometry, and
    \begin{equation*}
      \forall \omega \in \module{M} \quad \CDN_{\module{N}}\circ\Theta(\omega) = \CDN_{\module{M}}(\omega)  \text{.}
    \end{equation*}
\end{definition}

\begin{remark}\label{reached-remark}
  Let $\mathds{A} = (\module{M},\inner{\cdot}{\cdot}{\A},\CDN_{\module{M}},\A,\Lip_\A)$ and $\mathds{B} = (\module{N},\inner{\cdot}{\cdot}{\B},\CDN_{\module{N}},\B,\Lip_\B)$ be two {\gQVB s} and $(\theta,\Theta)$ a modular quantum isometry from $\mathds{A}$ to $\mathds{B}$. Let $\omega\in\dom{\CDN_{\module{N}}}$. By Definition (\ref{modular-isometry-def}), for all $n\in\N$, there exists $\eta_n \in \dom{\CDN_{\module{M}}}$ such that $\Theta(\eta_n) = \omega$ and $\CDN_{\module{N}}(\omega) \leq \CDN_{\module{M}}(\eta_n) \leq \CDN_{\module{N}}(\omega) + \frac{1}{n+1}$. As $\left\{ \eta \in\module{M} : \CDN_{\module{M}}(\eta) \leq \CDN_{\module{N}}(\omega) + 1 \right\}$ is compact, we then conclude that there exists a subsequence of $(\eta_n)_{n\in\N}$ converging to some $\eta$; by lower semi-continuity of $\CDN_{\module{M}}$,  we have $\CDN_{\module{M}}(\eta) \leq \CDN_{\module{N}}(\omega)$, while by continuity of $\Theta$, we have $\Theta(\eta) = \omega$ --- the latter implying that $\CDN_{\module{N}}(\omega)\leq \CDN_{\module{N}}(\eta)$ by Definition (\ref{modular-isometry-def}). 

 Hence we have shown that for all $\omega\in\dom{\CDN_{\module{N}}}$, there exists $\eta\in\module{M}$ such that $\Theta(\eta)=\omega$ and $\CDN_{\module{M}}(\eta)=\CDN_{\module{N}}(\omega)$.

The same reasoning applies to prove that for all $b \in \dom{\Lip_\B}$, there exists $a\in\dom{\Lip_\A}$  with $\Lip_\B(b) = \Lip_\A(a)$ and $\theta(a) = b$.
\end{remark}

Motivated by the study of convergence for spectral triples and other noncommutative differential structures, we wish to be able to work with a somewhat more general structure than {\gQVB s} where a {\qcms} acts on a module which is not a Hilbert module for this action --- yet is a Hilbert module over another {\qcms} \cite{Latremoliere18g}. Indeed, a spectral triple $(\A,\Hilbert,D)$ is given by a unital C*-algebra represented on a Hilbert space $\Hilbert$, together with a self-adjoint unbounded operator $D$ on  $\Hilbert$ subject to properties generalizing some properties of Dirac operators on connected compact Riemannian spin manifolds. Now, $\Hilbert$ is a module for $\A$, but not a Hilbert module for $\A$ in general since the inner product is not $\A$-linear --- for instance, if $(M,g)$ is a connected compact spin manifold with spinor bundle $S$ and Dirac operator $D$ acting on the Hilbert space of square integrable sections $L^2(S)$ of $S$, then $L^2(S)$ is not a Hilbert module over $C(X)$. This could be remedied in this special case, by replacing the inner product on $L^2(S)$ with a $C(M)$-valued inner product, and indeed this is the path we take to define {\gQVB s}. But this is a departure from the structure of a spectral triple. For instance, for the usual spectral triple $(\A_\theta,L^2(\T^2),D)$ over the quantum two torus $\A_\theta$ for $\theta\in\R\setminus\Q$, then there is no obvious way to replace the inner product on $L^2(\T^2)$ into a Hilbert module over $\A_\theta$.

Of course, any Hilbert space is a Hilbert module over the C*-algebra $\C$, which is trivially a {\qcms} in a unique way (for the L-seminorm $0$). So we need to work with an extension of {\gQVB s} where the module we work with is, at once, a module over a C*-algebra, and a Hilbert module over $\C$ --- that is, a Hilbert space. Now, as we shall see later on in the proof of the triangle inequality for the metrics we introduce in this work (see Theorem (\ref{triangle-thm})), we then will need to work for Hilbert modules over not only $\C$, but also $\C\oplus\C$, $\C\oplus\C\oplus\C$, and so on. 

We are led to the following definition (where incidentally, the condition which we removed from the definition of {\gQVB} in \cite[Definition 3.8]{Latremoliere16c} now re-appears). 

\begin{definition}
  A triple $(F,G,H)$ of functions is \emph{permissible} if $(F,H)$ is a permissible pair of function, and $G : [0,\infty)^3 \rightarrow [0,\infty)$ is weakly increasing from $[0,\infty)^3$ endowed with the product order, and such that:
  \begin{equation*}
    \forall x,y,z \geq 0 \quad G(x,y,z) \geq (x + y) z \text{.}
  \end{equation*}
\end{definition}

\begin{definition}\label{metrical-bundle-def}
  Let $(F,G,H)$ be a permissible triple of functions. A \emph{\MVB{F}{G}{H}} $(\module{M},\inner{\cdot}{\cdot}{\module{M}},\CDN,\A,\Lip_\A,\B,\Lip_\B)$ is given by:
  \begin{enumerate}
  \item a {\QVB{F}{H}} $(\module{M},\inner{\cdot}{\cdot}{\module{M}},\CDN,\A,\Lip_\A)$,
  \item a {\Qqcms{F}} $(\B,\Lip_\B)$,
  \item a unital *-morphism from $\B$ to the C*-algebra of adjoinable operators on $\module{M}$ --- we will regard $\module{M}$ as a left $\B$-module with no explicit notation for this *-homomorphism,
  \item $\forall b \in \dom{\Lip_\B}, \omega \in \dom{\CDN} \quad \CDN(b\omega)\leq G(\norm{b}{\B},\Lip_\B(b), \norm{\omega}{\module{M}})$.
  \end{enumerate}
\end{definition}

\begin{remark}
  If $(\module{M},\inner{\cdot}{\cdot}{\module{M}},\CDN,\A,\Lip_\A,\B,\Lip_\B)$ is a {\gMVB} then note that for all $\omega\in\module{M}$ and $b \in \B$, we have $\norm{b\omega}{\module{M}} \leq \norm{b}{\B} \norm{\omega}{\module{M}}$ by Definition (\ref{metrical-bundle-def}).
\end{remark}

For clarity, we find it helpful to follow a pattern in this paper whereby we first establish definitions and results for {\gQVB s} and then extend them to {\gMVB s}.

\bigskip

Our focus is to turn the class of {\qcms s}, {\gQVB s} and even {\gMVB s} into metric spaces, so that we can study these objects using such techniques as topological approximations. We begin by recalling from \cite{Latremoliere13,Latremoliere13b,Latremoliere14,Latremoliere15} how to do so for {\qcms s}, leading to the introduction of the dual Gromov-Hausdorff propinquity.

The Gromov-Hausdorff distance between two compact metric spaces $(X,d_X)$ and $(Y,d_Y)$ is the infimum, over all compact metric space $(Z,d_Z)$ containing an isometric copy of $(X,d_X)$ and $(Y,d_Y)$, of the Hausdorff distance between the isometric copies of $X$ and $Y$ in $(Z,d_Z)$. As a first step of our construction, we introduce the notion of a tunnel as the dual notion of an isometric embedding of two compact metric spaces into a third one, extended to our noncommutative setting.

\begin{definition}\label{tunnel-def}
  Let $(\A_1,\Lip_1)$ and $(\A_2,\Lip_2)$ be two {\Qqcms{F}s}. An \emph{$F$-tunnel} $(\D,\Lip_\D,\pi_1,\pi_2)$ from $(\A_1,\Lip_1)$ to $(\A_2,\Lip_2)$ is given by:
  \begin{enumerate}
    \item $(\D,\Lip_\D)$ is a {\Qqcms{F}},
    \item for $j \in \{1,2\}$, the map $\pi_j : (\D,\Lip_\D) \twoheadrightarrow (\A_j,\Lip_j)$ is a quantum isometry.
  \end{enumerate}
  The \emph{domain} $\dom{\tau}$ of $\tau$ is $(\A,\Lip_\A)$ and the \emph{codomain} $\codom{\tau}$ of $\tau$ is $(\B,\Lip_\B)$. 
\end{definition}

We associate to tunnels a number meant to measure how far apart its domain and codomain are --- taking the Hausdorff distance between certain state spaces of {\qcms s}, seen as metric spaces when endowed with their {\MongeKant}.

\begin{definition}\label{extent-def}
  The \emph{extent} $\tunnelextent{\tau}$ of a tunnel $\tau = (\D,\Lip_\D,\pi_1,\pi_2)$ from $(\A_1,\Lip_1)$ to $(\A_2,\Lip_2)$ is:
  \begin{equation*}
    \tunnelextent{\tau} = \max_{j\in\{1,2\}} \Haus{\Kantorovich{\Lip_\D}}\left( \StateSpace(\D), \left\{ \varphi\circ\pi_j : \varphi \in \StateSpace(\A_j) \right\} \right) \text{.}
  \end{equation*}
\end{definition}

We have actually some choice regarding what collection of tunnels we use to define the dual propinquity, though some compatibility is needed. We will use the following notion introduced in \cite{Latremoliere14}. This flexibility allows one to obtain a form of our metric where all {\qcms s} involved satisfy, for instance, the strong Leibniz property \cite{Rieffel10c}, or other desirable properties.

\begin{theorem-definition}[{\cite{Latremoliere14},\cite{Latremoliere15}}]
  Let $F$ be a permissible function. The class of all $F$-tunnels is appropriate for the class of all {\Qqcms{F}s}, where a class $\mathcal{T}$ of $F$-tunnels is \emph{appropriate} with a nonempty class $\mathcal{C}$ of {\Qqcms{F}s} when:
  \begin{enumerate}
  \item if $\tau \in \mathcal{T}$ then $\dom{\tau} , \codom{\tau} \in \mathcal{C}$,
  \item if $\mathds{A},\mathds{B} \in \mathcal{Q}$ then there exists a modular tunnel $\tau$ in $\mathcal{T}$ from $\mathds{A}$ to $\mathds{B}$,
  \item if $\mathcal{A}$ and $\mathcal{B}$ are in $\mathcal{C}$ and if there exists a full quantum isometry $\Theta : \mathds{A} \rightarrow \mathds{B}$, then the tunnel $(\mathds{A},\mathrm{id},\Theta)$ is in $\mathcal{T}$ --- where $\mathrm{id}$ is the identity *-automorphism on $\mathds{A}$,
  \item if $\tau = (\D,\Lip,\pi,\rho) \in \mathcal{T}$ then $\tau^{-1} = (\D,\Lip, \rho, \pi) \in \mathcal{T}$,
  \item if $\tau_1,\tau_2  \in \mathcal{T}$ with $\codom{\tau_1} = \dom{\tau_2}$ and if $\varepsilon > 0$, then there exists $\tau \in \mathcal{T}$ from $\dom{\tau_1}$ to $\codom{\tau_2}$ such that $\tunnelextent{\tau} \leq \tunnelextent{\tau_1} + \tunnelextent{\tau_2} + \varepsilon$.
  \end{enumerate}
\end{theorem-definition}

\begin{notation}
  If $\mathcal{T}$ is a class of tunnels appropriate with a nonempty class $\mathcal{C}$ of {\Qqcms{F}s} for some permissible function $F$, then the set of all tunnels from $\mathds{A} \in \mathcal{C}$ to $\mathds{B} \in \mathcal{C}$ in $\mathcal{T}$ is:
  \begin{equation*}
    \tunnelset{\mathds{A}}{\mathds{B}}{\mathcal{T}} \text{.}
  \end{equation*}
\end{notation}

In this paper, we will define and work with various pseudo-metric, defined on classes of objects in certain categories, where distance zero between any two objects for these pseudo-metrics is equivalent to the existence of an isomorphism between these objects. We introduce a convention which is a small abuse of the term ``metric'' but will make our exposition clearer.

\begin{convention}
  Let $\equiv$ be an equivalence relation on a class $\mathcal{C}$. A pseudo-metric $\delta$ on a class $\mathcal{C}$ is called a \emph{metric up to $\equiv$} when $\delta(x,y) = 0$ if and only if $x\equiv y$ for all $x,y \in \mathcal{C}$. In practice, $\equiv$ will often be the equivalence relation defined by some choice of isomorphisms, typically some notion of full quantum isometry.
\end{convention}

\bigskip

The \emph{dual propinquity} is our main tool for our research.

\begin{theorem-definition}[{\cite{Latremoliere13b,Latremoliere14,Latremoliere15}}]\label{prop-def}
  Let $F$ be a continuous permissible function and $\mathcal{C}$ be a nonempty class of {\Qqcms{F}s}. Let $\mathcal{T}$ be an appropriate class of tunnels for $\mathcal{C}$. If, for any two $(\A_1,\Lip_1)$ and $(\A_2,\Lip_2)$ in $\mathcal{C}$, we define:
  \begin{equation*}
    \dpropinquity{\mathcal{T}}((\A_1,\Lip_1),(\A_2,\Lip_2)) = \inf\left\{ \tunnelextent{\tau} : \tau \in \tunnelset{(\A_1,\Lip_1)}{(\A_2,\Lip_2)}{\mathcal{T}} \right\} \text{,}
  \end{equation*}
  then $\dpropinquity{\mathcal{T}}$ is a complete metric on the class of {\Qqcms{F}s} up to full quantum isometry, called the \emph{dual Gromov-Hausdorff propinquity}.
\end{theorem-definition}

We record that if we drop the assumption that the permissible function in Theorem-Definition (\ref{prop-def}), the resulting propinquity is still a metric up to full quantum isometry --- it may not be complete, however.

If $F$ is some permissible function, and if $\mathcal{T}$ is the class of all possible tunnels on the class $\mathcal{F}$ of all {$F$-\qcms s}, the dual propinquity $\dpropinquity{\mathcal{T}}$ on $\mathcal{F}$ is simply denoted as $\dpropinquity{F}$.

\bigskip

The question naturally arises: how to construct tunnels? The technique which has proven helpful in our program consists in introducing the following notion of a bridge. The source of our idea of a bridge between two {\qcms s} $(\A,\Lip_\A)$ and $(\B,\Lip_\B)$ started in Rieffel's groundbreaking paper \cite{Rieffel00}, where a more general notion of bridge meant a seminorm designed in a manner to help with the construction of a Lip-norm on $\A\oplus\B$. Now, our work, unlike Rieffel's, focuses on C*-algebras endowed with L-seminorms (possessing some Leibniz inequality property), and as our understanding of the technical needs for progressing with our work increases in time, we are led to a rather special notion of bridge, as follows. The ``bridge seminorm'' in the definition below is a form of bridge in the original sense of Rieffel.

\begin{definition}[{\cite[Definition 3.6]{Latremoliere13}}]
  Let $\A$ and $\B$ be two unital C*-algebras. A \emph{bridge} $\gamma = (\D,x,\pi_1,\pi_2)$ from $\A$ to $\B$ is given by:
  \begin{enumerate}
  \item a unital C*-algebra $\D$,
  \item $x\in \D$ such that:
    \begin{equation*}
      \StateSpace_1(\D|x) = \left\{ \varphi \in \StateSpace(\D) : \forall a \in \D \quad \varphi(a x) = \varphi(x a) = \varphi(a) \right\}
    \end{equation*}
    is not empty,
  \item two unital *-monomorphisms $\pi_1 : \D \hookrightarrow \A$ and $\pi_2 : \D \hookrightarrow \B$.
  \end{enumerate}
\end{definition}

A bridge does give rise to a tunnel as follows:

\begin{theorem}[{\cite[Theorem 6.3]{Latremoliere13}, \cite[Lemma 5.4]{Latremoliere13b}}]\label{tunnel-from-bridge-0-thm}
  If $\A_1$, $\A_2$ are two unital C*-algebras, if $\gamma = (\D,x,\pi_1,\pi_2)$ is a bridge from $\A_1$ to $\A_2$, and if $\Lip_1$, $\Lip_2$ are L-seminorms on $\A_1$ and $\A_2$ respectively, and if we define:
  \begin{enumerate}
  \item the \emph{height} $\bridgeheight{\gamma}{\Lip_1,\Lip_2}$ of $\gamma$ as:
    \begin{equation*}
      \max_{j\in\{1,2\}} \Haus{\Kantorovich{\Lip_j}}\left( \StateSpace(\A_j), \left\{\varphi\circ\pi_j : \varphi \in \StateSpace_1(\D|x) \right\}  \right)
    \end{equation*}
    \item the \emph{reach} $\bridgereach{\gamma}{\Lip_1,\Lip_2}$ of $\gamma$ as:
      \begin{equation*}
        \max_{\{j,k\} = \{1,2\}} \sup_{\substack{a_j\in\dom{\Lip_j} \\ \Lip_j(a_j) \leq 1 }} \inf_{\substack{a_k\in\dom{\Lip_k} \\ \Lip_k(a_k) \leq 1 }} \bridgenorm{\gamma}{a_1,a_2} \text{,}
      \end{equation*}
      where $\bridgenorm{\gamma}{a,b} = \norm{\pi_1(a) x - x \pi_2(b)}{\D}$ for all $a\in\A_1$, $b\in \A_2$,
      \item the \emph{length} $\bridgelength{\gamma}{\Lip_1,\Lip_2}$ of $\gamma$ as $\max\{\bridgeheight{\gamma}{\Lip_1,\Lip_2},\bridgereach{\gamma}{\Lip_1,\Lip_2} \}$,
  \end{enumerate}
  then for all $\lambda > 0$ with $\lambda\geq\bridgelength{\gamma}{\Lip_1,\Lip_2}$, setting for all $a\in\sa{\A}$ and $b\in \sa{\B}$:
  \begin{equation*}
    \Lip(a,b) = \max\left\{ \Lip_1(a), \Lip_2(b), \frac{1}{\lambda} \bridgenorm{\gamma}{a,b} \right\}
  \end{equation*}
  the quadruple $(\A_1\oplus\A_2,\Lip,\rho_1,\rho_2)$ is a tunnel from $(\A_1,\Lip_1)$ to $(\A_2,\Lip_2)$ of extent at most $\lambda$, with $\rho_j : (a_1,a_2)\in\A_1\oplus\A_2 \mapsto a_j$ is the canonical surjection for $j\in\{1,2\}$.
\end{theorem}

\bigskip

We started to study the problem of convergence of {\gQVB s} in \cite{Latremoliere15c}. As it was not necessarily obvious how to start such an endeavor, we begun our work by extending the notion of a bridge to modular bridges between {\gQVB s}, taking advantage of the Hilbert module structure to provide a sort of replacement for states, as seen below. In order to construct an analogue for modules of the bridge seminorm, we also found it helpful to introduce the notion of anchors and co-anchors. Intuitively, anchors and co-anchors are paired with each other, and thus provide a sort of almost correspondence between the two modules under consideration.

\begin{definition}\label{modular-bridge-def}
  Let $\mathds{A}_1 = (\module{M}_1,\inner{\cdot}{\cdot}{1},\CDN_1,\A_1,\Lip_1)$ and $\mathds{A}_2 = (\module{M}_2,\inner{\cdot}{\cdot}{2},\CDN_2,\A_2,\Lip_2)$ be two {\gQVB s}. 
  
  A \emph{modular bridge} $\gamma = (\D,x,\pi_1,\pi_2,(\omega_j)_{j \in J}, (\eta_j)_{j \in J})$ from $\mathds{A}_1$ to $\mathds{A}_2$ is given as a bridge $\gamma_\flat = (\D,x,\pi_1,\pi_2)$ from $\A_1$ to $\A_2$ with $\norm{x}{\D} = 1$, and for all $j\in J$:
  \begin{equation*}
    \omega_j \in \module{M}_1 \text{ with }\CDN_1(\omega_j) \leq 1 \text{ and } \eta_j \in \module{M}_2, \text{ with }\CDN_2(\eta_j) \leq 1 \text{.}
  \end{equation*}
  The family $(\omega_j)_{j\in J}$ is the family of \emph{anchors} of $\gamma$, while $(\eta_j)_{j\in J}$ is the family of \emph{co-anchors} of $\gamma$.
\end{definition}

As with bridges between unital C*-algebras \cite{Latremoliere13}, there is a mean to associate a number to a modular bridge and use this quantification to define a modular version of the propinquity \cite{Latremoliere16c}. To begin with, the data contained in a modular bridge between two {\gQVB s} includes a bridge between the base {\qcms s}, and the length of this bridge is a first number associated with a modular bridge.

To also numerically capture module-related data from a bridge, it is helpful to introduce a distance on the underlying module of a {\gQVB}, in a manner reminiscent of the {\MongeKant} on the state spaces of {\qcms s}.

\begin{definition}[{\cite[Definition 3.23]{Latremoliere16c}}]
  Let $(\module{M},\inner{\cdot}{\cdot}{\module{M}},\CDN,\A,\Lip)$ be a {\gQVB}. The \emph{modular \MongeKant} $\KantorovichMod{\CDN}$ is defined for all $\omega,\eta\in\module{M}$ by:
  \begin{equation*}
    \KantorovichMod{\CDN}(\omega,\eta) = \sup\left\{ \norm{\inner{\omega - \eta}{\zeta}{\module{M}}}{\A} : \CDN(\zeta) \leq 1 \right\} \text{.}
  \end{equation*}
\end{definition}

We record that, just as the {\MongeKant} induces the weak* topology, the modular {\MongeKant} also induces a natural topology, thanks to our assumptions on D-norms.

\begin{proposition}[{\cite[Proposition 3.24]{Latremoliere16c}}]\label{modular-mongekant-prop}
  If $(\module{M},\inner{\cdot}{\cdot}{\module{M}},\CDN,\A,\Lip)$ is a {\gQVB}, then $\KantorovichMod{\CDN}$ is a metric whose topology on $\module{M}$ is the $\A$-weak topology, i.e. the initial topology for $\{ \norm{\inner{\cdot}{\omega}{\module{M}}}{\A}: \omega \in \module{M} \}$. Moreover, when restricted to the unit ball of $\CDN$, both $\norm{\cdot}{\module{M}}$ and $\KantorovichMod{\CDN}$ induce the same topology.
\end{proposition}

Using the modular {\MongeKant}, we can then define two additional quantities naturally associated with modular bridges. The first quantity, which we call the imprint, measures the density of the sets of anchors and co-anchors in the closed unit balls for the appropriate D-norms. It thus quantifies the error we make by replacing the closed unit balls of D-norms simply by the sets of anchors and co-anchors. Then, as we noted, anchors and co-anchors are paired with each other, so we measure how far each pair of anchor and co-anchor is in the sense of the bridge seminorm for the underlying bridge between the base {\qcms s}, using the inner products to plunge elements of the module back in their respective base spaces. The length of a modular bridge is then a combination of all these quantities in a manner appropriate for the proof of the triangle inequality and desired coincidence property.

\begin{definition}\label{imprint-def}
  Let $\gamma = (\gamma_\flat, (\omega_j)_{j \in J}, (\eta_j)_{j \in J})$ be a modular bridge from $\mathds{A} = (\module{M},\inner{\cdot}{\cdot}{\A},\CDN_{\module{M}},\A,\Lip_\A)$ to $\mathds{B} = (\module{N},\inner{\cdot}{\cdot}{\B},\CDN_{\module{N}},\B,\Lip_\B)$, where $\gamma_\flat = (\D, x, \pi_\A, \pi_\B)$ a bridge from $\A$ to $\B$. We define:
  \begin{enumerate}
    \item the \emph{imprint} $\bridgeimprint{\gamma}$ of $\gamma$ is:
      \begin{multline*}
        \max\big\{ \Haus{\KantorovichMod{\CDN_{\module{M}}}}(\{\omega_j:j\in J\}, \{\omega\in\module{M}:\CDN_{\module{M}}(\omega)\leq 1 \}) \\
        \Haus{\KantorovichMod{\CDN_{\module{N}}}}( \{\eta_j:j\in J\}, \{\eta\in\module{N}:\CDN_{\module{N}}(\eta)\leq 1 \}  ) \big\}\text{;}
      \end{multline*}
    \item the \emph{modular reach} $\bridgemodularreach{\gamma}$ of $\gamma$ is:
      \begin{equation*}
        \max_{j \in J} \decknorm{\gamma}{\omega_j,\eta_j}
      \end{equation*}
      where:
      \begin{equation*}
        \decknorm{\gamma}{\omega,\eta} = \sup\left\{ \bridgenorm{\gamma_\flat}{\inner{\omega}{\omega_j}{\A}, \inner{\eta}{\eta_j}{\B}}, \bridgenorm{\gamma_\flat}{\inner{\omega_j}{\omega}{\A}, \inner{\eta_j}{\eta}{\B} } : j \in J \right\} \text{.}
      \end{equation*}
      \item the \emph{length} $\bridgemodularlength{\gamma}$ of $\gamma$ is:
        \begin{equation*}
          \max\left\{ \bridgelength{\gamma_\flat}{\Lip_\A,\Lip_\B}, \bridgeimprint{\gamma} + \bridgemodularreach{\gamma}  \right\}
        \end{equation*}
  \end{enumerate}
\end{definition}

The \emph{modular propinquity} \cite{Latremoliere16c} is the largest pseudo-metric on {\gQVB s} such that if $\mathds{A}$ and $\mathds{B}$ are {\QVB{F}{H}s}, and if $\gamma$ is a modular bridge form $\mathds{A}$ to $\mathds{B}$, then $\dmodpropinquity{F,H}(\mathds{A},\mathds{B}) \leq \bridgemodularlength{\gamma}$. While the actual definition is involved, its motivation is of course given by the original construction of Edwards and Gromov.

We proved in \cite{Latremoliere16c} that the modular propinquity is a metric up to full modular quantum isometry on a large class of {\gQVB s}, and we showed in particular in \cite{Latremoliere17c,Latremoliere18a} that Heisenberg modules over quantum $2$-tori form continuous families for this  metric. However, just as with the quantum propinquity for {\qcms s}, of which the modular propinquity is the modular analogue, we do not know whether the modular propinquity is complete, thus complicating the study of its geometric and topological properties. Moreover, its definition involves paths of modular bridges, which is sometimes difficult to work with. 

We now present how to define a dual-modular propinquity, which, we will prove, is a complete metric, still up to full modular quantum isometry, and whose construction can be extended to differential structures, such as spectral triples \cite{Latremoliere18g}. The dual modular propinquity dominates the modular propinquity while having a more flexible construction.

\section{The dual modular Propinquity for {\gQVB s}}

The basic ingredient for the construction of the dual-modular propinquity is the notion of a modular tunnel --- the construct analogue to an isometric embedding in the construction of the Gromov-Hausdorff distance.

\begin{definition}\label{modular-tunnel-def}
  Let $(F,H)$ be a permissible pair of functions. Let $\mathds{A}$ and $\mathds{B}$ be two {\QVB{F}{H}s}.  A \emph{$(F,H)$-modular tunnel} $\tau = \left(\mathds{D}, \pi_\A, \pi_\B \right)$ from $\mathds{A}$ to $\mathds{B}$ is given by:
  \begin{enumerate}
  \item a {\QVB{F}{H}} $\mathds{D}$,
  \item two isometries $\pi_\A : \mathds{D} \twoheadrightarrow \mathds{A}$ and $\pi_\B : \mathds{D} \twoheadrightarrow \mathds{B}$.
    \end{enumerate}
  \end{definition}

Modular tunnels allow us to quantify how far apart two {\gQVB s} may be, using a number called the extent, and it is natural to define our prospective metric as the infimum of the extents of all possible tunnels between two given {\gQVB s}. In fact, it is very notable, and comforting, that we simply use the same notion as for a regular tunnel.

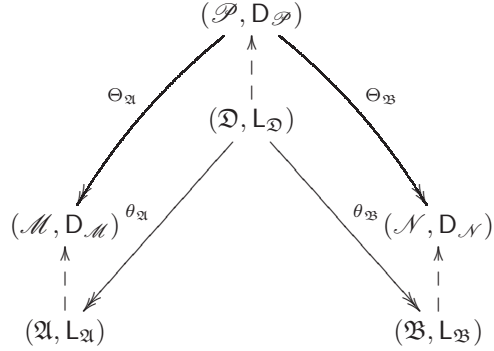
\begin{figure}[t]
\begin{equation*}
  \xymatrix{
    & (\module{P},\CDN_{\module{P}}) \ar@/_/@{>>}[ldd]_{\Theta_\A} \ar@/^/@{>>}[rdd]^{\Theta_\B} & \\
    & (\D,\Lip_\D) \ar@{>>}[ldd]_{\theta_\A} \ar@{>>}[rdd]^{\theta_\B} \ar@{-->}[u] & \\
    (\module{M},\CDN_{\module{M}}) & & (\module{N},\CDN_{\module{N}}) \\
    (\A,\Lip_\A) \ar@{-->}[u] & & (\B,\Lip_\B) \ar@{-->}[u]
    }
\end{equation*}
\caption{A modular tunnel}
\end{figure}

\begin{notation}
  Let $(F,H)$ be a pair of permissible functions, and let $\mathds{A}$, $\mathds{B}$ and $\mathds{D}$ be three {\QVB{F}{H}s}, with respective base quantum spaces $(\A,\Lip_\A)$, $(\B,\Lip_\B)$ and $(\D,\Lip_\B)$. If $\tau = (\mathds{D},(\theta_\A,\Theta_\A), (\theta_\B,\Theta_\B))$ is a modular tunnel with $\mathds{D} = (\module{P},\inner{\cdot}{\cdot}{\D},\CDN_\D,\D,\Lip_\D)$, then $\tau_\flat = (\D,\Lip_\D,\theta_\A,\theta_\B)$ is an $F$-tunnel from $(\A,\Lip_\A)$ to $(\B,\Lip_\B)$ (see Definition (\ref{tunnel-def})).
\end{notation}

\begin{definition}\label{tunnel-extent-def}
  The \emph{extent} of a modular tunnel $\tau$ is the extent of its basic tunnel $\tau_\flat$ (see Definition (\ref{extent-def})).
\end{definition}

The first problem to address is to find a good source of modular tunnels. We now prove that tunnels can be constructed from the type of modular bridges at the basis of \cite{Latremoliere16c}, whose definition we recalled as Definition (\ref{modular-bridge-def}). To this end, we make two remarks. First, our construction below does not, as far as we can tell, satisfy all the needed conditions to obtain a {\gQVB} in the sense of \cite{Latremoliere16c}, though of course it meets the more relaxed Definition (\ref{metrized-bundle-def}) we now use. This is not a small part of our decision to relax our definition.

Second of all, we note that given any modular bridge $\gamma$, there always exists a modular bridge between the same domains and codomains as $\gamma$, which has at most the same length as $\gamma$, yet involve convex, balanced sets of anchors and co-anchors.

We refer to \cite{Latremoliere16c} for the definition of deck norms.

\begin{lemma}\label{convex-bridge-lemma}
  Let $\mathds{A} = (\module{M},\inner{\cdot}{\cdot}{\module{M}},\CDN_{\module{M}},\A,\Lip_\A)$ and $\mathds{B} = (\module{N},\inner{\cdot}{\cdot}{\module{N}}, \CDN_{\module{N}},\B,\Lip_\B)$ be two {\gQVB s}.

  Let $\gamma = (\D,x,\pi_\A,\pi_\B,(\omega_j)_{j\in J}, (\eta_j)_{j\in J})$ be a modular bridge from $\mathds{A}$ to $\mathds{B}$ of length $\lambda$. We define:
  \begin{equation*}
    I(J) = \left\{ (t_j)_{j \in J} \in [-1,1]^{(J)} : \sum_{j \in J} |t_j| \leq 1 \right\}\text{,}
  \end{equation*}
  where $[0,1]^{(J)}$ is the set of $J$-indexed families valued in $[0,1]$ with only finitely many nonzero value. For all $\alpha = (t_j)_{j \in J} \in I(J)$, we set:
  \begin{equation*}
    \omega_\alpha = \sum_{j \in J} t_j \omega_j \text{ and }\eta_\alpha = \sum_{j \in J} t_j \eta_j \text{.}
  \end{equation*}

  Then $\coh{\gamma} = \left\{ \D, x, \pi_\A, \pi_\B, (\omega_\alpha)_{\alpha \in I(J)}, (\eta_\alpha)_{\alpha \in I(J)} \right\}$ is a modular bridge from $\mathds{A}$ to $\mathds{B}$ whose length is at most $\lambda$, and with the same basic bridge and decknorm as $\gamma$. Moreover, $\{\omega_\alpha:\alpha\in I(J)\}$ and $\{\eta_\alpha:\alpha\in I(J)\}$ are two convex, balanced sets.
\end{lemma}

\begin{proof}
  Let $\alpha = (t_j)_{j \in J} \in I(J)$. We then check:
  \begin{equation*}
    \CDN_{\module{M}}(\omega_\alpha) \leq \sum_{j \in J} |t_j| \CDN_{\module{M}}(\omega_j) \leq \sum_{j\in J} |t_j| \leq 1\text{.}
  \end{equation*}
  Similarly, $\CDN_{\module{N}}(\eta_\alpha) \leq 1$. So $\coh{\gamma}$ is indeed a modular bridge. 

  Since $\{\omega_j : j \in J\} \subseteq \left\{\omega_\alpha : \alpha \in I(J)\right\}$ and  $\{\eta_j : j \in J\} \subseteq \left\{\eta_\alpha : \alpha \in I(J)\right\}$, we conclude:
  \begin{equation*}
    \bridgeimprint{\coh{\gamma}} \leq \bridgeimprint{\gamma} \text{.}
  \end{equation*}

Moreover, for the same reason, we have $\decknorm{\gamma}{\cdot,\cdot} \leq \decknorm{\coh{\gamma}}{\cdot,\cdot}$. On the other hand, if $\omega \in \module{M}$ and $\eta\in\module{N}$:
\begin{align*}
  \norm{\inner{\omega}{\omega_\alpha}{\A} x - x \inner{\eta}{\eta_\alpha}{\B}}{\D} 
  &= \norm{\sum_{j \in J} t_j \left( \inner{\omega}{\omega_j}{\A} x - x \inner{\eta}{\eta_j}{\B} \right)}{\D} \\
  &\leq \sum_{j \in J} |t_j| \norm{\inner{\omega}{\omega_j}{\A} x - x  \inner{\eta}{\eta_j}{\B}}{\D} \\
  &\leq  \sum_{j \in J} |t_j| \decknorm{\gamma}{\omega,\eta} = \decknorm{\gamma}{\omega,\eta} \text{.}
\end{align*}

Hence $\decknorm{\coh{\gamma}}{\omega,\eta} \leq \decknorm{\gamma}{\omega,\eta}$. Therefore $\decknorm{\coh{\gamma}}{\cdot,\cdot} = \decknorm{\gamma}{\cdot,\cdot}$.

Moreover, a similar computation shows that $\bridgemodularreach{\coh{\gamma}} = \bridgemodularreach{\gamma}$: if $\beta = (t_j)_{j\in J} \in I(J)$ then:
\begin{equation*}
  \decknorm{\coh{\gamma}}{\omega_\beta,\eta_\beta} \leq \sum_{j \in J} |t_j| \decknorm{\gamma}{\omega_j,\eta_j} \leq \bridgemodularreach{\gamma}
\end{equation*}
so $\bridgemodularreach{\coh{\gamma}} \leq \bridgemodularreach{\gamma}$; on the other hand:
\begin{multline*}
  \bridgemodularreach{\coh{\gamma}} = \sup_{\beta\in I(J)} \decknorm{\coh{\gamma}}{\omega_\beta,\eta_\beta} \\ = \sup_{\beta\in I(J)} \decknorm{\gamma}{\omega_\beta,\eta_\beta} \geq \sup_{j\in J} \decknorm{\gamma}{\omega_j,\eta_j} = \bridgemodularreach{\gamma} \text{.}
\end{multline*}

By construction, the basic bridge of $\coh{\gamma}$ is the basic bridge $\gamma_\flat$ of $\gamma$. So $\bridgelength{\coh{\gamma}_\flat}{\Lip_\A,\Lip_\B} = \bridgelength{\gamma_\flat}{\Lip_\A,\Lip_\B}$.

Therefore, $\bridgemodularlength{\coh{\gamma}} \leq \bridgemodularlength{\gamma}$ as desired.

It is immediate to check that $\{\omega_\alpha:\alpha\in I(J)\}$ and $\{\eta_\alpha:\alpha\in I(J)\}$ are two convex, balanced sets.
\end{proof}

\begin{remark}\label{convex-bridge-rmk}
  The bridges constructed in Lemma (\ref{convex-bridge-lemma}) have the additional property that not only their sets of anchors and co-anchors are closed, convex and balanced, but also that there is a natural pairing between convex combinations of anchors and co-anchors. Using the notations of Lemma (\ref{convex-bridge-lemma}), if $\alpha_1 = (t_j)_{j\in J},\alpha_2 =(s_j)_{j\in J} \in I(J)$, and if $r \in [0,1]$, then of course $r \alpha_1 + (1-r) \alpha_2 = (r t_j + (1-r) s_j)_{j\in J}$ is also an element of $I(J)$. The notable fact is that $r \omega_{\alpha_1} + (1-r) \omega_{\alpha_2} = \omega_{r\alpha_1 + (1-r)\alpha_2}$ and at the same time, $r \eta_{\alpha_1} + (1-r) \eta_{\alpha_2} = \eta_{r\alpha_1 + (1-r)\alpha_2}$. Thus, if we take two anchors and their associated co-anchors, then their convex combinations are also associated anchors and co-anchors.
\end{remark}

We now prove that modular tunnels can be built from modular bridges. The assumptions of Theorem (\ref{tunnel-from-bridge-thm}) below are designed to fit the conclusion of Lemma (\ref{convex-bridge-lemma}) and Remark (\ref{convex-bridge-rmk}).

\begin{notation}
  The closure of a  subset $A$ of a  topological space is denoted by $\mathrm{cl}(A)$.
\end{notation}

\begin{theorem}\label{tunnel-from-bridge-thm}
  Let $\mathds{A} = (\module{M},\inner{\cdot}{\cdot}{\module{M}},\CDN_{\module{M}},\A,\Lip_\A)$ and $\mathds{B} = (\module{N},\inner{\cdot}{\cdot}{\module{N}},\CDN_{\module{N}},\B,\Lip_\B)$ be two {\QVB{F}{H}s}. Let:
  \begin{equation*}
    \gamma = (\D,x,\pi_\A,\pi_\B,(\omega_j)_{j\in J}, (\eta_j)_{j\in J})
  \end{equation*} 
  be a modular bridge from $\mathds{A}$ to $\mathds{B}$ of length $\lambda$ and such that $\{\omega_j:j\in J\}$ and $\{\eta_j:j\in J \}$ are convex balanced sets, with the additional property that:
  \begin{equation*}
    \forall j,k \in J \quad \forall t \in [0,1] \quad \exists q \in J \quad t\omega_j + (1-t)\omega_k = \omega_q \text{ and }t\eta_j + (1-t)\eta_k = \eta_q \text{.}
  \end{equation*}

  Let $\varepsilon \geq 0$ such that $\lambda + \varepsilon > 0$.

  Let $\module{P} = \module{M}\oplus\module{N}$, seen as an $\alg{E} = \A\oplus\B$ Hilbert module using the left action:
  \begin{equation*}
    \forall a\in\A, b \in \B \; \forall \omega\in\module{M}, \eta \in \module{N} \quad (a,b)\cdot(\omega,\eta) = (a\omega,b\eta)
  \end{equation*}
  and with inner product:
  \begin{equation*}
    \forall (\omega_1,\eta_1),(\omega_2,\eta_2) \in \module{P} \quad \inner{(\omega_1,\eta_1)}{(\omega_2,\eta_2)}{\module{P}} = \left( \inner{\omega_1}{\omega_2}{\module{M}}, \inner{\eta_1}{\eta_2}{\module{N}} \right) \text{.}
  \end{equation*}

  For all $a\in\dom{\Lip_\A}$, $b\in\dom{\Lip_\B}$, we set:
  \begin{equation*}
    \Lip(a,b) = \max\left\{ \Lip_\A(a), \Lip_\B(b), \frac{1}{\lambda + \varepsilon} \bridgenorm{\gamma}{a,b} \right\} \text{.}
  \end{equation*}

  Let:
  \begin{equation*}
    \mathcal{D} = \mathrm{cl}\left(\bigcup_{j \in J}\left\{ (\omega,\eta) \in \module{P} : \KantorovichMod{\CDN_{\module{M}}}(\omega,\omega_j) \leq \bridgeimprint{\gamma}, \KantorovichMod{\CDN_{\module{N}}}(\eta,\eta_j) \leq \bridgeimprint{\gamma}  \right\}\right)
  \end{equation*}
  where the closure is for the norm $\norm{\cdot}{\module{P}}$, and let $\mathsf{p}$ be the Minkowsky gauge functional of the closed, balanced convex $\mathcal{D}$.

  We then set for all $\omega\in\module{M}$, $\eta\in\module{N}$:
  \begin{equation*}
    \CDN(\omega,\eta) = \max\left\{ \CDN_{\module{M}}(\omega), \CDN_{\module{N}}(\eta),
      \frac{1}{\lambda + \varepsilon} \decknorm{\gamma}{\omega,\eta}, \mathsf{p}(\omega,\eta)  \right\}\text{.}
  \end{equation*}
  Then $\mathds{D} = \left( \module{P}, \inner{\cdot}{\cdot}{\module{P}}, \CDN, \alg{E}, \Lip \right)$ is a {\gQVB}.
  
  Let $\theta_\A : (a,b) \in \alg{E} \mapsto a \in \A$ and $\theta_\B : (a,b) \in \alg{E} \mapsto b \in \B$ and $\Theta_\A : (\omega,\eta) \in \module{P} \mapsto \omega \in \module{M}$ and $\Theta_\B : (\omega,\eta) \in \module{P} \mapsto \eta \in \module{N}$. Then:
  \begin{equation*}
    \left( \mathds{D}, (\theta_\A,\Theta_\A), (\theta_\B,\Theta_\B) \right)
  \end{equation*}
  is an $(F,H)$-modular tunnel of length $\lambda + \varepsilon$.
\end{theorem}

\begin{proof}
  We begin with proving that $\mathds{D}$ is indeed a {\QVB{F}{H}}. By Theorem (\ref{tunnel-from-bridge-0-thm}), the pair $(\alg{E},\Lip)$ is an {\Qqcms{F}}.

  It is immediate that $(\module{P},\inner{\cdot}{\cdot}{\module{P}})$ is a Hilbert $\D$-module, with $\norm{(\omega,\eta)}{\module{P}} = \max\left\{ \norm{\omega}{\module{M}}, \norm{\eta}{\module{N}}\right\}$ for all $(\omega,\eta)\in\module{P}$. It is thus enough to check that $\CDN$ is indeed a D-norm.

  To begin with, $\mathcal{D}$ is closed and convex. If
  \begin{equation*}
    (\omega,\eta),(\omega',\eta') \in \bigcup_{j \in J}\left\{ (\omega,\eta) \in \module{P} : \KantorovichMod{\CDN_{\module{M}}}(\omega,\omega_j) \leq \bridgeimprint{\gamma}, \KantorovichMod{\CDN_{\module{N}}}(\eta,\eta_j) \leq \bridgeimprint{\gamma}  \right\}
  \end{equation*}
  then there exists $j,k \in J$ such that:
  \begin{equation*}
    \KantorovichMod{\CDN_{\module{M}}}(\omega,\omega_j) \leq \bridgeimprint{\gamma} \text{ and } \KantorovichMod{\CDN_{\module{M}}}(\eta,\eta_j) \leq \bridgeimprint{\gamma}
  \end{equation*}
  and
  \begin{equation*}
    \KantorovichMod{\CDN_{\module{N}}}(\omega',\omega_k) \leq \bridgeimprint{\gamma} \text{ and } \KantorovichMod{\CDN_{\module{M}}}(\eta',\eta_k) \leq \bridgeimprint{\gamma}
  \end{equation*}
  Let now $t\in [0,1]$. By assumption, there exists $q \in J$ such that $\omega_q = t \omega_j + (1-t)\omega_k$ and $\eta_q = t \eta_j + (1-t)\eta_k$. Thus:
  \begin{equation*}
    \KantorovichMod{\CDN_{\module{M}}}(t\omega+(1-t)\omega',\omega_q) \leq t\KantorovichMod{\CDN_{\module{M}}}(\omega,\omega_j) + (1-t)t\KantorovichMod{\CDN_{\module{M}}}(\omega',\omega_k) \leq \bridgeimprint{\gamma}
  \end{equation*}
  and similarly, $\KantorovichMod{\CDN_{\module{N}}}(t\eta+(1-t)\eta',\eta_q) \leq \bridgeimprint{\gamma}$. Thus by construction:
  \begin{equation*}
    t (\omega,\eta) + (1-t)(\omega',\eta') \in  \bigcup_{j \in J}\left\{ (\omega,\eta) \in \module{P} : \KantorovichMod{\CDN_{\module{M}}}(\omega,\omega_j) \leq \bridgeimprint{\gamma}, \KantorovichMod{\CDN_{\module{N}}}(\eta,\eta_j) \leq \bridgeimprint{\gamma}  \right\}\text{.}
  \end{equation*}
  As the closure of a convex set, we indeed proved that $\mathcal{D}$ is a closed convex subset of $\mathcal{P}$. 

  The closed convex set $\mathcal{D}$ is also balanced since $0\in\mathcal{D}$. Consequently, $\mathsf{p}$ is a lower semi-continuous seminorm defined on the span of $\mathcal{D}$, which contains $\dom{\CDN_{\module{M}}}\times\dom{\CDN_{\module{N}}}$ by Definition (\ref{imprint-def}) of $\bridgeimprint{\gamma}$ (simply noting that, for instance, picking any $i\in J$,  we have $\max\{\KantorovichMod{\CDN_{\module{M}}}(\omega,\omega_i),\KantorovichMod{\CDN_{\module{N}}}(\eta,\eta_i)\} < \infty$ and  then scaling). It then follows that $\CDN$ is a norm defined (in particular, taking finite values) on $\dom{\CDN_{\module{M}}}\times\dom{\CDN_{\module{N}}}$ which is dense in $\module{P}$ (note that $\CDN_{\module{M}}$ and $\CDN_{\module{N}}$ are indeed norms on their respective domains).

As the supremum of lower semi-continuous functions and continuous functions, $\CDN$ is lower semi-continuous.

Now, by construction:
\begin{equation*}
  \left\{ (\omega,\eta) \in \module{P} : \CDN(\omega,\eta)\leq 1 \right\} \subseteq \{\omega\in\module{M}:\CDN_{\module{M}}(\omega)\leq 1\} \times \{\eta\in\module{N}:\CDN_{\module{N}}(\eta)\leq 1\}
\end{equation*}
and the set on the right hand side is the product of two norm-compact sets, so it is itself norm-compact. Hence as a closed subset of a compact set, the unit ball of $\CDN$ is compact as well.

Furthermore, for all $(\omega,\eta)\in\module{P}$, we have:
\begin{equation*}
  \norm{(\omega,\eta)}{\module{P}} = \max\left\{ \norm{\omega}{\module{M}}, \norm{\eta}{\module{N}} \right\} \leq \max\{ \CDN_{\module{M}}(\omega), \CDN_{\module{N}}(\eta) \} \leq \CDN(\omega,\eta) \text{.}
\end{equation*}

We now check that $\CDN$ has satisfies an appropriate form of the inner quasi-Leibniz inequality.

Let $(\omega,\eta),(\chi,\zeta) \in \module{P}$ such that $\CDN(\omega,\eta)\leq 1$ and $\CDN(\chi,\zeta)\leq 1$. Now, if $\CDN(\chi,\zeta) \leq 1$ then $\mathsf{p}(\chi,\zeta)\leq 1$. Let $\varepsilon > 0$. Since $\KantorovichMod{\CDN_{\module{M}}}$ and $\KantorovichMod{\CDN_{\module{N}}}$ are dominated, respectively, by $\norm{\cdot}{\module{M}}$ and $\norm{\cdot}{\module{N}}$, and since $\CDN_{\module{M}}(\omega)\leq 1$, $\CDN_{\module{M}}(\eta)\leq 1$, there exists $j\in J$ such that:
\begin{equation*}
  \norm{\inner{\omega}{\chi-\omega_j}{\module{M}}}{\A} \leq \KantorovichMod{\CDN_{\module{M}}}(\chi,\omega_j) \leq \varepsilon + \bridgeimprint{\gamma}
\end{equation*}
and $\norm{\inner{\eta}{\zeta-\eta_j}{\module{N}}}{\B} \leq \KantorovichMod{\CDN_{\module{N}}}(\zeta,\eta_j)\leq \bridgeimprint{\gamma} + \varepsilon$. We then compute (using $\norm{x}{\D} = 1$):
\begin{align*}
  \bridgenorm{\gamma}{\inner{(\omega,\eta)}{(\chi,\zeta)}{\module{P}}} 
  &= \norm{\pi_\A\left(\inner{\omega}{\chi}{\A}\right) x - x \pi_\B\left(\inner{\eta}{\zeta}{\B}\right)}{\D} \\
  &= \norm{\pi_\A\left(\inner{\omega}{\chi - \omega_j}{\module{M}}\right) x - x\pi_\B\left(\inner{\eta}{\zeta-\eta_j}{\module{N}}\right)}{\D} \\
  &\quad + \norm{\pi_\A\left(\inner{\omega}{\omega_j}{\module{M}}\right) x - x \pi_\B\left(\inner{\eta}{\eta_j}{\module{N}}\right)}{\D} \\
  &\leq \norm{\inner{\omega}{\chi-\omega_j}{\module{M}}}{\A} + \norm{\inner{\eta}{\zeta-\eta_j}{\module{N}}}{\module{N}} + \decknorm{\gamma}{\omega,\eta} \\
  &\leq 2 \bridgeimprint{\gamma} + 2\varepsilon + \bridgemodularlength{\gamma} \text{ using \cite[Proposition 4.17]{Latremoliere16c}}\\
  &\leq 2 \bridgemodularlength{\gamma} + 2\varepsilon \text{.}
\end{align*}
Since $\varepsilon > 0$ is arbitrary, we conclude:
\begin{equation*}
    \bridgenorm{\gamma}{\inner{(\omega,\eta)}{(\chi,\zeta)}{\module{P}}} \leq 2\bridgemodularlength{\gamma} \text{.}
\end{equation*}

Therefore, if $(\omega,\eta), (\chi,\zeta) \in \dom{\CDN}\setminus\{(0,0)\}$, then by homogeneity:
\begin{align*}
  \bridgenorm{\gamma}{\inner{(\omega,\eta)}{(\chi,\zeta)}{\module{P}}} &\leq 2 \bridgemodularlength{\gamma}  \CDN(\omega,\eta)\CDN(\chi,\zeta) \\
  &\leq 2 \CDN(\omega,\eta) \CDN(\chi,\zeta) \bridgemodularlength{\gamma} \\
  &\leq  H(\CDN(\omega,\eta),\CDN(\chi,\zeta)) (\lambda + \varepsilon) \text{.}
\end{align*}

We record $\bridgenorm{\gamma}{\Re\inner{\omega}{\chi}{\A},\Re\inner{\eta}{\zeta}{\B}} \leq \bridgenorm{\gamma}{\inner{\omega}{\chi}{\A},\inner{\eta}{\zeta}{\B}}$, since $\Re$ is a linear contraction on any C*-algebra.

From this, and since $\mathds{A}$ and $\mathds{B}$ both satisfy the $H$-inner quasi-Leibniz inequality, we conclude that for all $(\omega,\eta), (\chi,\zeta) \in \dom{\CDN}$
\begin{align*}
  \Lip\left(\Re\inner{(\omega,\eta)}{(\chi,\zeta)}{\module{P}}\right)
  &= \max\left\{ 
    \begin{array}{l}
      \Lip_\A(\Re\inner{\omega}{\chi}{\A}) \\
      \Lip_\B(\Re\inner{\eta}{\zeta}{\B}) \\
      \frac{1}{\lambda+\varepsilon} \bridgenorm{\gamma}{\inner{\omega}{\chi}{\A},\inner{\eta}{\zeta}{\B}}
    \end{array}
    \right\} \\
  &\leq H(\CDN(\omega,\eta),\CDN(\chi,\zeta)) \text{.}
\end{align*}
A similar computation shows that $\Lip\left(\Im\inner{(\omega,\eta)}{(\chi,\zeta)}{\module{P}}\right)\leq H(\CDN(\omega,\eta),\CDN(\chi,\zeta))$.  Thus the inner quasi-Leibniz inequality holds for $\CDN$ and $\Lip$. Note that it is important that we used the modular length of $\gamma$ (rather than the length of its basic bridge) in the definition of $\Lip$ at this exact point.

Hence $\CDN$ is indeed a D-norm. Therefore $(\module{P},\inner{\cdot}{\cdot}{\module{P}},\CDN,\alg{E},\Lip)$ is indeed a {\gQVB}.

To now prove that $\tau$ is a modular tunnel, we first note that the basic tunnel $(\D,\Lip_\D,\theta_\A,\theta_\B)$ of $\tau$ is a tunnel of extent no more than $\lambda + \varepsilon$ by \cite{Latremoliere14}. It is therefore sufficient to prove that $\Theta_\A$ and $\Theta_\B$ are modular quantum isometries --- the proof is identical for both maps, so we work with $\Theta_\A$.

Let $\omega \in \module{M}$ with $\CDN_{\module{M}}(\omega) \leq 1$. There exists $j \in J$ such that $\KantorovichMod{\CDN_{\module{M}}}(\omega,\omega_j) \leq \bridgeimprint{\gamma}$. Since $\KantorovichMod{\CDN_{\module{N}}}(\eta_j,\eta_j) = 0$, we conclude that $\mathsf{p}(\omega,\eta_j) \leq 1$.  Moreover:
\begin{align*}
  \decknorm{\gamma}{\omega,\eta_j} 
  &= \sup_{k \in J} \norm{\pi_\A(\inner{\omega}{\omega_k}{\module{M}}) x - x \pi_\B(\inner{\eta_j}{\eta_k}{\module{N}}) }{\D} \\
  &\leq \sup_{k\in J} \left( \norm{\inner{\omega-\omega_j}{\omega_k}{\module{M}}}{\A} + \norm{\pi_\A(\inner{\omega_j}{\omega_k}{\module{M}}) x - x \pi_\B(\inner{\eta_j}{\eta_k}{\module{N}}) }{\D}  \right) \\
  &\leq \KantorovichMod{\CDN_{\module{M}}}(\omega,\omega_j) + \bridgemodularreach{\gamma} \\
  &\leq \bridgeimprint{\gamma} + \bridgemodularreach{\gamma} \\
  &\leq \bridgemodularlength{\gamma} \leq \lambda + \varepsilon \text{.}
\end{align*}
Since $\CDN_{\module{N}}(\eta_j) \leq 1$, we conclude that $\CDN(\omega,\eta_j) \leq 1$. Thus $\Theta_\A$ is a modular isometry (by homogeneity).

Thus $\tau$ is indeed a modular tunnel of extent no more than $\lambda + \varepsilon$. This concludes our proof.
\end{proof}

A corollary of Theorem (\ref{tunnel-from-bridge-thm}), thanks to the construction of Lemma (\ref{convex-bridge-lemma}), is that all modular bridges give rise to modular tunnels.

\begin{corollary}\label{tunnel-from-bridge-cor}
  If $\mathds{A}$ and $\mathds{B}$ are two {\QVB{F}{H}s}, and if $\gamma$ is a modular bridge of length at most $\lambda$, then for all $\varepsilon \geq 0$ such that $\lambda + \varepsilon > 0$, there exists an $(F,H)$-modular tunnel $\tau$ from $\mathds{A}$ to $\mathds{B}$ of extent at most $\lambda + \varepsilon$.
\end{corollary}

\begin{proof}
  We apply Theorem (\ref{tunnel-from-bridge-thm}) to the modular bridge $\coh{\gamma}$ constructed from $\gamma$ by Lemma (\ref{convex-bridge-lemma}), noting that Remark (\ref{convex-bridge-rmk}) applies.
\end{proof}

A first, important corollary of Theorem (\ref{tunnel-from-bridge-thm}), is that modular tunnels actually exist.

\begin{notation}
  If $(E,d)$ is a metric space, then we write $\diam{E}{d}$ for its diameter. If $(\A,\Lip)$ is a {\qcms} then the diameter $\diam{\StateSpace(\A)}{\Kantorovich{\Lip}}$ of $\StateSpace(\A)$ for the {\MongeKant} $\Kantorovich{\Lip}$ is simply denoted by $\diam{\A}{\Lip}$.
\end{notation}

\begin{corollary}\label{diam-bound-cor}
  If $\mathds{A}$ and $\mathds{B}$ are two {\QVB{F}{H}s} then there exists a $(F,H)$-modular tunnel $\tau$ from $\mathds{A}$ to $\mathds{B}$ of extent at most:
  \begin{equation*}
    \max\left\{ 2, \diam{\A}{\Lip_\A}, \diam{\B}{\Lip_\B} \right\}\text{.}
  \end{equation*}
\end{corollary}

\begin{proof}
  Write $\mathds{A} = (\module{M},\inner{\cdot}{\cdot}{\module{M}},\CDN_{\module{M}}, \A,\Lip_\A)$ and $\mathds{B} = (\module{N},\inner{\cdot}{\cdot}{\module{N}},\CDN_{\module{N}}, \B ,\Lip_\B)$.

  By \cite[Proposition 4.6]{Latremoliere13}, there exists a bridge $(\D,x,\pi_\A,\pi_\B)$ with $\norm{x}{\D} = 1$, of length at most $\max\left\{\diam{\A}{\Lip_\A},\diam{\B}{\Lip_\B}\right\}$. Pick any $\omega\in\module{M}$ with $\CDN_{\module{M}}(\omega)\leq 1$ and $\eta\in\module{N}$ with $\CDN_{\module{N}}(\eta)\leq 1$. It is then immediate that if:
  \begin{equation*}
    \delta = (\D,x,\pi_\A,\pi_\B, (\omega), (\eta))
  \end{equation*}
  then $\delta$ is a modular bridge with length at most:
  \begin{equation*}
    \max\left\{ 2, \diam{\A}{\Lip_\A}, \diam{\B}{\Lip_\B} \right\} \text{.}
  \end{equation*}

  By Theorem (\ref{tunnel-from-bridge-thm}), we thus conclude that there exists a $(F,H)$-modular tunnel of extent at most $\max\left\{ 2, \diam{\A}{\Lip_\A}, \diam{\B}{\Lip_\B} \right\}$ constructed from $\gamma$.
\end{proof}

We now describe another mean to construct modular tunnels, by almost composition of other modular tunnels. This construction will in fact ensures that our modular propinquity will satisfy the triangle inequality. The proof extends \cite{Latremoliere14}.

\begin{theorem}\label{triangle-thm}
  Let $\mathds{A}$, $\mathds{B}$ and $\mathds{E}$ be three {\QVB{F}{H}s}. Write $\mathds{B} = (\module{M},\inner{\cdot}{\cdot}{\module{M}},\CDN_{\module{M}},\B,\Lip_\B)$.

  Let $\tau_1 = (\mathds{D}_1,(\theta_\A,\Theta_\A),(\theta_\B,\Theta_\B))$ be a modular tunnel from $\mathds{A}$ to $\mathds{B}$ with $\mathds{D}_1 = (\module{P},\inner{\cdot}{\cdot}{\module{P}},\CDN_1,\D_1,\Lip^1)$. Let $\tau_2 = (\mathds{D}_2, (\pi_\B,\Pi_\B), (\pi_{\alg{E}},\Pi_{\alg{E}}))$ be a modular tunnel from $\mathds{B}$ to $\mathds{E}$, with $\mathds{D}_2 = (\module{R},\inner{\cdot}{\cdot}{\module{R}},\CDN_2,\D_2,\Lip^2)$.

  Let $\D = \D_1 \oplus \D_2$ and, for all $(d_1,d_2) \in \sa{\D}$, set:
  \begin{equation*}
    \Lip(d_1,d_2) = \max\left\{ \Lip^1(d_1), \Lip^2(d_2), \frac{1}{\varepsilon}\norm{\theta_\B(d_1) - \pi_\B(d_2)}{\B} \right\} \text{.}
  \end{equation*}

  Let $\module{B} = \module{P} \oplus \module{R}$, see as a Hilbert $\D$-module with the action:
  \begin{equation*}
    \forall d_1\in\D_1,d_2\in \D_2,\omega_1\in\module{P},\omega_2\in\module{R} \quad (d_1,d_2)\cdot(\omega_1,\omega_2) = (d_1 \omega_1, d_2 \omega_2)
  \end{equation*}
  and inner product:
  \begin{equation*}
    \forall \omega_1,\eta_1\in\module{P},\omega_2,\eta_2\in\module{R} \quad \inner{(\omega_1,\omega_2)}{(\eta_1,\eta_2)}{\module{B}} = \left(\inner{\omega_1}{\eta_1}{\module{P}}, \inner{\omega_2}{\eta_2}{\module{R}} \right)\text{.}
  \end{equation*}

  We define, for all $(\omega,\eta)\in\module{B}$:
  \begin{equation*}
    \CDN(\omega,\eta) = \max\left\{ \CDN_1(\omega), \CDN_2(\eta), \frac{1}{\varepsilon}\norm{\Theta_\B(\omega) - \Pi_\B(\eta)}{\module{M}}  \right\}\text{.}
  \end{equation*}

  Let
  \begin{equation*}
    \chi_\A : (d_1,d_2) \in \D \mapsto \theta_\A(d_1) \text{ and }\Xi_{\A} : (\omega,\eta)\in\module{B} \mapsto \Theta_\A(\omega) \text{,}
  \end{equation*}
  and
  \begin{equation*}
    \upsilon_{\alg{E}} : (d_1,d_2) \in \D \mapsto \pi_{\alg{E}}(d_2)\text{ and  }\Upsilon_{\alg{E}} : (\omega,\eta)\in\module{B} \mapsto \Pi_{\alg{E}}(\eta) \text{.}
  \end{equation*}

  If $\mathds{D} = \left(\module{B},\inner{\cdot}{\cdot}{\module{B}},\CDN,\D,\Lip_\D\right)$ and if $\tau = (\mathds{D},(\chi_\A, \Xi_\A),(\upsilon_{\alg{E}}, \Upsilon_{\alg{E}}))$, then $\mathds{D}$ is a {\QVB{F}{H}}, and $\tau$ is a modular tunnel from $\mathds{A}$ to $\mathds{E}$ of extent at most $\tunnelextent{\tau_1} + \tunnelextent{\tau_2} + \varepsilon$.
\end{theorem}

\begin{proof}
  First, we note that for all $(\omega,\eta) \in \module{B}$:
  \begin{equation*}
    \norm{(\omega,\eta)}{\module{B}} = \max\{\norm{\omega}{\module{P}},\norm{\eta}{\module{R}}\} \leq \max\{\CDN_1(\omega),\CDN_2(\eta)\} \leq \CDN(\omega,\eta) \text{.}
  \end{equation*}

  Moreover, $\dom{\CDN}$ is obviously dense in $\module{B}$.

  We now check the inner quasi-Leibniz condition. Let now $(\omega_1,\omega_2), (\eta_1,\eta_2) \in \module{B}$. To begin with:
  \begin{align*}
    \norm{\theta_\B(\inner{\omega_1}{\eta_1}{\module{P}}) - \pi_\B(\inner{\omega_2}{\eta_2}{\module{R}})}{\B}
    &= \norm{\inner{\Theta_\B(\omega_1)}{\Theta_\B(\eta_1)}{\module{M}} - \inner{\Pi_\B(\omega_2)}{\Pi_\B(\eta_2)}{\module{M}}}{\B}\\
    &\leq \norm{\inner{\Theta_\B(\omega_1)}{\Theta_\B(\eta_1)}{\module{M}} - \inner{\Theta_\B(\omega_1)}{\Pi_\B(\eta_2)}{\module{M}}}{\B}\\
    &\quad + \norm{\inner{\Theta_\B(\omega_1)}{\Pi_\B(\eta_2)}{\module{M}} - \inner{\Pi_\B(\omega_2)}{\Pi_\B(\eta_2)}{\module{M}}}{\B} \\
    &= \norm{\inner{\Theta_\B(\omega_1)}{\Theta_\B(\eta_1)-\Pi_\B(\eta_2)}{\module{M}}}{\B} \\
    &\quad + \norm{\inner{\Theta_\B(\omega_1) - \Pi_\B(\eta_1)}{\Pi_\B(\eta_2)}{\module{M}}}{\B} \\
    &\leq \norm{\omega_1}{\module{R}}\norm{\Theta_\B(\eta_1)-\Pi_\B(\eta_2)}{\module{M}}  \\
    &\quad + \norm{\Theta_\B(\omega_1) - \Pi_\B(\omega_2)}{\module{M}} \norm{\eta_2}{\module{R}}  \\ 
    &\leq  \varepsilon \left( \norm{\omega_1}{\module{R}}\CDN((\eta_1,\eta_2)) + \CDN((\omega_1,\omega_2))\norm{\eta_2}{\module{P}} \right) \\
    &\leq 2 \varepsilon \CDN(\omega_1,\omega_2) \CDN(\eta_1,\eta_2) \text{.}
  \end{align*}

  Therefore, by definition (and using the linearity and contractive property of $\Re$ and $\Im$ on any C*-algebra):
  \begin{align*}
    \Lip\left(\Re\inner{(\omega_1,\omega_2)}{(\eta_1,\eta_2)}{\module{B}}\right) 
    &\leq \max\left\{
      \begin{array}{l}
        \Lip_1(\Re\inner{\omega_1}{\eta_1}{\D_{1}}) \\
        \Lip_2(\Re\inner{\omega_2}{\eta_2}{\D_{2}}) \\
        2 \CDN((\eta_1,\eta_2)) \CDN((\omega_1,\omega_2))
      \end{array}
    \right\} \\
    &\leq \max\{ H(\CDN_{1}(\omega_1),\CDN_1(\eta_1)), H(\CDN_2(\omega_2),\CDN_2(\eta_2)),\\
    &\quad \quad H(\CDN(\eta_1,\eta_2), \CDN(\omega_1,\omega_2))\}\\
    &\leq H(\CDN(\omega_1,\omega_2),\CDN(\eta_1,\eta_2)) \text{,}
  \end{align*}
and similarly, $\Lip\left(\Im\inner{(\omega_1,\omega_2)}{(\eta_1,\eta_2)}{\module{B}}\right)\leq H(\CDN(\omega_1,\omega_2),\CDN(\eta_1,\eta_2))$.

By construction, $\CDN$ is lower semi-continuous since $\CDN_1$ and $\CDN_2$ are. Moreover, the unit ball of $\CDN$ is a (closed, by lower semi-continuity) subset of the product $\{\omega\in\module{P}:\CDN_1(\omega)\leq 1\}\times\{\omega\in\module{R}:\CDN_2(\omega)\leq 1\}$ of compact sets, and so it is compact.

Thus $\mathds{D}$ is indeed a {\QVB{F}{H}}.

Last, we prove that $(\mathds{D},(\theta_\A,\Theta_\A),(\pi_{\alg{E}},\Pi_{\alg{E}}))$ is a modular tunnel.

Write $\mathds{A} = \left(\module{A},\inner{\cdot}{\cdot}{\module{A}},\CDN_{\module{A}},\A,\Lip_\A\right)$. Let $\omega\in\dom{\CDN_{\module{A}}}$. Since $\tau_1$ is a modular tunnel, there exists $\zeta \in \module{P}$ such that $\Theta_\A(\zeta) = \omega$ and $\CDN_{\module{P}}(\zeta) = \CDN_{\module{A}}(\omega)$. As $\Theta_\B$ is also a quantum isometry, we conclude that $\CDN_{\module{B}}(\Theta_\B(\zeta))$. Since $\tau_2$ is also a tunnel, there exists $\eta\in\mathscr{R}$ such that $\Pi_\B(\eta) = \Theta_\B(\zeta)$ and $\CDN_{\module{B}}(\eta) = \CDN_{\module{M}}(\Theta_\B(\omega))$. Thus $\CDN(\zeta,\eta) = \CDN_{\module{A}}(\omega)$. Thus by construction:
  \begin{equation*}
    \forall \omega \in \dom{\CDN_{\module{A}}} \quad \CDN_{\module{A}}(\omega) = \inf\left\{ \CDN(\xi) : \xi \in \module{B}, \Xi_\A(\xi) = \omega  \right\} \text{.}
  \end{equation*}
This result is symmetric in $\module{R}$ and $\module{P}$.

Therefore, by \cite[Theorem 3.1]{Latremoliere14}, we conclude that indeed $\tau$ is a modular tunnel, whose extent is at most $\tunnelextent{\tau_1} + \tunnelextent{\tau_2} + \varepsilon$.
\end{proof}

 As is customary with our work on the propinquity, we allow for more restrictive choices of the class of modular tunnels involved in the definition of our metric, as long as such restriction meets the following condition.

\begin{definition}\label{appropriate-def}
  Let $\mathcal{Q}$ be a nonempty class of {\QVB{F}{H}s}. A class $\mathcal{T}$ of modular tunnels is \emph{appropriate} for $\mathcal{Q}$ when:
  \begin{enumerate}
    \item if $\tau \in \mathcal{T}$ then $\dom{\tau} , \codom{\tau} \in \mathcal{Q}$,
    \item if $\mathds{A},\mathds{B} \in \mathcal{Q}$ then there exists a modular tunnel $\tau$ in $\mathcal{T}$ from $\mathds{A}$ to $\mathds{B}$,
    \item if $\mathcal{A} = (\module{M},\inner{\cdot}{\cdot}{\module{M}},\CDN,\A,\Lip)$ and $\mathcal{B}$ are in $\mathcal{Q}$ and if there exists a full modular quantum isometry $\Theta : \mathds{A} \rightarrow \mathds{B}$, then the modular tunnel $(\mathds{A},\mathrm{id},\Theta)$ is in $\mathcal{T}$ --- where $\mathrm{id} = (\mathrm{id}_\A,\mathrm{id}_{\module{M}})$ with $\mathrm{id}_\A :\A\rightarrow\A$ is the identity automorphism and $\mathrm{id}_{\module{M}}:\module{M}\rightarrow\module{M}$ is the identity map on $\module{M}$,
    \item if $\tau = (\mathds{D},\Theta,\Pi) \in \mathcal{T}$ then $\tau^{-1} = (\mathds{D}, \Pi, \Theta) \in \mathcal{T}$,
    \item if $\tau_1,\tau_2  \in \mathcal{T}$ with $\codom{\tau_1} = \dom{\tau_2}$ and if $\varepsilon > 0$, then there exists $\tau \in \mathcal{T}$ from $\dom{\tau_1}$ to $\codom{\tau_2}$ such that $\tunnelextent{\tau} \leq \tunnelextent{\tau_1} + \tunnelextent{\tau_2} + \varepsilon$.
  \end{enumerate}
\end{definition}

We record that the most natural choices of classes of tunnels are in fact appropriate for the natural classes of {\gQVB s}:

\begin{proposition}
  The class of all $(F,H)$-modular tunnels is appropriate for the class of all {\QVB{F}{H}s}.
\end{proposition}

\begin{proof}
  All properties of Definition (\ref{appropriate-def}) are obvious except for Assertion (5), which is established by Theorem (\ref{triangle-thm}).
\end{proof}

It is convenient to introduce a simple notation when working with classes of modular tunnels.

\begin{notation}
  Let $\mathcal{C}$ be a nonempty class of {\QVB{F}{H}s} and $\mathcal{T}$ be a class of modular tunnels appropriate for $\mathcal{C}$.

  Let $\mathds{A}$ and $\mathds{B}$ be two {\QVB{F}{H}s} in $\mathcal{C}$. The set of all $(F,H)$-modular tunnels from $\mathds{A}$ to $\mathds{B}$ in $\mathcal{T}$ is denoted by:
  \begin{equation*}
    \tunnelset{\mathds{A}}{\mathds{B}}{\mathcal{T}} \text{.}
  \end{equation*}

  In particular, the class of all $(F,H)$-tunnels from $\mathds{A}$ to $\mathds{B}$, with no restriction that they belong to $\mathcal{T}$, is denoted by:
  \begin{equation*}
    \tunnelset{\mathds{A}}{\mathds{B}}{F,H} \text{.}
  \end{equation*}
\end{notation}

We now define the titular object of this paper: the dual modular propinquity between {\gQVB s}.

\begin{definition}\label{dual-modular-propinquity-def}
  Let $\mathcal{C}$ be a nonempty class of {\QVB{F}{H}s}, for some permissible pair $(F,H)$ of functions, and $\mathcal{T}$ be a class of modular tunnels appropriate for $\mathcal{C}$.

  The $\mathcal{T}$-dual modular Gromov-Hausdorff propinquity $\dmodpropinquity{\mathcal{T}}(\mathds{A},\mathds{B})$ between $\mathds{A},\mathds{B} \in \mathcal{C}$ is the nonnegative number:
  \begin{equation*}
    \dmodpropinquity{\mathcal{T}}(\mathds{A},\mathds{B}) = \inf\left\{ \tunnelextent{\tau} : \tau \in \tunnelset{\mathds{A}}{\mathds{B}}{\mathcal{T}} \right\}
  \end{equation*}
\end{definition}

\begin{notation}
  For any permissible pair $(F,H)$, the dual-modular propinquity $\dmodpropinquity{\mathcal{T}}$, where $\mathcal{T}$ is the class of \emph{all} $(F,H)$-modular tunnels, is simply denoted by $\dmodpropinquity{F,H}$.
\end{notation}

Our work so far ensures that the dual-modular propinquity is a pseudo-metric.

\begin{proposition}\label{pseudo-metric-prop}
  Let $\mathcal{C}$ be a nonempty class of {\QVB{F}{H}s} and $\mathcal{T}$ be a class of modular tunnels appropriate for $\mathcal{C}$.

  For all $\mathds{A},\mathds{B},\mathds{D} \in \mathcal{C}$, we have:
  \begin{enumerate}
    \item $\dmodpropinquity{\mathcal{T}}(\mathds{B},\mathds{A}) = \dmodpropinquity{\mathcal{T}}(\mathds{A},\mathds{B}) < \infty$,
    \item $\dmodpropinquity{\mathcal{T}}(\mathds{A},\mathds{D}) \leq \dmodpropinquity{\mathcal{T}}(\mathds{A},\mathds{B}) + \dmodpropinquity{\mathcal{T}}(\mathds{B},\mathds{D})$,
    \item if there exists a full modular quantum isometry $\Pi : \mathds{A} \rightarrow \mathds{B}$ then $\dmodpropinquity{\mathcal{T}}(\mathds{A},\mathds{B}) = 0$; in particular $\dmodpropinquity{\mathcal{T}}(\mathds{A},\mathds{A}) = 0$.
    \item if $(\A,\Lip_\A)$  is the base space of $\mathds{A}$ and $(\B,\Lip_\B)$ is the base space of $\mathds{B}$, then:
      \begin{equation*}
        \dpropinquity{\mathcal{T'}}((\A,\Lip_\A),(\B,\Lip_\B)) \leq \dmodpropinquity{\mathcal{T}}(\mathds{A},\mathds{B}) \text{,}
      \end{equation*}
      where $\mathcal{T}' = \left\{ \tau_\flat : \tau \in \mathcal{T} \right\}$.
    \end{enumerate}

    Moreover, we also record that for any permissible pair $(F,H)$ and for all {\QVB{F}{H}s} $\mathds{A}$, $\mathds{B}$:
    \begin{enumerate}
      \setcounter{enumi}{4}
    \item $\dmodpropinquity{F,H}(\mathds{A},\mathds{B}) \leq \max\left\{2,\diam{\A}{\Lip_\A},\diam{\B}{\Lip_\B}\right\}$,
    \item $\dmodpropinquity{F,H}(\mathds{A},\mathds{B}) \leq \modpropinquity{F,H}(\mathds{A},\mathds{B})$.
  \end{enumerate}
\end{proposition}

\begin{proof}
  The first four properties listed in this proposition reflect the properties defining an appropriate class from Definition (\ref{appropriate-def}).

  Let $\varepsilon > 0$. There exists a modular tunnel $\tau_1$ from $\mathds{A}$ to $\mathds{B}$ and a tunnel $\tau_2$ from $\mathds{B}$ to $\mathds{E}$ such that:
  \begin{equation*}
    \tunnelextent{\tau_1} \leq \dmodpropinquity{\mathcal{T}}(\mathds{A},\mathds{B}) + \frac{\varepsilon}{3} \text{ and }\tunnelextent{\tau_2} \leq \dmodpropinquity{\mathcal{T}}(\mathds{B},\mathds{D}) + \frac{\varepsilon}{3} \text{.}
  \end{equation*}

  Then, since $\tunnelextent{\tau_1^{-1}} = \tunnelextent{\tau_1}$, we have:
  \begin{equation*}
    \dmodpropinquity{\mathcal{T}}(\mathds{B},\mathds{A}) \leq \tunnelextent{\tau_1^{-1}} = \tunnelextent{\tau_1} \leq \dmodpropinquity{\mathcal{T}}(\mathds{A},\mathds{B}) + \frac{\varepsilon}{3}
  \end{equation*}
and thus $\dmodpropinquity{\mathcal{T}}(\mathds{B},\mathds{A}) \leq \dmodpropinquity{\mathcal{T}}(\mathds{A},\mathds{B})$ as $\varepsilon > 0$ is arbitrary. Thus (1) follows by symmetry.

Similarly, by Definition (\ref{appropriate-def}), there exists $\tau \in \mathcal{T}$ from $\mathds{A}$ to $\mathds{D}$ with extent at most $\tunnelextent{\tau_1} + \tunnelextent{\tau_2} + \frac{\varepsilon}{3}$. Hence:
\begin{align*}
  \dmodpropinquity{\mathcal{T}}(\mathds{A},\mathds{D}) &\leq \tunnelextent{\tau} \\
  &\leq \tunnelextent{\tau_1} + \tunnelextent{\tau_2} + \frac{\varepsilon}{3} \\
  &\leq \dmodpropinquity{\mathcal{T}}(\mathds{A},\mathds{B}) + \frac{\varepsilon}{3} + \dmodpropinquity{\mathcal{T}}(\mathds{B},\mathds{D}) + \frac{\varepsilon}{3} + \frac{\varepsilon}{3} \\
  &\leq \dmodpropinquity{\mathcal{T}}(\mathds{A},\mathds{B}) + \dmodpropinquity{\mathcal{T}}(\mathds{B},\mathds{D}) + \varepsilon\text{.}
\end{align*}

Again, as $\varepsilon > 0$ is arbitrary, it follows that (2) holds.

Now, assume that $(\pi,\Pi)$ is a full modular quantum isometry form $\mathds{A}$ to $\mathds{B}$ and write $\mathds{A} = (\module{M},\inner{\cdot}{\cdot}{\module{M}},\CDN,\A,\Lip)$. Then we note that $\tau = (\mathds{A},(\pi,\Pi),(\mathrm{id}_{\A},\mathrm{id}_{\module{M}})) \in \mathcal{T}$ by Definition (\ref{appropriate-def}), where $\mathrm{id}_E$ is the identity map of any set $E$. Now, $\tau_\flat$ is of the form $(\A,\Lip,\pi,\mathrm{id}_\A)$, which has extent 0 (since $\pi$ is a full quantum isometry). Thus (3) holds.

Further, it is a straightforward check that $\mathcal{T}'$ as defined is indeed a class of tunnels appropriate for the class of all base spaces of {\gQVB s} in $\mathcal{C}$, and since the extent of a modular tunnel is the extent of its base tunnel, (4) holds as well.

Assertion (5) follows from Corollary (\ref{diam-bound-cor}), and thus we obtain the desired bound.

Assertion (6) follows from Corollary (\ref{tunnel-from-bridge-cor}) of Theorem (\ref{tunnel-from-bridge-thm}) since $\dmodpropinquity{\mathcal{T}}$ satisfies the triangle inequality.
\end{proof}

If $\dmodpropinquity{\mathcal{T}}(\mathds{A},\mathds{B}) = 0$ then the base spaces of $\mathds{A}$ and $\mathds{B}$ are fully quantum isometric by \cite{Latremoliere13b} and Assertion (4) of Proposition (\ref{pseudo-metric-prop}). We want to prove that in fact, under this condition, more is true: $\mathds{A}$ and $\mathds{B}$ are fully modular-quantum isometric. To this end, as in \cite{Latremoliere13,Latremoliere13b,Latremoliere16c,Latremoliere18a}, we study the morphism-like properties of modular tunnels. These properties are expressed using target sets, as defined below, which are compact-set valued maps induced by tunnels, somewhat akin to correspondences in metric geometry.

\begin{definition}\label{targetset-def}
  Let $\mathds{A} = (\module{M}, \inner{\cdot}{\cdot}{\module{M}},\CDN_{\module{M}},\A,\Lip_\A)$ and $\mathds{B} = (\module{N},\inner{\cdot}{\cdot}{\module{N}},\CDN_{\module{N}},\B,\Lip_\B)$ be two {\gQVB s}.  Let $\tau$ be a modular tunnel from $\mathds{A}$ to $\mathds{B}$. 

  For any $a\in\dom{\Lip_\A}$ and $l\geq\Lip_\A(a)$, the \emph{$l$-target set} $\targetsettunnel{\tau}{a}{l}$ is $\targetsettunnel{\tau_\flat}{a}{l}$, i.e.:
  \begin{equation*}
    \targetsettunnel{\tau}{a}{l} = \left\{ \pi_\B(d) : d\in\pi_\A^{-1}\left(\left\{d\right\}\right), \Lip_\D(d) \leq l \right\} \text{.}
  \end{equation*}

  For any $\omega\in\dom{\CDN_{\module{M}}}$ and $l \geq \CDN_{\module{M}}(\omega)$, the \emph{$l$-target set} of $\omega$ is the subset of $\module{N}$ defined by:
\begin{equation*}
  \targetsettunnel{\tau}{\omega}{l} = \left\{ \Theta_\B(\zeta) : \zeta \in \Theta_\B^{-1}\left(\left\{\omega\right\}\right), \CDN(\zeta) \leq l \right\} \text{.}
\end{equation*}
\end{definition}

We recall from \cite{Latremoliere13b,Latremoliere14}:

\begin{proposition}\label{targetset-diam-prop}
  If $\tau$ is a tunnel from $(\A,\Lip_\A)$ to $(\B,\Lip_\B)$, and if $a\in\dom{\Lip_\A}$ with $l\geq \Lip_\A(a)$, then for all $b \in \targetsettunnel{\tau}{a}{l}$ then:
  \begin{equation*}
    \norm{b}{\B} \leq \norm{a}{\A} + l \tunnelextent{\tau}\text{ and }\diam{\targetsettunnel{\tau}{a}{l}}{\norm{\cdot}{\B}} \leq 2 l \tunnelextent{\tau}\text{.}
  \end{equation*}
\end{proposition}

\begin{proof}
  We apply \cite[Proposition 2.12]{Latremoliere14} to \cite[Proposition 4.4]{Latremoliere13b}.
\end{proof}

We now can establish the morphism-like properties of the target sets as set-valued functions defined on modules.

\begin{proposition}\label{targetset-morphism-prop}
  Let $\mathds{A} = (\module{M}, \inner{\cdot}{\cdot}{\module{M}},\CDN_{\module{M}},\A,\Lip_\A)$ and $\mathds{B} = (\module{N},\inner{\cdot}{\cdot}{\module{N}},\CDN_{\module{N}},\B,\Lip_\B)$ be two {\gQVB s}.  Let $\tau$ be a modular tunnel from $\mathds{A}$ to $\mathds{B}$. 

  Let $\omega, \omega' \in \module{M}$, $l \geq \max\left\{\CDN_{\module{M}}(\omega),\CDN_{\module{M}}(\omega')\right\}$. If $\eta \in \targetsettunnel{\tau}{\omega}{l}$ and $\eta' \in \targetsettunnel{\tau}{\omega'}{l}$, then:
  \begin{equation*}
    \KantorovichMod{\CDN_{\module{N}}}(\eta,\eta') \leq \sqrt{2} \left( \KantorovichMod{\CDN_{\module{M}}}(\omega,\omega') + H(2 l,  1) \tunnellength{\tau} \right)
  \end{equation*}
  In particular:
  \begin{equation*}
    \diam{\targetsettunnel{\tau}{\omega}{l}}{\KantorovichMod{\CDN_{\module{N}}}} \leq \sqrt{2} \left( H(2 l, 1) \tunnellength{\tau}\right) \text{.}
  \end{equation*}

  We also have for all $t \in \C$:
  \begin{equation*}
    \eta + t \eta' \in \targetsettunnel{\tau}{\omega + t \omega'}{l(1+|t|)} \text{.}
  \end{equation*}

  Last, if $b \in \targetsettunnel{\tau}{\inner{\omega}{\omega}{\module{M}}}{H(l,l)}$ and if $\eta \in \targetsettunnel{\tau}{\omega}{l}$ then:
  \begin{equation*}
    \norm{b - \inner{\eta}{\eta}{\B}}{\B} \leq 2 H(l,l) \tunnelextent{\tau} \text{.}
  \end{equation*}
\end{proposition}

\begin{proof}
  Write $\tau = (\mathds{D},(\theta_\A,\Theta_\A),(\theta_\B,\Theta_\B))$ with $\mathds{D} = (\module{P},\inner{\cdot}{\cdot}{\module{P}},\CDN,\D,\Lip_\D)$.

  Let $\eta\in\targetsettunnel{\tau}{\omega}{l}$ and $\eta'\in\targetsettunnel{\tau}{\omega'}{l}$. Let  $\xi \in \Theta_\A^{-1}(\{\omega\})$ with $\CDN(\xi) \leq l$ such that $\Theta_\B(\xi) = \eta$, and similarly, let  $\xi' \in \Theta_\A^{-1}(\{\omega'\})$ with $\CDN(\xi') \leq l$ such that $\Theta_\B(\xi') = \eta'$.
  
  Let $\mu = \xi - \xi'$, and note $\CDN(\mu) \leq 2l$. We have $\Theta_\A(\mu) = \omega-\omega'$ and $\Theta_\B(\mu) = \eta-\eta'$.

  Let $\chi \in \module{N}$ with $\CDN_{\module{N}}(\chi)\leq 1$. Let $\nu \in \module{P}$ with $\CDN(\nu)\leq 1$ and $\chi = \Theta_\B(\nu)$. Write $\zeta = \Theta_\A(\nu)$. We have:
  \begin{align*}
    \Lip_\D\left(\Re\inner{\mu}{\nu}{}\right) \leq H(\CDN(\mu),\CDN(\nu)) = H(2l,1) \text{.}
  \end{align*}
  Moreover by linearity:
  \begin{equation*}
    \theta_\A\left(\Re\inner{\mu}{\nu}{\D}\right) = \Re\inner{\omega-\omega'}{\zeta}{\A}
  \end{equation*}
  so (noting that $\Lip_\A(\Re\inner{\omega-\omega'}{\zeta}{\A}) \leq H(2l,1)$):
  \begin{equation*}
    \Re\inner{\eta-\eta'}{\chi}{\B} = \theta_\B\left(\Re\inner{\mu}{\nu}{\D}\right) \in \targetsettunnel{\tau}{\Re\inner{\omega-\omega'}{\zeta}{\A}}{H(2l,1)} \text{.}
  \end{equation*}
  We then conclude by Proposition (\ref{targetset-diam-prop}):
  \begin{align*}
    \norm{\Re\inner{\eta-\eta'}{\chi}{\B}}{\B} 
    &\leq \norm{\Re\inner{\omega-\omega'}{\zeta}{\A}}{\A} + H(2l,1) \tunnellength{\tau}\\
    &\leq \norm{\inner{\omega-\omega'}{\zeta}{\A}}{\A} + H(2l,1) \tunnellength{\tau} \\
    &\leq \KantorovichMod{\CDN_{\module{M}}}(\omega,\omega') + H(2l,1) \tunnellength{\tau} \text{.}
  \end{align*}
  
  Now, the same as above applies to conclude that:
  \begin{equation*}
    \norm{\Im\inner{\eta-\eta'}{\chi}{\B}}{\B} \leq \KantorovichMod{\CDN_{\module{M}}}(\omega,\omega') + H(2l,1) \tunnellength{\tau} \text{.}
  \end{equation*}
  
  Hence we conclude:
  \begin{equation*}
    \KantorovichMod{\CDN_{\module{N}}}(\eta,\eta') \leq \sqrt{2}\left(\KantorovichMod{\CDN_{\module{M}}}(\omega,\omega') + H(2l,1) \tunnellength{\tau}\right) \text{.}
  \end{equation*}
  
  In particular, if $\omega=\omega'$, we then have:
  \begin{align*}
    \diam{\targetsettunnel{\tau}{\omega}{l}}{\KantorovichMod{\CDN_{\module{N}}}} 
    &= \sup\left\{ \KantorovichMod{\CDN_{\module{N}}}(\eta,\eta') : \eta,\eta' \in \targetsettunnel{\tau}{\omega}{l} \right\} \\
    &\leq \sqrt{2} H( 2l, 1 ) \tunnellength{\tau} \text{.}
  \end{align*}
  
  We now check the algebraic properties of the target sets. Let $t\in\C$. We note that $\CDN(\xi + t\xi') \leq (1+|t|) l$, while $\Theta_\A(\xi + t\xi') = \omega + t\omega'$ and $\Theta_\B(\xi+t\xi') = \eta + t\eta'$ by linearity. Hence by definition, $\eta + t \eta' \in \targetsettunnel{\tau}{\omega + t\omega'}{l(1+|t|)}$ as desired.

  Last, we observe that $\inner{\omega}{\omega}{\A}\in\sa{\A}$ with $\Lip_\A(\inner{\omega}{\omega}{\A}) \leq H(l,l)$, and $\inner{\eta}{\eta}{\B} \in \sa{\B}$, $\inner{\xi}{\xi}{\D} \in \sa{\D}$, with $\theta_\A(\inner{\xi}{\xi}{\D}) = \inner{\omega}{\omega}{\A}$, $\theta_\B(\inner{\xi}{\xi}{\D}) = \inner{\eta}{\eta}{\A}$, and $\Lip_\D(\inner{\xi}{\xi}{\D}) \leq H(l,l)$, so:
  \begin{equation*}
    \inner{\eta}{\eta}{\B} \in \targetsettunnel{\tau}{\inner{\omega}{\omega}{\A}}{H(l,l)} \text{.}
  \end{equation*}

  Therefore if $b \in \targetsettunnel{\tau}{\inner{\omega}{\omega}{\A}}{H(l,l)}$ then by Proposition (\ref{targetset-diam-prop}):
  \begin{equation*}
    \norm{b - \inner{\eta}{\eta}{\B}}{\B} \leq 2 H(l,l) \tunnellength{\tau} \text{.}
  \end{equation*}

This concludes our proof.
\end{proof}

We also record that target sets are topologically small, i.e. formally, compact, while also not empty.

\begin{proposition}\label{targetset-compact-prop}
Let $\mathds{A} = (\module{M},\inner{\cdot}{\cdot}{\A},\CDN_{\module{M}},\A,\Lip)$ and $\mathds{B}$ be two {\gQVB s} and $\tau$ be a modular tunnel from $\mathds{A}$ to $\mathds{B}$. If $\omega\in\module{M}$ and $l\geq \CDN_{\module{M}}(\omega)$, then $\targetsettunnel{\tau}{\omega}{l}$ is a nonempty compact subset of $\module{N}$.
\end{proposition}

\begin{proof}
  Write $\tau = (\mathds{D},(\theta_\A,\Theta_\A),(\theta_\B,\Theta_\B))$ with $\mathds{D} = (\module{P},\inner{\cdot}{\cdot}{\module{P}},\CDN,\D,\Lip_\D)$. By definition:
  \begin{equation*}
    \targetsettunnel{\tau}{\omega}{l} = \Theta_\B\left( \Theta_\A^{-1}\left( \{ \omega  \} \right) \cap \left\{ \eta \in \module{P} : \CDN_\D(\eta) \leq l \right\} \right) \text{.}
  \end{equation*}
 By Remark (\ref{reached-remark}), $\Theta_\A^{-1}\left( \{ \omega  \} \right) \cap \left\{ \eta \in \module{P} : \CDN_\D(\eta) \leq l \right\}$ is compact in $\module{P}$ and not empty. As $\Theta_\B$ is continuous, we conclude that indeed $\targetsettunnel{\tau}{\omega}{l}$ is nonempty, compact in $\module{N}$.
\end{proof}

We are now able to prove our main theorem for this section.

\begin{theorem}\label{zero-thm}
  Let $(F,H)$ be a pair of permissible functions. Let $\mathcal{C}$ be a nonempty class of {\QVB{F}{H}s} and let $\mathcal{T}$ be a class of tunnels appropriate for $\mathcal{C}$.

  The dual modular $\mathcal{T}$-propinquity $\dmodpropinquity{\mathcal{T}}$ is a metric, up to full quantum isometry on $\mathcal{C}$.
\end{theorem}

\begin{proof}
  Proposition (\ref{pseudo-metric-prop}) gives us that the dual modular propinquity is a pseudo-metric. We are left to prove that distance zero implies full modular quantum isometry.

  Let $\mathds{A} = (\module{M},\inner{\cdot}{\cdot}{\module{M}},\CDN_{\module{M}},\A,\Lip_\A)$ and $\mathds{B} = (\module{N},\inner{\cdot}{\cdot}{\module{N}},\CDN_{\module{N}},\B,\Lip_\B)$ be two {\QVB{F}{H}s} such that $\dmodpropinquity{\mathcal{T}}(\mathds{A},\mathds{B}) = 0$.

  By Definition (\ref{dual-modular-propinquity-def}), for all $n\in\N$, there exists a modular tunnel $\tau_n$ from $\mathds{A}$ to $\mathds{B}$ such that $\tunnelextent{\tau_n} \leq \frac{1}{n+1}$.

  Our proof follows most of the claims of \cite[Theorem 6.11]{Latremoliere16c}, replacing \cite[Proposition 6.7]{Latremoliere16c} with Proposition (\ref{targetset-morphism-prop}) and \cite[Proposition 6.8]{Latremoliere16c} with Proposition (\ref{targetset-compact-prop}). We however need a small change since we are now missing Assertion (4) of \cite[Proposition 6.7]{Latremoliere16c}. We thus provide below the modified proof of \cite[Theorem 6.11]{Latremoliere16c}, though we refer to \cite{Latremoliere16c} for supplementary details.

  The first step is to appeal to our proof of \cite[Theorem 4.16]{Latremoliere13b}, itself based on \cite[Theorem 5.13]{Latremoliere13}, which established the desired coincidence property for the dual propinquity, upon which the dual-modular propinquity is built --- with the immediate modification that we use here the extent instead of the length of tunnels (which are equivalent by \cite[Proposition 2.12,Theorem 3.8]{Latremoliere14}), and thus we replace target sets for journeys by the usual target sets for tunnels as per Definition (\ref{targetset-def}). Thus, we recall from \cite{Latremoliere13b} that there exists a strictly increasing function $f : \N \rightarrow \N$ and a full quantum isometry $\theta : (\A,\Lip_\A) \rightarrow (\B,\Lip_\B)$ such that, for all $a\in\dom{\Lip_\A}$ and $l\geq \Lip_\A(a)$, and for all $b \in \dom{\Lip_\B}$ and $l' \geq \Lip_\B(b)$:
  \begin{equation*}
    \lim_{n\rightarrow\infty} \Haus{\norm{\cdot}{\B}}\left(\targetsettunnel{\tau_{f(n)}}{a}{l}, \{\theta(a)\}\right) = 0
  \end{equation*}
  and
  \begin{equation*}
    \lim_{n\rightarrow\infty} \Haus{\norm{\cdot}{\A}}\left(\targetsettunnel{\tau_{f(n)}^{-1}}{b}{l'}, \{\theta^{-1}(b)\}\right) = 0 \text{.}
  \end{equation*}

  We now turn to working with the modules. First, let $\omega \in \dom{\CDN_{\module{M}}}$, and let $l\geq \CDN_{\module{M}}(\omega)$. Let $g : \N \rightarrow \N$ be any strictly increasing function. By Proposition (\ref{targetset-compact-prop}), the sequence $(\targetsettunnel{\tau_{f(g(n))}}{\omega}{l})_{n\in\N}$ is a sequence of nonempty compact subsets of the $\KantorovichMod{\CDN_{\module{N}}}$-compact set $\mathcal{B}_l = \CDN_{\module{N}}^{-1}[0,l]$. Hence, there exists a subsequence $(\targetsettunnel{\tau_{f(g(h(n)))}}{\omega}{l})_{n\in\N}$ converging to some nonempty set $\mathcal{L}$ for the Hausdorff distance $\Haus{\KantorovichMod{\CDN_{\module{N}}}}$ --- since the space of closed subsets of a compact metric space is compact for the induced Hausdorff distance. By Proposition (\ref{targetset-morphism-prop}), the diameter of $\mathcal{L}$ is the limit of $(\diam{\targetsettunnel{\tau_{f(g(h(n)))}}{\omega}{l}}{\KantorovichMod{\CDN_{\module{N}}}})_{n\in\N}$, which is dominated by $\lim_{n\rightarrow\infty} \sqrt{2} H(2l, 1) \tunnellength{\tau_n} = 0$, so $\mathcal{L}$ is a singleton.

  We then observe that since $\KantorovichMod{\CDN_{\module{N}}}$ and $\norm{\cdot}{\module{N}}$ induce the same topology on the unit ball of $\CDN_{\module{N}}$, and therefore on all bounded subsets of $\dom{\CDN_{\module{N}}}$ for $\CDN_{\module{N}}$ by Proposition (\ref{modular-mongekant-prop}), so do $\Haus{\KantorovichMod{\CDN_{\module{N}}}}$ and $\Haus{\norm{\cdot}{\module{N}}}$ induce the same topology on the space of closed subset of any closed ball for $\CDN_{\module{N}}$ (see for instance \cite[Lemma 6.9]{Latremoliere16c}). Since $(\targetsettunnel{\tau_{f(g(h(n)))}}{\omega}{l})_{n\in\N}$ lies inside $\mathcal{B}_l = \{\eta\in\module{N}:\CDN_{\module{N}}(\eta)\leq l\}$, we can conclude that $(\targetsettunnel{\tau_{f(g(h(n)))}}{\omega}{l})_{n\in\N}$ converges to the singleton $\mathcal{L}$ for $\Haus{\norm{\cdot}{\module{N}}}$ as well.

  As $\{\omega\in\module{M}:\CDN(\omega)\leq 1\}$ is compact for the norm $\norm{\cdot}{\module{M}}$, we conclude that it is separable. Therefore, so is $\dom{\CDN_{\module{M}}} = \bigcup_{N\in\N} N\cdot\{\omega\in\dom{\CDN_{\module{M}}}:\CDN_{\module{M}}(\omega)\leq 1\}$. Let $\{\omega^n : n\in\N \} \subseteq \dom{\CDN_{\module{M}}}$ be a $\norm{\cdot}{\module{M}}$-dense subset of $\dom{\CDN_{\module{M}}}$. We can then use a diagonal argument (similar to, for instance, \cite[Claim 6.14]{Latremoliere16c}) to conclude that there exists a strictly increasing function $g : \N \rightarrow \N$ such that for all $k\in\N$, the sequence $\left(\targetsettunnel{\tau_{f(g(n))}}{\omega^k}{\CDN_{\module{M}}(\omega^k)}\right)_{n\in\N}$ converges, for $\Haus{\KantorovichMod{\CDN_{\module{N}}}}$, to a singleton which we denote $\{\Theta(\omega^k)\}$. By \cite[Claim 6.13]{Latremoliere16c}, we then conclude that for all $k\in\N$ and $l\geq\CDN_{\module{M}}(\omega^k)$, the sequence $\left(\targetsettunnel{\tau_{f(g(n))}}{\omega^k}{l} \right)_{n\in\N}$ still converges to $\{\Theta(\omega^k)\}$ for $\Haus{\KantorovichMod{\CDN_{\module{N}}}}$ and therefore, as discussed before, for $\Haus{\norm{\cdot}{\module{N}}}$.

  We can now prove that, in fact, for all $\omega \in \dom{\CDN_{\module{M}}}$ and for all $l\geq\CDN_{\module{M}}(\omega)$, the sequence $\left(\targetsettunnel{\tau_{f(g(n))}}{\omega}{l}\right)_{n\in\N}$ converges to a singleton for $\Haus{\norm{\cdot}{\module{N}}}$. Indeed, let $\varepsilon > 0$. There exists, by density, $k\in\N$ such that:
  \begin{equation*}
    \KantorovichMod{\CDN_{\module{M}}}(\omega,\omega_k) \leq \norm{\omega-\omega_k}{\module{M}} < \frac{\varepsilon}{6\sqrt{2}}\text{.}
  \end{equation*}
  There also exists $N\in\N$ such that for all $n\geq N$, we have $\tunnellength{\tau_{f(g(n))}}\leq\frac{\varepsilon}{6H(2l,1)\sqrt{2}}$. By Proposition (\ref{targetset-morphism-prop}), we then note that for all $n\in\N$:
  \begin{align*}
    \Haus{\KantorovichMod{\CDN_{\module{N}}}}(\targetsettunnel{\tau_{f(g(p))}}{\omega^k}{l}, \targetsettunnel{\tau_{f(g(p))}}{\omega}{l}) 
      &\leq \sqrt{2}\left(\KantorovichMod{\module{M}}(\omega,\omega^k) + H(2l,1) \tunnellength{\tau_{f(g(n))}}\right)\\
      &\leq \sqrt{2} \left(\frac{\varepsilon}{6\sqrt{2}} + \frac{\varepsilon}{6\sqrt{2}}\right) \leq \frac{\varepsilon}{3} \text{.}
  \end{align*}

    Therefore, for all $p,q \geq N$, we conclude:
    \begin{equation*}
      \Haus{\KantorovichMod{\CDN_{\module{N}}}}\left(\targetsettunnel{\tau_{f(g(p))}}{\omega}{l},\targetsettunnel{\tau_{f(g(q))}}{\omega}{l}\right) \leq \frac{2\varepsilon}{3} + \Haus{\KantorovichMod{\CDN_{\module{N}}}}\left(\targetsettunnel{\tau_{f(g(p))}}{\omega^k}{l},\targetsettunnel{\tau_{f(g(q))}}{\omega^k}{l}\right) \text{.}
    \end{equation*}

    Now, the sequence $\left(\targetsettunnel{\tau_{f(g(n))}}{\omega^k}{l}\right)_{n\in\N}$ is convergent, hence Cauchy, for $\KantorovichMod{\CDN_{\module{N}}}$. Hence, there exists $N_2 \in \N$ such that if $p,q \geq N_2$ then:
    \begin{equation*}
      \Haus{\KantorovichMod{\CDN_{\module{N}}}}\left(\targetsettunnel{\tau_{f(g(p))}}{\omega^k}{l},\targetsettunnel{\tau_{f(g(q))}}{\omega^k}{l}\right) < \frac{\varepsilon}{3} \text{.}
    \end{equation*}
    Thus, for all $p,q\geq \max\{N,N_2\}$, we conclude:
    \begin{equation*}
      \Haus{\KantorovichMod{\CDN_{\module{N}}}}\left(\targetsettunnel{\tau_{f(g(p))}}{\omega}{l},\targetsettunnel{\tau_{f(g(q))}}{\omega}{l}\right) \leq \varepsilon \text{.}
    \end{equation*}
    Now, $\left\{ \eta \in \module{N} : \CDN_{\module{N}}(\eta)\leq l\right\}$ is compact, therefore complete, for $\KantorovichMod{\CDN_{\module{N}}}$ by Proposition (\ref{modular-mongekant-prop}). Therefore, $\Haus{\KantorovichMod{\CDN_{\module{N}}}}$ is complete on the collection of nonempty closed subsets of $\left\{ \eta \in \module{N} : \CDN_{\module{N}}(\eta)\leq l\right\}$. As a Cauchy sequence in that space, $\left(\targetsettunnel{\tau_{f(g(n))}}{\omega}{l}\right)_{n\in\N}$ therefore converges for $\Haus{\KantorovichMod{\CDN_{\module{N}}}}$ since the Hausdorff distance induced by a complete metric is itself complete. By Proposition (\ref{targetset-morphism-prop}), we conclude again that its limit is a singleton which we denote by $\{\Theta(\omega)\}$. By \cite[Claim 6.13]{Latremoliere16c}, we conclude that $\left(\targetsettunnel{\tau_{f(g(n))}}{\omega}{l'}\right)_{n\in\N}$ converges to the same singleton for all $l'  \geq \CDN_{\module{M}}(\omega)$. By topological equivalence again, we conclude that $\left(\targetsettunnel{\tau_{f(g(n))}}{\omega}{l'}\right)_{n\in\N}$ converges for $\Haus{\norm{\cdot}{\module{N}}}$ to $\{\Theta(\omega)\}$ for all $l'\geq\CDN_{\module{N}}(\omega)$.

  For any $\omega\in\dom{\CDN_{\module{M}}}$ and $l\geq\CDN_{\module{N}}(\omega)$, the element $\Theta(\omega)$ is the limit of any sequence of the form $(\eta_n)_{n\in\N}$ with $\eta_n \in \targetsettunnel{\tau_{f(n)}}{\omega}{l}$ for all $n\in\N$ by \cite[Lemma 6.10]{Latremoliere16c}. This has several consequences. First, as $\CDN_{\module{N}}$ is lower semi-continuous, we conclude that $\CDN_{\module{N}}(\Theta(\omega)) \leq \CDN_{\module{M}}(\omega)$. 

Second, let $\omega,\omega'\in\dom{\CDN_{\module{M}}}$, $t\in \C$ and $l\geq\max\{\CDN_{\module{N}}(\omega),\CDN_{\module{N}}(\omega')\}$. Then for all $n\in\N$, let $\eta_n \in \targetsettunnel{\tau_{f(g(n))}}{\omega}{l}$ and $\eta'_n \in \targetsettunnel{\tau_{f(g(n))}}{\omega'}{l}$. While $(\eta_n)_{n\in\N}$ converges to $\Theta(\omega)$ and $(\eta'_n)_{n\in\N}$ converges to $\Theta(\omega')$, Proposition (\ref{targetset-morphism-prop}) shows that $\eta_n + t\eta'_n \in \targetsettunnel{\tau_{f(g(n))}}{\omega+t\omega'}{(1+|t|)l}$, so $(\eta_n + \eta'_n)_{n\in\N}$ converges to $\Theta(\omega+t\omega')$. By uniqueness of the limit, we conclude that $\Theta(\omega)+t\Theta(\omega') = \Theta(\omega+t\omega')$. Thus $\Theta$ is $\C$-linear on $\dom{\CDN_{\module{M}}}$.

Third, for each $n\in\N$, let $\eta_n \in \targetsettunnel{\tau_{f(g(n))}}{\omega}{l}$ and $b_n \in \targetsettunnel{\tau_{f(g(n))}}{\inner{\omega}{\omega}{\module{M}}}{H(l,l)}$ for all $n\in\N$. While $(b_n)_{n\in\N}$ converges to $\theta(\inner{\omega}{\omega}{\module{M}})$ and $(\eta_n)_{n\in\N}$ converges to $\Theta(\omega)$, so that by continuity $\left(\inner{\eta_n}{\eta_n}{\module{N}}\right)_{n\in\N}$ converges to $\inner{\Theta(\omega)}{\Theta(\omega)}{\module{N}}$, by Proposition (\ref{targetset-morphism-prop}), we have $\lim_{n\rightarrow\infty}\norm{b_n - \inner{\eta_n}{\eta_n}{\module{N}}}{\B} = 0$. Therefore, again by uniqueness of the limit, $\theta(\inner{\omega}{\omega}{\module{M}}) = \inner{\Theta(\omega)}{\Theta(\omega)}{\module{N}}$. Using the standard polarization identity, we thus obtain:
\begin{equation*}
  \forall \omega,\eta\in \dom{\CDN_{\module{M}}} \quad \theta\left(\inner{\omega}{\eta}{\module{M}}\right) = \inner{\Theta(\omega)}{\Theta(\eta)}{\module{N}} \text{.}
\end{equation*}

Moreover, we conclude from the identity $\theta(\inner{\omega}{\omega}{\module{M}}) = \inner{\Theta(\omega)}{\Theta(\omega)}{\module{N}}$ that $\norm{\Theta(\omega)}{\module{N}} = \norm{\omega}{\module{M}}$ since $\theta$, as a *-automorphism, is an isometry from $\norm{\cdot}{\A}$ to  $\norm{\cdot}{\B}$. Therefore, $\Theta$, as it is linear, is an isometry, and in particular, is uniformly continuous on $\dom{\CDN_{\module{M}}}$ from $\norm{\cdot}{\module{M}}$ to $\norm{\cdot}{\module{N}}$. It therefore admits a unique linear, continuous extension to $\module{M}$, which we continue to denote by $\Theta$. By continuity of $\Theta$, $\theta$ and the inner products, we then easily verify that:
\begin{equation*}
  \forall \omega,\eta \in \module{M} \quad \theta(\inner{\omega}{\eta}{\module{M}}) = \inner{\Theta(\omega)}{\Theta(\eta)}{\module{N}} \text{.}
\end{equation*}

We observe that unlike \cite[Claim 6.19]{Latremoliere16c}, we have not yet proven that $(\theta,\Theta)$ is a module morphism. We will see later on that this holds true, but in quite a different way from our previous work, as it does not rely on Assertion (4) of \cite[Proposition 6.7]{Latremoliere16c} which we do not have any longer. For now, let us observe the following --- note that his is the only point in this proof where we use the fact that we work with an $\A$-inner product $\inner{\cdot}{\cdot}{\A}$.

Let $\omega\in\module{M}$ and $a\in\A$. Using what we have proven so far, in particular, that $\Theta$ preserves the inner product up to $\theta$, and that $\theta$ is a *-automorphism, we see that for all $\zeta\in\module{M}$:
\begin{equation}\label{module-morphism-eq}
  \begin{split}
    \inner{\Theta(a\omega) - \theta(a)\Theta(\omega)}{\Theta(\zeta)}{\module{N}} 
    &= \inner{\Theta(a\omega)}{\Theta(\zeta)}{\module{N}} - \theta(a)\inner{\Theta(\omega)}{\Theta(\zeta)}{\module{N}}\\
    &= \theta\left(\inner{a\omega}{\zeta}{\module{M}}\right) - \theta(a)\theta\left(\inner{\omega}{\zeta}{\module{M}}\right) \\
    &= \theta\left(\inner{a\omega}{\zeta}{\module{M}}\right) - \theta(a\inner{\omega}{\zeta}{\module{M}}) \\
    &= \theta(\inner{a\omega}{\zeta}{\module{M}} -\inner{a\omega}{\zeta}{\module{M}}) = 0 \text{.}
  \end{split}
\end{equation}

Thus, it would be sufficient to show that $\Theta$ is onto to conclude that $(\theta,\Theta)$ is a module morphism.

  We now in fact prove that $\Theta$ is invertible. The entire work we have done so far can be done as well with the sequence of modular tunnels $(\tau_{f(g(n))}^{-1})_{n\in\N}$. Thus, we would prove that there exists a $\C$-linear map $\Pi : \module{N} \rightarrow \module{M}$ such that for all $\omega,\eta\in\module{N}$:
  \begin{equation*}
    \pi^{-1}\left(\inner{\omega}{\eta}{\module{N}}\right) = \inner{\Pi(\omega)}{\Pi(\eta)}{\module{N}}
  \end{equation*}
  noting in particular that $\Pi$ is an isometry, and such that there exists a strictly increasing function $h : \N \rightarrow \N$ such that for all $\omega\in\dom{\CDN_{\module{N}}}$ and $l\geq\CDN_{\module{N}}(\omega)$:
  \begin{equation*}
    \lim_{n\rightarrow\infty} \Haus{\KantorovichMod{\CDN_\module{M}}}\left(\{\Pi(\omega)\}, \targetsettunnel{\tau_{f(g(h(n)))}^{-1}}{\omega}{l} \right) = 0 \text{,}
  \end{equation*}
   while $\CDN_{\module{M}}\circ\Pi \leq \CDN_{\module{N}}$.

Our goal is to check that $\Pi$ is the inverse of $\Theta$. Let $\omega\in\dom{\CDN_{\module{M}}}$ with $\omega\not=0$ and $l = \CDN_{\module{M}}(\omega)$.

Let $\varepsilon > 0$. There exists $N\in\N$ such that for all $n\geq  N$:
\begin{equation*}
  \Haus{\KantorovichMod{\CDN_{\module{N}}}}\left(\{\Theta(\omega)\}, \targetsettunnel{\tau_{f(g(h(n)))}}{\omega}{l} \right) \leq \frac{\varepsilon}{3\sqrt{2}}
\end{equation*}
and
\begin{equation*}
  \Haus{\KantorovichMod{\CDN_{\module{M}}}}\left(\{\Pi(\Theta(\omega))\}, \targetsettunnel{\tau_{f(g(h(n)))}^{-1}}{\Theta(\omega)}{l} \right) \leq \frac{\varepsilon}{3}
\end{equation*}
while $\tunnelextent{\tau_{f(g(h(n)))}} \leq \frac{\varepsilon}{3 \sqrt{2} H(2l,1)}$.

Let $\zeta \in \targetsettunnel{\tau_{f(g(h(n)))}^{-1}}{\Theta(\omega)}{l}$. Let $\eta\in\targetsettunnel{\tau_{f(g(h(n)))}}{\omega}{l}$. Note that by Definition (\ref{targetset-def}), we also have $\omega\in\targetsettunnel{\tau_{f(g(h(n)))}^{-1}}{\eta}{l}$.  This is the key reason behind the desired result. 

Using Proposition (\ref{targetset-morphism-prop}), we then estimate:
\begin{align*}
  \KantorovichMod{\CDN_{\module{M}}}(\Pi(\Theta(\omega)) , \omega)
  &\leq \KantorovichMod{\CDN_{\module{M}}}(\Pi(\Theta(\omega)) , \zeta) + \KantorovichMod{\CDN_{\module{M}}}(\zeta,\omega) \\
  &\leq \frac{\varepsilon}{3} + \sqrt{2}\left( \KantorovichMod{\module{N}}(\Theta(\omega) , \eta) + H(2l,1)\tunnelextent{\tau_{f(g(h(n)))}^{-1}} \right) \\
  &\leq \frac{\varepsilon}{3} + \sqrt{2}\left( \frac{\varepsilon}{3\sqrt{2}} + \frac{\varepsilon}{3\sqrt{2}} \right) \\
  &\leq \varepsilon \text{.}
\end{align*}
As $\varepsilon > 0$ is arbitrary, we conclude that $\norm{\Pi(\Theta(\omega)) - \omega}{\module{M}} = 0$, i.e. $\Pi(\Theta(\omega)) = \omega$ for all $\omega\in\dom{\CDN_{\module{M}}}$ (since $\Pi(\Theta(0)) = 0$ by linearity). By continuity, we conclude that $\Pi\circ\Theta$ is the identity on $\module{M}$. We would prove similarly that $\Theta\circ\Pi$ is the identity on $\module{N}$.

In particular, since $\Theta$ is onto, we then conclude from Equality (\ref{module-morphism-eq})  that $\Theta(a\omega) = \theta(a)\Theta(\omega)$ for all $a\in\A$ and $\omega\in\module{M}$.

Moreover, for all $\omega\in\dom{\CDN_{\module{M}}}$:
\begin{equation*}
  \CDN_{\module{N}}(\Theta(\omega)) \leq  \CDN_{\module{M}}(\omega) = \CDN_{\module{M}}(\Pi(\Theta(\omega))) \leq \CDN_{\module{N}}(\Theta(\omega)) \text{,}
\end{equation*}
so $\CDN_{\module{N}}(\Theta(\omega)) = \CDN_{\module{N}}(\omega)$.

Thus $(\theta,\Theta)$ is indeed a full modular quantum isometry. This concludes our proof.
\end{proof}

We conclude this section with some applications of our work so far.

\begin{example}
By Assertion (6) of Proposition (\ref{pseudo-metric-prop}), together with the conclusion of Theorem (\ref{zero-thm}), we note that the modular propinquity $\modpropinquity{F,H}$ is a metric up to full modular quantum isometry, thus slightly improving on \cite[Theorem 6.11]{Latremoliere16c}.
\end{example}

\begin{example}
Proposition (\ref{pseudo-metric-prop}) allows us to conclude that the conclusion of \cite{Latremoliere17c,Latremoliere18a} hold for the dual modular propinquity: therefore, Heisenberg modules form continuous families of {\gQVB s} for the dual-modular propinquity.
\end{example}

There is, by \cite[Example 3.15]{Latremoliere16c}, for any $p \in \N\setminus\{0\}$, a natural function $\qvba{\cdot}$ from {\Qqcms{F}s} to {\QVB{F}{H}}, where $H:x,y\mapsto 8 d F(x,y,x,y)$, defined by setting, for any {\Qqcms{F}} $(\B,\Lip_\B)$ for all $b_1,\ldots,b_p, c_1,\ldots,c_p \in \B$:
\begin{equation*}
  \CDN_\B^p (b_1,\ldots,b_n) = \max\left\{\norm{b_j}{\B}, \Lip_\B(\Re b_j),\Lip_\B(\Im b_j) : j \in J\right\}
\end{equation*}
and
\begin{equation*}
\inner{(b_1,\ldots.b_p)}{(c_1,\ldots,c_p)}{\B}^p = \sum_{j=1}^p b_j c_j^\ast \text{,}
\end{equation*}
and then $\qvba{\B,\Lip_\B} = (\B^p, \inner{\cdot}{\cdot}{\B}^p,\CDN_\B^p,\B,\Lip_\B)$. 

Proposition (\ref{pseudo-metric-prop}) and \cite[Theorem 7.2]{Latremoliere16c} then implies the function $\qvba{\cdot}$ is continuous from the class of {\Qqcms{F}s} endowed with the quantum propinquity, to the class of {\QVB{F}{H}s} endowed with the dual modular propinquity. We may then naturally ask about convergence of free modules under the weaker condition of convergence of their base quantum space for the dual propinquity. 

A first idea would be to start from a tunnel $(\D,\Lip,\pi,\rho)$ with domain $(\A,\Lip_\A)$ and $p\in\N\setminus\{0\}$, and then try to make a modular tunnel out of $\qvba{\D,\Lip}$. The natural surjections from $\D^p$ to $\A^p$ is given by $\pi^p : (d_1,\ldots,d_p)\in\D^p \mapsto (\pi(d_1),\ldots,\pi(d_p))$, but we encounter a difficulty: if $a\in\dom{\Lip_\A}$, then we  can find $d\in\sa{\D}$ with $\Lip_\D(d) = \Lip_\A(a)$ and $\pi(d) = a$. However, the best available estimate on $\norm{d}{\D}$ is that $\norm{d}{\D} \leq \norm{a}{\A} + 2 \Lip_\A(a) \tunnellength{\tau}$, and thus we can not expect $\pi^p$ to be a modular quantum isometry from $\CDN^p$ to $\CDN^p$. We thus propose a modified version of this construction, borrowing heavily form \cite[Section 7]{Latremoliere13} though adapted to the dual propinquity. This technique may prove helpful in more general situations (for the dual propinquity or the dual modular propinquity) when natural morphisms are only ``almost'' quantum isometries.

\begin{example}
  Let $(\A,\Lip_\A)$ and $(\B,\Lip_\B)$ be two {\Qqcms{F}s}. Let $(\D,\Lip_\D,\pi_\A,\pi_\B)$ be an $F$-tunnel from $(\A,\Lip_\A)$ to $(\B,\Lip_\B)$ and let $p \in N\setminus\{0\}$. Let $H:x,y\geq 0 \mapsto \max\{ 8 p F(x,y,x,y), 2 p x^2 y^2 \}$. We define $\module{P} = \A^p\oplus\D^p\oplus\B^p$, seen as a Hilbert $\alg{E}$-module where $\alg{E} = \A\oplus\D\oplus\B$ in the obvious way:
  \begin{equation*}
    \forall (a,d,b)\in\alg{E}, (\omega,\eta,\zeta)\in\module{P} \quad (a,d,b)\cdot(\omega,\eta,\zeta) = (a\omega,d\eta,b\zeta)
  \end{equation*}
  and for all $\omega_1,\omega_2\in\A^p, \eta_1,\eta_2\in\D^p, \zeta_1,\zeta_2\in\B^p$:
  \begin{equation*}
    \inner{(\omega_1,\eta_1,\zeta_1)}{(\omega_2,\eta_2,\zeta_2)}{\module{P}} = \left(\inner{\omega_1}{\omega_2}{\A}^p, \inner{\eta_1}{\eta_2}{\D}^p, \inner{\zeta_1}{\zeta_2}{\B}^p\right) \text{.}
  \end{equation*}
  We set $\gamma = \sqrt{1 + 4 p F(1+2\lambda,1+2\lambda,1,1)\lambda}$ where $\lambda=\tunnelextent{\tau}$. For all $(\omega,\eta,\zeta)\in\module{P}$, we define:
  \begin{equation*}
    \CDN(\omega,\eta,\zeta) = \max\left\{\CDN_\A^p(\omega),\CDN_\D^p(\eta),\CDN_\B^p(\zeta), \frac{\gamma}{\gamma - 1}\norm{\omega-\pi_\A^p(\eta)}{\A}^p, \frac{\gamma}{\gamma-1}\norm{\eta-\pi_\B^p(\zeta)}{\B}^p \right\} \text{,}
  \end{equation*}
  and for all $(a,d,b) \in \alg{E}$, we define:
  \begin{equation*}
    \Lip(a,d,b) = \max\left\{\Lip_\A(a), \Lip_\B(b), \Lip_\D(d), \frac{\gamma}{\gamma-1}\norm{a-\pi_\A(d)}{\A}, \frac{\gamma}{\gamma-1}\norm{b-\pi_\B(d)}{\B} \right\}\text{.}
  \end{equation*}
  Standard arguments similar to the one in \cite[Theorem 3.11]{Latremoliere14} show that $\Lip$ thus defined on $\alg{E}$ is a $F$-quasi Leibniz L-seminorm on $\alg{E}$. Similar argument to the proof of Theorem (\ref{triangle-thm}) prove that $\CDN$ is a D-norm, once we perform a direct computation to show that $\Lip(\inner{\omega}{\eta}{\module{P}}) \leq H(\CDN(\omega),\CDN(\eta))$ for all $\omega,\eta\in\module{P}$. Thus $(\module{P},\inner{\cdot}{\cdot}{\module{P}},\CDN,\alg{E},\Lip)$ is a {\QVB{F}{H}}.

  We then define $\Theta_\A : (\omega,\eta,\zeta)\in\module{P} \mapsto \omega \in \A^p$ and $\Theta_\B : (\omega,\eta,\zeta)\in\module{P} \mapsto \zeta \in \B^p$. We also define $\theta_\A : (a,d,b)\in\alg{E}\mapsto a\in\A$ and $\theta_\B : (a,d,b)\in\alg{E}\mapsto b\in\B$. A direct computation shows that $(\theta_\A,\Theta_\A)$ and $(\theta_\B,\Theta_\B)$ are Hilbert module morphisms.

  Let $a\in\sa{\A}$. Since $\tau$ is a tunnel, there exists $d\in\sa{\D}$ such that $\Lip_\D(d) = \Lip_\A(a)$ and $\pi_\A(d) = a$. It is immediate that $e = (a,d,\pi_\B(d)) \in \alg{E}$ satisfies $\theta_\A(e) = a$ and $\Lip(e) = \Lip_\A(a)$. Therefore, $\theta_\A$ is a quantum isometry from $(\alg{E},\Lip)$ to $(\A,\Lip_\A)$. Similarly, $\theta_\B$ is a quantum isometry from $(\alg{E},\Lip)$ to $(\B,\Lip_\B)$.

  Let now $(a_1,\ldots,a_p) \in \A^p$. For each $j\in\{1,\ldots,p\}$, we define $d_j = e_j + i f_j$ where $e_j, f_j \in \sa{\D}$ such that $\Lip(e_j) = \Lip(\Re a_j)$, $\Lip(f_j) = \Lip(\Im a_j)$, and $\pi_\A(e_j) = \Re a_j$, $\pi_\A(f_j) = \Im a_j$ --- which exists since $\tau$ is a tunnel. We also note that $\norm{e_j^2}{\D} \leq \norm{(\Re a_j)^2}{\A} + 2 F(1+2\lambda,1+2\lambda,1,1) \lambda$. 

If $\varphi \in \StateSpace(\D)$ then there exists $\psi \in \StateSpace(\D)$ such that $\Kantorovich{\Lip_\D}(\varphi,\psi) \leq \lambda$, and thus:
  \begin{align*}
    \varphi\left(\sum_{j=1}^p d_j d_j^\ast\right)
    &= \sum_{j=1}^p \varphi\left(\Re(d_j)^2\right) + \sum_{j=1}^p\varphi\left(\Im(d_j)^2\right) + 2\sum_{j=1}^p \varphi\left(\Lie{\Re d_j}{\Im d_j}\right)\\
    &\leq \sum_{j=1}^p\left( \psi\left((\Re(a_j))^2\right) + \psi\left((\Im(a_j))^2\right) + 2 \psi\left(\Lie{\Re a_j}{\Im a_j}\right) \right) + 4 p \beta \lambda \\
    &= \psi\left(\sum_{j=1}^p a_j a_j^\ast\right) + 4 p \beta \lambda \leq \gamma^2 \text{.}
  \end{align*}
  Therefore, $\norm{(d_j)_{j=1}^p}{\D^p} \leq \gamma$ since $\varphi$ was arbitrary in $\StateSpace(\D)$. Hence it follows that $\norm{(\pi_\B(d_j))_{j=1}^p}{\B^p} \leq \gamma$. Altogether, $\CDN((a_j)_{j=1}^p, (\frac{1}{\gamma} d_j)_{j=1}^p, (\frac{1}{\gamma} b_j)_{j=1}^p) \leq 1$ as desired. Hence, $\Theta_\A$ is a modular quantum isometry.  The same holds for $\Theta_\B$, so $(\module{P},\inner{\cdot}{\cdot}{\module{P}},\CDN,\alg{E},\Lip)$ is an $(F,H)$-modular tunnel. We can then compute, using techniques already encountered, that the extend of $\tau_p$ is no more than:
\begin{equation*}
  \frac{2(\gamma-1)}{\gamma} + \lambda = \frac{2\sqrt{1+4 p F(1+2\lambda,1+2\lambda,1,1) \lambda} - 1}{\sqrt{1+4 p F(1+2\lambda,1+2\lambda,1,1) \lambda}} + \lambda\text{.} 
\end{equation*}

Therefore, by continuity of $F$, and if we now set $\lambda = \dpropinquity{F}((\A,\Lip_\A),(\B,\Lip_B))$, then:
\begin{equation*}
  \dmodpropinquity{F,H}(\qvba{\A,\Lip_\A},\qvba{\B,\Lip_\B}) \leq \frac{2\sqrt{1+4 p F(1+2\lambda,1+2\lambda,1,1) \lambda} - 1}{\sqrt{1+4 p F(1+2\lambda,1+2\lambda,1,1) \lambda}} + \lambda\text{.}
\end{equation*}

In particular, $\qvba{\cdot}$ is continuous form the class of {\Qqcms{F}s} endowed with the modular $F$-propinquity to the class of {\QVB{F}{H}s} endowed with the dual modular $(F,H)$-propinquity.
\end{example}

\section{The Dual Modular Propinquity for {\gMVB s}}

We now explain how to extend the propinquity to {\gMVB s}. Our strategy follows the pattern we detailed above about {\gQVB s}, and thus begin with a notion of a metrical tunnel. A metrical tunnel consists of a modular tunnel and a tunnel between {\qcms s}, paired together in a natural manner.

\begin{notation}
  If $\mathds{M} = (\module{M},\inner{\cdot}{\cdot}{\module{M}},\CDN,\A,\Lip_\A,\B,\Lip_\B)$ is a {\MVB{F}{G}{H}}, then $\mathds{M}_\flat = (\module{M},\inner{\cdot}{\cdot}{\module{M}},\CDN,\A,\Lip_\A)$ is a {\QVB{F}{H}}. We set $\mathrm{alt}(\mathds{M}) = (\B,\Lip_\B)$.
\end{notation}

\begin{definition}\label{metrical-tunnel-def}
  Let $\mathds{A}^j = (\module{M}_j,\inner{\cdot}{\cdot}{\module{M}_j},\CDN_{\module{M}_j},\A^j,\Lip^j,\B^j,\Lip_j)$ for $j\in\{1,2\}$.

  A \emph{metrical tunnel} $(\tau,\tau')$ from $\mathds{A}^1$ to $\mathds{A}^2$ is given by:
   \begin{enumerate}
     \item an $(F,H)$-modular tunnel $\tau = (\mathds{D},(\theta_1,\Theta_1),(\theta_2,\Theta_2))$ from $\mathds{A}^1_\flat$ to $\mathds{A}^2_\flat$, where we write $\mathds{D} = (\module{P},\inner{\cdot}{\cdot}{\module{P}},\CDN,\D,\Lip_\D)$,
     \item an $F$-tunnel $\tau' = (\D',\Lip_\D',\pi^1,\pi^2)$ from $\mathrm{alt}(\mathds{A}^1)$ to $\mathrm{alt}(\mathds{A}^2)$,
     \item $\module{P}$ is also a $\D'$-left module,
     \item $\forall d\in \sa{\D'} \quad \forall \xi \in \module{P} \quad \CDN(d\xi) \leq G(\norm{d}{\D'},\Lip_\D'(d),\CDN(\xi))$,
     \item for all $j\in\{1,2\}$, the pair $(\pi^j, \Theta^j)$ is a left module morphism from the left $\D'$-module $\module{P}$ to the left $\A^j$-module $\module{M}_j$.
   \end{enumerate}
\end{definition}

In particular, in Definition (\ref{metrical-tunnel-def}), the tuple $(\module{P},\CDN,\D,\Lip_\D,\D',\Lip_\D')$ is a {\MVB{F}{G}{H}}. 

\begin{definition}
  The \emph{extent} of a metrical tunnel $(\tau,\tau')$ is $\max\left\{\tunnelextent{\tau},\tunnelextent{\tau'}\right\}$.
\end{definition}

\begin{figure}[t]
\begin{equation*}
  \xymatrix{
    & & (\D',\Lip'_\D) \ar@{.>}[d] \ar@/_/@{>>}[llddd]_{\pi_\A}  \ar@/^/@{>>}[rrddd]^{\pi_\B}  & & \\
    & & (\module{P},\CDN_{\module{P}}) \ar@/_/@{>>}[ldd]_{\Theta_\A} \ar@/^/@{>>}[rdd]^{\Theta_\B} & & \\
     & & (\D,\Lip_\D) \ar@{>>}[ldd]_{\theta_\A} \ar@{>>}[rdd]^{\theta_\B} \ar@{-->}[u] &  \\
    (\A',\Lip'_\A) \ar@{.>}[r] & (\module{M},\CDN_{\module{M}}) & & (\module{N},\CDN_{\module{N}}) & (\B',\Lip'_\B) \ar@{.>}[l] \\
    & (\A,\Lip_\A) \ar@{-->}[u] & & (\B,\Lip_\B) \ar@{-->}[u]
    }
\end{equation*}
\caption{A metrical modular tunnel}
\end{figure}
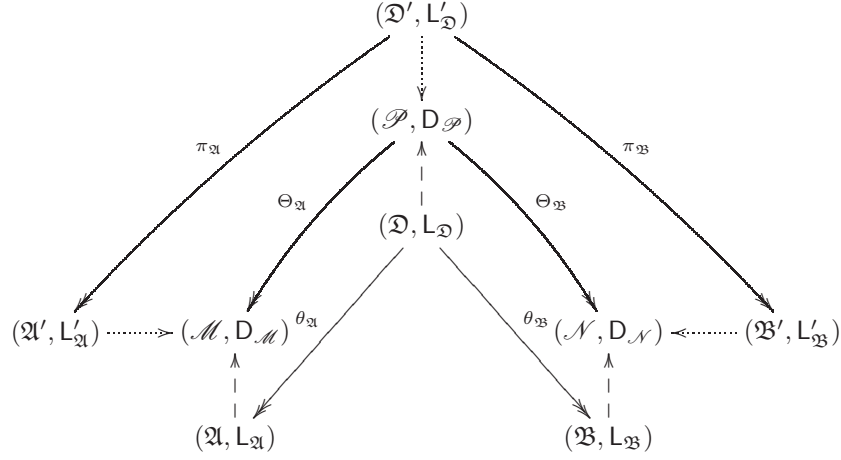

We can extend the proof of Theorem (\ref{triangle-thm}) to metrical tunnels.

\begin{proposition}
  Let $\mathds{A}$, $\mathds{B}$ and $\mathds{E}$ be three {\MVB{F}{G}{H}s} and let $(\tau_1,\upsilon_1)$ and $(\tau_2,\upsilon_2)$ be two $(F,G,H)$-metrical tunnels, respectively from $\mathds{A}$ to $\mathds{B}$ and $\mathds{B}$ to $\mathds{E}$. If $\varepsilon > 0$ then there exists a $(F,G,H)$-metrical tunnel from $\mathds{A}$ to $\mathds{E}$ whose extent is no more than $\tunnelextent{\tau_1} + \tunnelextent{\tau_2} + \varepsilon$.
\end{proposition}

\begin{proof}
  We apply Theorem (\ref{triangle-thm}) to the modular tunnels $\tau_1$ and $\tau_2$ to obtain a modular tunnel $\tau$ from $\mathds{A}_\flat$ to $\mathds{E}_\flat$ of extent at most $\tunnelextent{\tau_1} + \tunnelextent{\tau_2} + \varepsilon$.

  We then apply \cite[Theorem 3.1]{Latremoliere14} to the tunnels $\upsilon_1$ and $\upsilon_2$ to obtain a tunnel $\upsilon$ from $\mathrm{alt}(\mathds{A})$ to $\mathrm{alt}(\mathrm{B})$ with $\tunnelextent{\upsilon}\leq\tunnelextent{\tau_1} + \tunnelextent{\tau_2} + \varepsilon$. 

Writing $\upsilon_1 = (\D'_1,\mathsf{Q}_1,\alpha,\rho)$, $\upsilon_2 = (\D'_2,\mathsf{Q}_2,\mu,\beta)$ and $\upsilon = (\D,\mathsf{Q},.,.)$, we then observe that the modular quasi-Leibniz relation holds: if $d_1\in\D'_1,d_2\in\D'_2,(\omega,\eta)\in\module{B}$ then:
  \begin{align*}
    \CDN(d_1\omega,d_2\eta) 
    &= \max\left\{\CDN_1(d_1\omega_1),\CDN_2(d_2\eta), \frac{1}{\varepsilon}\norm{\Theta_\B(\omega)-\Pi_\B(\eta)}{\module{M}} \right\} \\
    &\leq \max\left\{\begin{array}{l}
                       G(\norm{d_1}{\D_1'},\mathsf{Q}_1(d_1),\CDN_1(\omega)),\\
                       G(\norm{d_2}{\D_2'},\mathsf{Q}_2(d_2),\CDN_2(\eta)),\\
                       \frac{\norm{d_1}{\D_1} \norm{\Theta_\B(\omega)-\Pi_\B(\eta)}{} + \norm{\rho(d_1)-\mu(d_2)}{} \norm{\eta}{}}{\varepsilon}
                       \end{array} \right\}\\
      &\leq G(\norm{(d_1',d_2')}{\D'},\mathsf{Q}(d_1,d_2), \CDN(\omega,\eta) ) \text{.}
  \end{align*}

  It is then straightforward to check that $(\tau,\upsilon)$ is the desired metrical tunnel.
\end{proof}

We now define the metric analogue of the dual modular propinquity.

\begin{definition}
  Let $\mathcal{C}$ be a nonempty class of {\MVB{F}{G}{H}s}. A class $\mathcal{T}$ of $(F,G,H)$-tunnels is \emph{appropriate} if:
  \begin{enumerate}
    \item $\left\{ \tau : \exists \upsilon \quad (\tau,\upsilon) \in \mathcal{T} \right\}$ is appropriate for $\{\mathds{E}_\flat:\mathds{E}\in\mathcal{C}\}$,
    \item $\left\{ \upsilon : \exists \tau \quad (\tau,\upsilon)\in\mathcal{T} \right\}$ is appropriate for $\{\mathrm{alt}(\mathrm{E}): \mathrm{E}\in\mathcal{C}\}$.
  \end{enumerate}
\end{definition}

We immediately check that the class of all $(F,G,H)$-metrical tunnels is indeed appropriate for the class of all {\MVB{F}{G}{H}s}. If $\mathds{A}$ and $\mathds{B}$ is a pair of {\MVB{F}{G}{H}s} and $\mathcal{T}$ is a collection of $(F,G,H)$-metrical tunnels, we once again denote the class of all tunnels from $\mathds{A}$ to $\mathds{B}$ in $\mathcal{T}$ by $\tunnelset{\mathds{A}}{\mathds{B}}{\mathcal{T}}$.

\begin{definition}\label{metrical-propinquity-def}
  Let $\mathcal{C}$ be a nonempty class of {\MVB{F}{G}{H}s} and $\mathcal{T}$ a class of tunnels appropriate for $\mathcal{C}$. The \emph{dual-metrical $\mathcal{T}$-propinquity} $\dmetpropinquity{\mathcal{T}}$ is defined for any $\mathds{A},\mathds{B} \in \mathcal{C}$ by:
  \begin{equation*}
    \dmetpropinquity{\mathcal{T}}\left(\mathds{A},\mathds{B}\right) = \inf\left\{ \tunnelextent{\tau} : \tau\in\tunnelset{\mathds{A}}{\mathds{B}}{\mathcal{T}} \right\} \text{.}
  \end{equation*}
\end{definition}

The dual metrical propinquity is a metric up to full metrical quantum isometry. To prove this, we make one observation.

\begin{proposition}\label{action-prop}
  Let $(\tau,\tau')$ be a metrical tunnel between two {\MVB{F}{G}{H}s} $\mathds{A}_1$ and $\mathds{A}_2$, where we will use the notations of Definition (\ref{metrical-tunnel-def}). Let $\omega\in\module{M}_1$ and $a\in\sa{\B_1}$. If $\eta\in\targetsettunnel{\tau}{\omega}{l}$ and $b \in \targetsettunnel{\tau'}{a}{l'}$ then:
  \begin{equation*}
    b \eta \in \targetsettunnel{\tau}{a\omega}{G(\norm{a}{\A} + 2 l \tunnelextent{\tau'},l',l)} \text{.}
  \end{equation*}
\end{proposition}

\begin{proof}
  Let $d\in \sa{\D'}$ and $\xi \in \module{P}$ such that $\pi^1(d) = a$, $\pi^2(d) = b$, and $\Lip_\D'(d) \leq l$, while $\Theta_1(\xi) = \omega$, $\Theta_2(\xi) = \eta$ and $\CDN_{\module{P}}(\eta) \leq l$.
  We have:
  \begin{equation*}
    \Theta_1(d \xi) = \pi^1(d)\Theta_1(\xi) = a\omega \text{ while }\Theta_2(d \xi) = b \eta\text{.} 
  \end{equation*}

  Moreover, by Definition (\ref{metrical-bundle-def}):
    \begin{equation*}
      \CDN_{\module{P}}(d\xi) \leq G(\norm{d}{\D'}, l, l) \leq G(\norm{a}{\A} + 2 l \tunnelextent{\tau'}{},l ,l) \text{.}
    \end{equation*}

    This concludes our proof by Definition (\ref{targetset-def}).
\end{proof}

\begin{definition}
  Let $\mathds{A}$ and $\mathds{B}$ be two {\gMVB s}. A \emph{full metrical quantum isometry} $(\theta,\Theta,\pi)$ is given by:
  \begin{enumerate}
    \item a full modular quantum isometry $(\theta,\Theta)$ from $\mathds{A}_\flat$ to $\mathds{B}_\flat$,
    \item a full quantum isometry $\pi$ from $\mathrm{alt}(\mathds{A})$ to $\mathrm{alt}(\mathds{B})$ such that $(\pi,\Theta)$ is a module morphism.
  \end{enumerate}
\end{definition}

\begin{theorem}
  The dual  metrical propinquity is a metric up to full metrical quantum isometry.
\end{theorem}

\begin{proof}
  The dual metrical propinquity $\dmetpropinquity{\mathcal{T}}$ is a pseudo-metric on the class $\mathcal{C}$ by a similar argument as in Proposition (\ref{pseudo-metric-prop}). We thus focus on proving that distance zero implies the existence of a full metrical quantum isometry.

  Let $\mathds{A} = (\module{M},\inner{\cdot}{\cdot}{\module{M}},\CDN_{\module{M}},\A,\Lip_\A,\A',\Lip'_\A)$ and $\mathds{B} = (\module{N},\inner{\cdot}{\cdot}{\module{N}},\CDN_{\module{N}},\B,\Lip_\B,\B',\Lip'_\B)$ be two {\MVB{F}{G}{H}s}.

  If $\dmetpropinquity{}(\mathds{A},\mathds{B}) = 0$ then there exists a sequence of metrical tunnels $(\tau_n,\upsilon_n)_{n\in\N}$ with $\lim_{n\rightarrow\infty} \tunnelextent{\tau_n,\upsilon_n} = 0$ --- we may as well assume $\tunnelextent{\tau_n,\upsilon_n} \leq 1$ for all $n\in\N$. Using both \cite{Latremoliere13b} and the proof of Theorem (\ref{zero-thm}), there exists a strictly increasing function $f : \N\rightarrow\N$, a full modular quantum isometry $(\theta,\Theta)$ and a full quantum isometry $\pi$ such that:
\begin{enumerate}
  \item  for all $a\in\dom{\Lip_\A'}$ and $l\geq \Lip_\A'(a)$, the sequence $\left(\targetsettunnel{\upsilon_{f(n)}}{a}{l}\right)_{n\in\N}$ converges to $\{\pi(a)\}$ for $\Haus{\norm{\cdot}{\A'}}$,
  \item for all $\omega\in\dom{\CDN_{\module{M}}}$ and $l\geq \CDN_{\module{M}}(\omega)$, the sequence $\left(\targetsettunnel{\tau_{f(n)}}{\omega}{l}\right)_{n\in\N}$ converges to $\Theta(\omega)$ for $\Haus{\norm{\cdot}{\module{N}}}$.
\end{enumerate}

 let $a\in\sa{\A'}$ and $\omega\in \module{M}$. If for all $n\in\N$, we choose $a_n \in \targetsettunnel{\upsilon_{f(n)}}{a'}{\Lip_\A'(a')}$ and $\omega_n \in \targetsettunnel{\tau_{f(n)}}{\omega}{\CDN_{\module{M}}(\omega)}$, then Proposition (\ref{action-prop}) applies to give us:
 \begin{multline*}
   a_n \omega_n \in \targetsettunnel{\tau_{f(n)}}{a\omega}{G(\norm{a}{\A} + 2 \Lip'_\A(a) \tunnelextent{\upsilon_{f(n)}}),\Lip'_\A(a),\CDN_{\module{M}}(\omega)} \\ \subseteq \targetsettunnel{\tau_{f(n)}}{a\omega}{G(\norm{a}{\A} + 2 \Lip'_\A(a)),\Lip'_\A(a),\CDN_{\module{M}}(\omega)} \text{.}
 \end{multline*}

 Since $(a_n)_{n\in\N}$ converges to $\pi(a)$, $(\omega_n)_{n\in\N}$ converges to $\Theta(\omega)$ and $a_n\omega_n$ converges to $\Theta(a\omega)$, we conclude that $\pi(a)\Theta(\omega) = \Theta(a\omega)$.

  By linearity and continuity, it then follows that $\pi(a)\Theta(\omega) = \Theta(a\omega)$ for all $a\in\A'$ and $\omega\in\module{M}$, as desired. Thus $(\theta,\Theta,\pi)$ is a full metrical quantum isometry from $\mathds{A}$ to $\mathds{B}$.
\end{proof}

\section{Completeness}

We prove that the dual-modular propinquity is complete, which is one of the main justification for its introduction. We begin by checking that taking a quotient of a D-norm gives a D-norm.

\begin{lemma}\label{quotient-lemma}
  Let $(\module{M},\inner{\cdot}{\cdot}{\module{M}},\CDN,\A,\Lip)$ be a {\QVB{F}{H}} with $H$ continuous. If $(\B,\Lip_\B)$ is a {\Qqcms{F}}, $\module{N}$ is a Hilbert $\B$-module, and $(\pi,\Pi)$ is a surjective Hilbert module morphism from $\module{M}$ to $\module{N}$ such that $\pi$ is a quantum isometry from $(\A,\Lip)$ to $(\B,\Lip_\B)$, and if $\CDN'$ is defined as:
  \begin{equation*}
    \forall \omega \in \module{N} \quad \CDN'(\omega) = \inf \left\{ \CDN(\eta) : \Pi(\eta) = \omega \right\}
  \end{equation*}
  allowing $\CDN'$ to take the value $\infty$, then $(\module{N},\inner{\cdot}{\cdot}{\module{N}},\CDN',\B,\Lip_\B)$ is a {\QVB{F}{H}}.
\end{lemma}

\begin{proof}
  Note that $\module{N} = \Pi(\module{M}) = \Pi(\mathrm{cl}(\dom{\CDN})) \subseteq \mathrm{cl}(\Pi(\dom{\CDN}))$, and if $\omega\in\Pi(\dom{\CDN})$ then $\CDN'(\omega) < \infty$, so $\Pi(\dom{\CDN}) \subseteq\dom{\CDN'}$ (in fact, these two sets are obviously equal). So $\CDN'$ has dense domain. It is a standard argument to show that $\CDN'$ is a norm on its domain since $\CDN$ is and $\Pi$ is a $\C$-linear map.

  Let $\omega\in\module{N}$ such that $\CDN'(\eta)\leq 1$. For all $n\in\N$, there exists $\eta_n \in \module{M}$ such that $\Pi(\eta_n) = \omega$ and $\CDN(\eta_n)\leq 1 + \frac{1}{n+1}$. As $\{\eta\in\module{M} : \CDN(\eta)\leq 2\}$ is compact, there exists a subsequence $(\eta_{f(n)})_{n\in\N}$ converging to some $\eta\in\module{M}$; by lower semi-continuity of $\CDN$, we have $\CDN(\eta)\leq 1$, and by continuity of $\Pi$, we have $\Pi(\eta) = \omega$. Therefore, the unit ball of $\CDN'$ is the image of the unit ball of $\CDN$, and since the unit ball of $\CDN$ is compact and $\Pi$ is continuous, the unit ball of $\CDN'$ is compact as well.

  Moreover, if $\omega\in\module{N}$, then for any $\eta\in \Pi^{-1}(\{\omega\})$, we have:
  \begin{equation*}
    \norm{\omega}{\module{N}} = \norm{\Pi(\eta)}{\module{N}} \leq \norm{\eta}{\module{M}} \leq \CDN(\eta) \text{.}
  \end{equation*}
  Hence $\norm{\omega}{\module{N}} \leq \CDN'(\omega)$ by construction.

  Let $\omega,\eta \in \module{R}$. Let $\varepsilon > 0$. Since $H$ is continuous at  $(\CDN_\infty(\omega),\CDN_\infty(\eta))$, there exists $\delta > 0$, such that if $|\CDN_\infty(\omega)-t| \leq \delta$ and $|\CDN_\infty(\eta)-s| \leq \delta$ then $|H(t,s) - H(\CDN_\infty(\omega),\CDN_\infty(\eta))| < \varepsilon$.

 There exists $\zeta,\xi \in \module{M}$ with $\CDN(\zeta) \leq \CDN(\omega) + \delta$ and $\CDN(\xi) \leq \CDN(\eta) + \delta$. We then compute:
\begin{align*}
  \Lip_\B\left(\inner{\omega}{\eta}{\infty}\right) 
  &\leq \Lip_\B(\inner{\Pi(\zeta)}{\Pi(\xi)}{\module{N}}) \\
  &= \Lip_\B(\pi(\inner{\zeta}{\xi}{\module{M}}))\\
  &\leq \Lip(\inner{\zeta}{\eta}{\module{M}}) \\
  &\leq H(\CDN(\zeta),\CDN(\eta)) \\
  &\leq H(\CDN'(\omega) + \delta,\CDN'(\eta) + \delta) \\
  &\leq H(\CDN'(\omega),\CDN'(\eta)) + \varepsilon \text{.}
\end{align*}
As $\varepsilon > 0$ is arbitrary, we conclude that:
\begin{equation*}
  \Lip_\B\left(\inner{\omega}{\eta}{\module{N}}\right) \leq H(\CDN'(\omega),\CDN'(\eta))\text{.}
\end{equation*}
Thus, $(\module{N},\inner{\cdot}{\cdot}{\module{N}}, \CDN', \B, \Lip_\B)$ is a {\QVB{F}{H}}.
\end{proof}

\begin{notation}
  Let $(\A_j)_{j \in J}$ be a family of unital C*-algebras indexed by $J$. We denote by $\prod_{j \in J}\A_j$ the C*-algebra:
  \begin{equation*}
    \left\{ (a_j)_{j \in J} : \forall j \in J \quad a_j \in \A_j\text{ and }\sup_{j \in J} \norm{a_j}{\A_j} < \infty \right\}
  \end{equation*}
  for the norm $\norm{(a_j)_{j \in J}}{\prod_{j \in J}\A_j} = \sup_{j \in J}\norm{a_j}{\A_j}$.
\end{notation}

\begin{theorem}\label{completeness-thm}
  The metric $\dmodpropinquity{F,H}$ is complete on the class of all {\QVB{F}{H}s} if $F$ and $H$ are continuous.
\end{theorem}

\begin{proof}
Let us be given a sequence $\mathds{A}_n = (\module{M}_n,\inner{\cdot}{\cdot}{n}, \CDN_n, \A_n, \Lip_n)_{n\in\N}$ of {\QVB{F}{H}s} such that:
\begin{equation*}
  \sum_{n\in\N} \dmodpropinquity{F,H}(\mathds{A}_n,\mathds{A}_{n+1})  < \infty \text{.}
\end{equation*}

For each $n\in\N$, let:
\begin{equation*}
  \tau_n = \left( \mathds{D}_n, (\theta_n,\Theta_n), (\sigma_n,\Sigma_n) \right)
\end{equation*}
be a modular tunnel from $\mathds{A}_n$ to $\mathds{A}_{n+1}$ with $\tunnellength{\tau_n} \leq \dmodpropinquity{}(\mathds{A}_n,\mathds{A}_{n+1}) + \frac{1}{2^n}$, where:
\begin{equation*}
  \mathds{D}_n = \left( \module{P}_n, \inner{\cdot}{\cdot}{}^n, \CDN^n, \D_n, \Lip^n \right)\text{.}
\end{equation*}

Fix $N\in\N$, and denote by $\N_N$ the subset $\{ k \in \N : k \geq N \}$ of $\N$.

We begin by recalling some constructions from \cite{Latremoliere13b}. Set:
\begin{equation*}
  \alg{S}_N = \left\{ (d_n)_{n\in\N_N} \in \prod_{n\in\N_N}\D_n \middle\vert
    \begin{array}{l}
      \forall n \in \N_N \quad \sigma_n(d_n) = \theta_{n+1}(d_{n+1})\text{,} \\
      \sup\{\Lip^n(\Re d_n),\Lip^n(\Im d_n):n\in\N\} < \infty
      \end{array}
    \right\} \text{.}
\end{equation*}
For $d = (d_n)_{n\in\N_N} \in \alg{S}_N$, we set $\mathsf{S}_N(d) = \sup_{n\in\N_N}\Lip^n(d_n)$. Let $\alg{E}_n$ be the closure of $\alg{S}_N$ in the C*-algebra $\prod_{n\in\N_N}\D_n$. By \cite[Proposition 6.17]{Latremoliere13b}, the pair $(\alg{E}_N,\mathsf{S}_N)$ is a {\Qqcms{F}} (the proof of \cite[Proposition 6.17]{Latremoliere13b} applies to $F$-quasi-Leibniz when $F$ is continuous as discussed in \cite{Latremoliere15}). We define the map $\pi_N : (d_n)_{n\in\N_N} \in \alg{E}_N  \mapsto \theta_N(d_N) \in \A_N$. By \cite[Corollary 6.9]{Latremoliere13b}, the map $\pi_N$ is a quantum isometry from $(\alg{E}_N,\mathsf{S}_N)$ onto $(\A_N,\Lip_N)$.

Now, we define $\alg{I}_N = \left\{ (d_n)_{n\in\N_N} : \lim_{n\rightarrow\infty} \norm{d_n}{\D_n} = 0 \right\}$ --- note that $\alg{I}_N$ is a closed, two sided ideal in $\alg{E}_N$. Let $\alg{F}^N = \bigslant{\alg{E}_N}{\alg{I}_N}$ and let $\mathsf{Q}^N$ be the quotient seminorm of $\mathsf{S}_N$ on $\alg{F}^N$. Using \cite[Lemma 6.20, Lemma 6.21]{Latremoliere13b}, we in fact know that $(\alg{F}^N,\mathsf{Q}^N)$ is a {\Qqcms{F}} which is fully quantum isometric to $(\alg{F}^{N+1},\mathsf{Q}^{N+1})$. It will be helpful to explicit the full quantum isometry for our purpose, so we present a construction here.

Let:
\begin{equation*}
  m_N : (d_k)_{k\in\N_N} \in \alg{E}_N \longmapsto (d_k)_{k\in\N_{N+1}} \in \alg{E}_{N+1}\text{.}
\end{equation*}
We note that $m_N$ is a *-epimorphism and that the quotient seminorm of $\mathsf{S}_N$ via $m_N$ is $\mathsf{S}_{N+1}$, for the following reason: if $d = (d_n)_{n\in\N_{N+1}}$ then, since $\tau_N$ is a tunnel, there exists $d_N \in \sa{D_N}$ such that $\Lip^N(d_N) \leq \Lip^{N+1}(d_{N+1}) \leq \mathsf{S}_{N+1}(d)$ and $\sigma_N(d_N) = \theta_{N+1}(d_{N+1})$,  and thus $(d_n)_{n\in\N_N} \in \alg{E}_N$, with $m_N((d_n)_{n\in\N_N}) = d$ and $\mathsf{S}_N((d_n)_{n\in\N_N}) = \mathsf{S}_{N+1}(d)$. Of course if $d\in\alg{E}_{N+1}$ is not self-adjoint, the previous construction can be applied to $\Re d$ and $\Im d$ to prove that $m_N$ is a surjection.

Moreover, it is immediate that $m_N^{-1}(\alg{I}_{N+1}) = \alg{I}_N$. Therefore, there exists a *-automorphism $y_N : \alg{F}^N \rightarrow \alg{F}^{N+1}$ uniquely determined by the following commutative diagram:
\begin{equation*}
  \xymatrix{
    \alg{E}_N \ar@{>>}[d]_{p_N} \ar@{>>}[r]^{m_N} & \alg{E}_{N+1} \ar@{>>}[d]^{p_{N+1}}\\
    \alg{F}^N \ar[r]^{y_N} & \alg{F}^{N+1}
    }
\end{equation*}
where $p_N : \alg{E}_N \twoheadrightarrow \alg{F}^N$ and $p_{N+1} : \alg{E}_{N+1} \twoheadrightarrow \alg{F}^{N+1}$ are the canonical surjections. Moreover, if $a \in \alg{F}^{N}$ and we write $a' = y_N(a)$, and if $\varepsilon > 0$, then there exists $d' = (d'_n)_{n\in\N_{N+1}} \in \alg{E}_{N+1}$ with $\mathsf{S}_{N+1}(d') \leq \mathsf{Q}^{N+1}(a') + \varepsilon$ and $p_{N+1}(q') = a'$. As we have seen, there exists $d \in \alg{E}_N$ such that $\mathsf{S}_N(d) = \mathsf{S}_{N+1}(d')$ and $m_N(d) = d'$. Using our commuting diagram, we then have $p_N(d) = a$. Thus $\mathsf{Q}^N(a) \leq \mathsf{S}_N(d') = \mathsf{S}_{N+1}(d) \leq \mathsf{Q}^{N+1}(y_N(a)) + \varepsilon$. As $\varepsilon > 0$ is arbitrary, the *-automorphism $y_N$ is $1$-Lipschitz. On the other hand, for any $d\in\alg{E}_N$ with $p_N(d) = a$, we estimate:
\begin{equation*}
 \mathsf{S}_N(d) \geq \mathsf{S}_{N+1}(m_N(d)) \geq \mathsf{Q}^{N+1}(p_{N+1}\circ m_N(d)) = \mathsf{Q}^{N+1}(y_N(a))\text{,}
\end{equation*}
so  $\mathsf{Q}^N(a) \geq  \mathsf{Q}^{N+1}(y_N(a))$ so $y_N$ is a full quantum isometry. Whenever convenient, we identify $(\alg{F}^N,\Lip_\infty^N)$ with $(\alg{F}^{N+1},\Lip_\infty^{N+1})$ as {\Qqcms{F}} and drop the superscript $N$.

By \cite[Corollary 6.25]{Latremoliere13b}, for all $\varepsilon > 0$, there exists $N\in\N$ such that if $n\geq N$ then $(\alg{E}_n,\mathsf{S}_n,\pi_n,p_n)$ is a tunnel from $(\A_n,\Lip_n)$ to $(\alg{F},\mathsf{Q})$ of extent at most $\varepsilon$. This is how we proved that $(\A_n,\Lip_n)_{n\in\N}$ converges to $(\alg{F},\mathsf{Q})$.

We now extend this construction to modular tunnels.

Fix $N \in \N$ again. We define:
\begin{equation*}
  \module{Q}_N = \left\{ (\omega_n)_{n\in\N_N} \in \prod_{n\in\N_N} \module{P}_n \middle\vert 
    \begin{array}{l}
      \forall n\in\N_N \quad \Sigma_n(\omega_n) = \Theta_{n+1}(\omega_{n+1}) \\
      \sup_{n\in\N_N}\CDN^n(\omega_n) < \infty
    \end{array} \right\}\text{.}
\end{equation*}

For all $\omega\in\module{Q}_N$, we set:
\begin{equation*}
  \CDN_{\geq N}(\omega) = \sup_{n\in\N_N} \CDN^n(\omega_n) \text{.}
\end{equation*}
We also set for all $\omega = (\omega_n)_{n\in\N_N} ,\eta = (\eta_n)_{n\in\N_N} \in\module{N}_N$:
\begin{equation*}
  \inner{\omega}{\eta}{N} = \left(\inner{\omega_n}{\eta_n}{\D_n}\right)_{n\in\N_N}\text{.}
\end{equation*}
Note that $\inner{\omega}{\eta}{N} \in \alg{S}_N$ by construction since for all $n\in\N_N$:
\begin{align*}
  \theta_{n+1}\left(\inner{\omega_{n+1}}{\eta_{n+1}}{\D_{n+1}} \right) 
  &= \inner{\Theta_{n+1}(\omega_{n+1})}{\Theta_{n+1}(\eta_{n+1})}{\D_{n+1}} \\
  &= \inner{\Sigma_{n}(\omega_n)}{\Sigma_n(\eta_n)}{\D_n} \\
  &= \sigma_n\left( \inner{\omega_n}{\eta_n}{\D_n}  \right) \text{.}
\end{align*}

Let $\module{Q}_N$ be the closure of $\module{Q}_N$ in $\prod_{n\in\N_N}\module{P}_n$. It is immediate to check that the $\alg{S}_N$-inner product of $\module{R}_N$ extends to an $\alg{E}_N$-inner product on $\module{Q}_N$.

We now want to prove that $(\module{Q}_N,\inner{\cdot}{\cdot}{N}, \CDN_{\geq N}, \alg{E}_N,\mathsf{S}_N)$ is a {\QVB{F}{H}}. We already know that $(\alg{E}_N,\mathsf{S}_N)$ is a {\Qqcms{F}} and that $(\module{Q}_N,\inner{\cdot}{\cdot}{N})$ is a Hilbert $\alg{E}_N$-module. We are left proving that $\CDN_{\geq N}$ is a D-norm.

First, if $(\omega_n)_{n\in\N_N}$ then we check:
\begin{equation*}
  \norm{(\omega_n)_{n\in\N}}{\module{Q}_N} = \sup_{n\in\N_N} \norm{\omega_n}{\module{P}_n} \leq \sup_{n\in\N_N} \CDN^n(\omega_n) = \CDN_{\geq N}( (\omega_n)_{n\in\N} )\text{.}
\end{equation*}

Second, we note that if $\omega=(\omega_n)_{n\in\N_N}, \eta=(\eta_n)_{n\in\N_N} \in \module{Q}_N$ then:
\begin{align*}
  \mathsf{S}_N\left(\inner{\omega}{\eta}{N}\right) 
  &= \sup_{n\in\N_N} \Lip^n\left(\inner{\omega_n}{\eta_n}{\D_n}\right) \\
  &\leq \sup_{n\in\N_N} H\left( \CDN^n(\omega_n), \CDN^n(\eta_n)  \right) \\
  &\leq H(\CDN_{\geq N}(\omega), \CDN_{\geq N}(\eta))
\end{align*}
since $H$ is increasing in both its variables.

Last, note that:
\begin{equation*}
  \left\{ \omega\in\module{Q}_N : \CDN_{\geq N}(\omega)\leq 1\right\} \subseteq \prod_{n\in\N_N} \left\{ \omega \in \module{P}_N : \CDN^n(\omega) \leq 1 \right\}
\end{equation*}
and the right-hand side is compact by Tychonoff theorem. On the other hand, as the supremum of lower semi-continuous functions, $\CDN_{\geq N}$ is lower semi-continuous, and thus its unit ball is closed. As a closed set of a compact set, the unit ball of $\CDN_{\geq N}$ is compact. Hence, $(\module{Q}_N,\inner{\cdot}{\cdot}{N}, \CDN_{\geq N}, \alg{E}_N, \mathsf{S}_N)$ is indeed a {\QVB{F}{H}}.

Let $\module{J}_N = \left\{ (\omega_n)_{n\in\N_N} \in \module{Q}_N : \lim_{n\rightarrow\infty}\norm{\omega_n}{\module{N}_n} = 0 \right\}$. Note that $\module{J}_N$ is a closed submodule of $\module{Q}_N$.

We define $\module{R}^N$ as $\bigslant{\module{Q}_N}{\module{J}_N}$. Now if $d = (d_n)_{n\in\N_N}$ and $d' = (d'_n)_{n\in\N_N}$ in $\alg{E}_N$ with $d-d' \in \alg{I}_N$ and if $\omega = (\omega_n)_{n\in\N_N}$ and $\omega' = (\omega'_n)_{\N_N}$ in $\module{Q}_N$ with $\omega-\omega'\in\module{J}_N$, then:
\begin{equation*}
  \norm{d_n \omega_n - d'_n \omega'_n}{\module{P}_n} \leq \norm{d_n \omega_n - d_n \omega'_n}{\module{P}_n} + \norm{(d'_n - d_n)\omega'_n}{\module{P}_n} \xrightarrow{n\rightarrow\infty} 0 \text{,}
\end{equation*}
so $d\omega-d'\omega' \in \module{J}_N$. From this, it is easy to see that $\module{R}^N$ is a left $\alg{F}^N$-module. It is also easily checked that if $d=(d_n)_{n\in\N_N}, d'=(d'_n)_{n\in\N_N}, e=(e_n)_{n\in\N_N}, e' = (e'_n)_{n\in\N_N} \in \alg{E}_N$ and $\omega=(\omega_n)_{n\in\N_N}, \omega'=(\omega'_n)_{n\in\N_N}, \eta=(\eta_n)_{n\in\N_N}, \eta'=(\eta'_n)_{n\in\N_N} \in \module{Q}_N$, with $d-d', e-e'\in \alg{I}_N$ and $\omega-\omega',\eta-\eta' \in \module{J}_N$, and for all $n\geq N$:
\begin{multline*}
  \norm{\inner{d_n\omega_n}{e_n\eta_n}{\module{P}_n} - \inner{d'_n\omega'_n}{e'_n\eta'_n}{\module{P}_n}}{\D_n} \\
  \begin{split}
    &\leq \norm{\inner{d_n \omega_n - d'_n \omega'_n}{e_n\eta_n}{\module{P}_n}}{\D_n} + \norm{\inner{d'_n\omega'_n}{e_n\eta_n - e'_n\eta'_n}{\module{P}_n}}{\D_n} \\
    &\leq \norm{d_n\omega_n-d'_n\omega'_n}{\module{P}_n} \norm{e_n\eta_n}{\module{P}_n} + \norm{d'_n \omega'_n}{\module{P}_n} \norm{e_n \eta_n -e'_n \eta'_n}{\module{P}_n} \\
    &\xrightarrow{n\rightarrow\infty} 0 \text{,}
  \end{split}
\end{multline*}
so $\module{R}^N$ is in fact a Hilbert $\alg{F}^N$-module, in a natural way. We denote the canonical surjection $\module{Q}_N \twoheadrightarrow \module{R}^N$ by $q_N$. The above computation show that $(p_N,q_N)$ is a Hilbert module morphism from $\module{Q}_N$ (over $\alg{E}_N$) to $\module{R}^N$ (over $\alg{F}^N$). We denote the inner product on $\mathcal{R}^N$ as $\inner{\cdot}{\cdot}{\module{R}^N}$.'

Let:
\begin{equation*}\forall \omega \in \module{Q}_N \quad \CDNa^N(\omega) = \inf\left\{ \CDN_{\geq N}(\eta) : q_N(\eta) = \omega \right\} \text{,}
\end{equation*}
allowing for the possibility that $\CDNa^N$ takes the value $\infty$.

We already know that $p_N$ is a quantum isometry from $(\D_n,\Lip^n)$ to $(\A_n,\Lip_n)$. Hence by Lemma (\ref{quotient-lemma}), we conclude that $(\module{R}^N,\inner{\cdot}{\cdot}{\module{R}^N}, \CDNa^N, \alg{F}^N, \mathsf{Q}^N)$ is indeed a {\QVB{F}{H}}.

Just as was the case for the {\gQqcms s}, we check that the {\gQVB s} $(\module{R}^N,\inner{\cdot}{\cdot}{\module{R}^N}, \CDNa^N, \alg{F}^N, \mathsf{Q}^N)$ are all fully quantum isometric.

The method is identical. We define:
\begin{equation*}
  t_N : (\omega_n)_{n\in\N_N} \in \module{Q}_N \longmapsto (\omega_n)_{n\in\N_{N+1}} \in \module{Q}_{N+1}
\end{equation*}
and easily check that $(m_N,t_N)$ is a Hilbert module morphism from the Hilbert $\alg{E}_N$-module $\module{Q}_N$ to the Hilbert $\alg{E}_{N+1}$-module $\module{Q}_{N+1}$. We know that the map $m_N$ is a quantum isometry. Now, if $\omega = (\omega_n)_{n\in\N_{N+1}} \in \module{Q}_{N+1}$, then since $\tau_N$ is a modular tunnel, there exists $\omega_N \in \module{Q}_N$ such that $\Sigma_N(\omega_N) = \Theta_{N+1}(\omega_{N+1})$ and $\CDN^N(\omega_N) \leq \CDN^{N+1}(\omega_{N+1}) \leq \CDN_{\geq N+1}(\omega)$. Of course $t_N((\omega)_{\N_N}) = \omega$ and $\CDN_{\geq N}((\omega_n)_{n\in\N_N}) = \CDN_{\geq N+1}(\omega)$.

It is also immediate that $t_N^{-1}(\module{J}_{N+1}) = \module{J}_N$. We thus have a unique surjective linear map $z_n : \module{R}^N \rightarrow \module{R}^{N+1}$ such that the following diagram commutes:
\begin{equation*}
  \xymatrix{
    \module{Q}_N \ar@{>>}[d]_{q_N} \ar@{>>}[r]^{t_N} & \module{Q}_{N+1} \ar@{>>}[d]^{q_{N+1}}\\
    \module{R}^N \ar[r]^{z_N} & \module{R}^{N+1}
    }
\end{equation*}
and $(y_N,z_N)$ is a Hilbert module isomorphism from $\module{R}^N$ to $\module{R}^{N+1}$. Following the same argument as before, we then show that $(y_N,z_N)$ is a full modular quantum isometry: if $\zeta\in \dom{\CDNa^N}$ and if $\eta = z_N(\zeta)$, and if $\varepsilon > 0$, then there exists $\omega \in \module{Q}_{N+1}$ such that $q_{N+1}(\omega)=\eta$ and $\CDN_{\geq N+1}(\omega) \leq \CDNa^N(\eta) + \varepsilon$; then there exists $\omega'\in\dom{\CDN_{\geq N}}$ such that $t_N(\omega')=\omega$ and $\CDN_{\geq N}(\omega') = \CDN_{\geq N+1}(\omega)$ (note that we used the compactness of the balls for D-norms). As our diagram commutes, $q_N(\omega) = \zeta$. Therefore $\CDNa^N(\zeta) \leq \CDN_{\geq N}(\omega') \leq \CDNa^W(\eta) + \varepsilon$. Thus $\CDNa^{N+1} \circ z_N \leq \CDNa^N$. On the other hand, if $\omega \in \module{R}^N$ and $\zeta\in\module{Q}_N$ with $q_N(\zeta) = \omega$, then $\CDN_{\geq N}(\zeta) \geq \CDN_{\geq N+1}(t_N(\zeta)) \geq \CDNa^{N+1}(q_{N+1}\circ t_N(\zeta))$ so $\CDNa^N \geq \CDNa^{N+1}\circ z_N$.

Let us write $(\module{R},\inner{\cdot}{\cdot}{\module{R}}, \CDNa, \alg{F}, \mathsf{Q})$ for any representative of this isometry class.

We want to use the {\QVB{F}{H}}
\begin{equation*}
  (\module{Q}_N,\inner{\cdot}{\cdot}{N},\CDN_{\geq N},\alg{E}_N,\mathsf{S}_N)
\end{equation*}
to construct a modular tunnel from
\begin{equation*}
  (\module{M}_N,\inner{\cdot}{\cdot}{\module{M}_N},\CDN_N,\A_N,\Lip_N)
\end{equation*}
to
\begin{equation*}
  (\module{R}^N,\inner{\cdot}{\cdot}{\module{R}^N}, \CDNa^N, \alg{F}^N, \mathsf{Q}^N)\text{.}
\end{equation*}

Let us define a modular quantum isometry from $(\module{Q}_N,\inner{\cdot}{\cdot}{N},\CDN_{\geq N},\alg{E}_N,\mathsf{S}_N)$ onto $(\module{M}_N,\inner{\cdot}{\cdot}{\module{M}_N},\CDN_N,\A_N,\Lip_N)\text{.}$

If $\omega = (\omega_n)_{n\in\N_N} \in \module{Q}_N$, then we set $\Pi_N(\omega) = \Theta_N(\omega_N) \in \module{M}_N$. We easily check that if $a=(a_n)\in\alg{S}_N$ then:
\begin{align*}
  \Pi_N( a \omega ) = \Theta_N(a_N \omega_N) = \theta_N(a_N) \Theta_N(\omega_N) = \pi_N(a) \Pi_N(\omega)
\end{align*} 
and
\begin{align*}
  \inner{\Pi_N(\omega)}{\Pi_N(\eta)}{\module{M}_N} = \inner{\Theta_N(\omega_N)}{\Theta_N(\eta_N)}{\module{M}_N} = \theta_N(\inner{\omega_N}{\eta_N}{\module{P}_N}) = \pi_N(\inner{\omega}{\eta}{\module{Q}_N}) \text{.}
\end{align*}
So $(\pi_N,\Pi_N)$ is a Hilbert modular morphism from $\module{Q}_N$ onto $\module{M}_N$. We also know that $\pi_N$ is a quantum isometry, so it is sufficient to prove the following in order to conclude $(\pi_N,\Pi_N)$ is a modular isometry.

Let $\eta \in \module{M}_N$ and $\CDN_N(\eta) < \infty$. Write $l = \CDN_N(\eta)$. As $\tau_N$ is a modular tunnel, there exists $\omega_N \in \module{P}_N$ with $\Theta_N(\omega_N) = \eta$ and $\CDN^N(\omega_N) = l$. Now, assume that for some $n\geq N$, there exists $\omega_n \in \module{P}_n$ with $\CDN(\omega_n) \leq l$. As $\tau_n$ is a modular tunnel, there exists $\omega_{n+1} \in \module{P}_{n+1}$ with $\CDN_{n+1}(\omega_{n+1}) \leq l$ and $\Theta_{n+1}(\omega_{n+1}) = \Sigma_n(\omega_n)$ (since $\CDN_n(\Sigma_n(\omega_n)) \leq l$). By induction, we conclude that there exists $(\omega_n)_{n\in\N_N}$ with $\Theta_N(\omega_N) = \eta_N$ and $\CDN_{\geq N}( (\omega_n)_{n\in\N_N} ) = l = \CDN_N(\eta)$. This proves that $(\pi_N,\Pi_N)$ is a modular isometry.

Therefore, writing $\mathds{E}_N = (\module{Q}_N,\inner{\cdot}{\cdot}{N}, \CDN_{\geq N}, \alg{E}_N, \mathsf{S}_N)$, then
\begin{equation*}
  \tau_{\geq N} = (\mathds{E}_N, (\pi_N,\Pi_N), (p_N, q_N))
\end{equation*}
is a modular tunnel from $\mathds{A}_N$ to $\left(\module{R}^N,\inner{\cdot}{\cdot}{\module{R}^N},\CDNa^N,\alg{F}^N,\mathsf{Q}^N\right)$. By definition, the extent of this tunnel is given by the extent of its underlying basic tunnel.

As stated above, by \cite[Corollary 6.25]{Latremoliere13b}, for all $\varepsilon > 0$, there exists $N\in\N$ such that if $n\geq N$, then the extent of $\tau_n$ is no more than $\varepsilon$.

Therefore:
\begin{equation*}
  \lim_{n\rightarrow\infty} \dmodpropinquity{}\left(\mathds{A}_n, \left(\module{R},\inner{\cdot}{\cdot}{}, \CDNa, \alg{F}, \mathsf{Q} \right) \right) = 0\text{.}
\end{equation*}

We now conclude our theorem. If $(\mathds{A}_n)_{n\in\N}$ is a Cauchy sequence of {\QVB{F}{H}s} for $\dmodpropinquity{F,H}$, then it admits a subsequence $(\mathds{A}_{f(n)})_{n\in\N}$ satisfying $\sum_{n\in\N} \dmodpropinquity{F,H}(\mathds{A}_{f(n)},\mathds{A}_{f(n+1)}) < \infty$. We then can apply our work here to show that $(\mathds{A}_{f(n)})_{n\in\N}$ has a limit for $\dmodpropinquity{F,H}$, and as a Cauchy sequence with a convergent subsequence, $(\mathds{A}_n)_{n\in\N}$ converges for $\dmodpropinquity{F,H}$, as desired.
\end{proof}

\begin{theorem}
  The metric $\dmetpropinquity{F,G,H}$ is complete on the class of all {\MVB{F}{G}{H}s} if $F$, $G$ and $H$ are continuous.
\end{theorem}

\begin{proof}
  Let us be given a sequence $\left(\mathds{M}_n\right)_{n\in\N} = (\module{M}_n,\inner{\cdot}{\cdot}{n}, \CDN_n, \A_n, \Lip_n,\B_n,\mathsf{M}_n)_{n\in\N}$ of {\MVB{F}{G}{H}s} such that:
\begin{equation*}
  \sum_{n\in\N} \dmetpropinquity{F,G,H}(\mathds{M}_n,\mathds{M}_{n+1})  < \infty \text{.}
\end{equation*}

For each $n\in\N$, let $(\tau_n,\upsilon_n)$ be a metrical $(F,G,H)$-tunnel from $\mathds{M}_n$ to $\mathds{M}_{n+1}$ of extent at most $\dmetpropinquity{F,G,H}(\mathds{A}_n,\mathds{A}_{n+1}) + \frac{1}{2^n}$ and write $\mathds{A}_n = (\mathds{M}_n)_\flat$.

We use the same notation as in the proof of Theorem (\ref{completeness-thm}), and in fact start from its conclusion. We write $\upsilon = (\alg{Y}_n,\mathsf{J}_n,\mu_n,\varpi_n)$ for all $n\in\N$. 

Fix $N\in\N$. We write $\mathsf{T}_N((d_n)_{n\in\N_N}) = \sup_{n\in\N_N} \mathsf{J}_n(d_n)$ for all $(d_n)_{n\in\N_N} \in \prod_{n\in\N_N}\alg{Y}_n$.

By \cite{Latremoliere13b} and by Definition (\ref{metrical-propinquity-def}), if we set:
\begin{equation*}
  \alg{W}_N = \mathrm{cl}\left(\left\{ (d_n)_{n\in\N_N} \in \prod_{n\in\N_N} \alg{Y}_n : 
    \begin{array}{l}
      \forall_{\N_N} n \quad \varpi_n(d_n) = \mu_{n+1}(d_{n+1}) \\
      \max\{\mathsf{T}_N((d_n)_{n\in\N}), \mathsf{T}_N((d_n)_{n\in\N})\} < \infty
    \end{array}
\right\}\right)
\end{equation*}
where the closure is taken in $\prod_{n\in\N_N}\alg{Y}_n$.  

By \cite{Latremoliere13b}, if $\alg{I}_N = \{ (d_n)_{n\in\N_N}\in\alg{W}_N : \lim_{n\rightarrow\infty}\norm{d_n}{\alg{Y}_n} = 0 \}$ and if $\mathsf{U}_N$ is the quotient norm of $\mathsf{T}_N$ on $\alg{U}_N = \bigslant{\alg{W}_N}{\alg{I}_N}$ then $\left(\alg{U}_N,\mathsf{U}_N\right)$ is a {\Qqcms{F}}. Let $r_N$ be the canonical surjection $\alg{W}_N\twoheadrightarrow\alg{U}_N$ and $k_N : (d_n)_{n\in\N_N} \in \alg{W}_N \mapsto \varphi_N(d_N)$.

Again by \cite{Latremoliere13b}, the quadruple $\upsilon_{\geq N} = (\alg{W}_N,\mathsf{T}_N, k_N, r_N)$ is a tunnel and
\begin{equation*}
  \lim_{N\rightarrow\infty} \tunnelextent{\upsilon_{\geq N}} = 0\text{.}
\end{equation*}
Moreover, for all $N, M \in\N$, as in the proof of Theorem (\ref{zero-thm}), the {\Qqcms{F}s} $(\alg{U}_N,\mathsf{U}_N)$ and $(\alg{U}_M,\mathsf{U}_M)$ are all fully quantum isometric in a canonical manner. Write $(\alg{U},\mathsf{U})$ for $(\alg{U}_0,\mathsf{U}_0)$.

We now prove that $\alg{U}$ acts on $\module{R}$ and satisfies the appropriate quasi-Leibniz identity. It is immediate to check that setting:
\begin{equation*}
  \forall d = (d_n)_{n\in\N} \in \alg{W}, \omega = (\omega_n)_{n\in\N} \in \module{Q}_N \quad d\omega = (d_n \omega_n)_{n\in\N}
\end{equation*}
turns $\module{Q}_0$ into a left $\alg{W}_0$-module. Following the same method as in the proof of Theorem (\ref{completeness-thm}), we then conclude that $\module{R}$ is indeed a left $\alg{U}$-module.

Now, let $a \in \alg{U}_0$. Let $d = (d_n)_{n\in\N} \in \alg{W}_0$ with $r(d) = a$. Let $\omega \in \module{R}_0$, and let $\eta = (\eta_n)_{n\in\N_N} \in \module{Q}_0$ with $q_0(\eta) = \omega$. Since $G$ is increasing in each of its variables:
  \begin{align*}
    \CDNa(a\omega) 
    &\leq \sup_{n\in\N}\CDN^n(d_n \eta_n) \\
    &\leq \sup_{n\in\N} G(\norm{d_n}{\alg{Y}_n}, \mathsf{J}_n(d_n), \CDN^n(\eta_n)) \\
    &\leq G(\norm{d}{\alg{W}_0}, \mathsf{T}_0(d), \CDN_{\geq 0}(\eta))\text{.}
  \end{align*}
As $G$ is continuous, we conclude that:
\begin{align*}
  \CDNa(a\omega) 
  &\leq  \inf \left\{ G(\norm{d}{\alg{W}_0}, \mathsf{T}_0(d), \CDN_{\geq 0}(\eta)) : d\in r^{-1}_0(a), \eta \in q_0^{-1}(\omega) \right\} \\
  &\leq G(\norm{a}{\alg{U}},\mathsf{U}(a), \CDNa(\omega)) \text{.}
\end{align*}

Therefore, $(\module{R}, \inner{\cdot}{\cdot}{\module{R}}, \CDNa, \alg{F}, \mathsf{Q}, \alg{U}, \mathsf{U})$ is a {\MVB{F}{G}{H}}. Now, by definition, $(\tau_{\geq N},\upsilon_{\geq N})$ is a metrical tunnel and
\begin{equation*}
  \lim_{n\rightarrow\infty} \tunnelextent{\tau_{\geq N},\upsilon_{\geq N}} = 0
\end{equation*}
as desired. This completes our proof.
\end{proof}

\bibliographystyle{amsplain}
\bibliography{../thesis}

\providecommand{\bysame}{\leavevmode\hbox to3em{\hrulefill}\thinspace}
\providecommand{\MR}{\relax\ifhmode\unskip\space\fi MR }
\providecommand{\MRhref}[2]{%
  \href{http://www.ams.org/mathscinet-getitem?mr=#1}{#2}
}
\providecommand{\href}[2]{#2}
\begin{thebibliography}{10}

\bibitem{Kaad18}
{K}. {A}guilar and {J}. {K}aad, \emph{The podleś sphere as a spectral metric
  space.}, J. Geom. Phys. \textbf{133} (2018), 260--278.

\bibitem{Rieffel15b}
{M}. {C}hrist and {M.}~{A.} {R}ieffel, \emph{Nilpotent group
  {$C^\ast$-algebras}-algebras as compact quantum metric spaces}, Canadian
  Mathematical Bulletin \textbf{60} (2017), no.~1, 77–94, ArXiv: 1508.00980.

\bibitem{Connes89}
A.~{C}onnes, \emph{Compact metric spaces, {F}redholm modules and
  hyperfiniteness}, Ergodic Theory and Dynamical Systems \textbf{9} (1989),
  no.~2, 207–220.

\bibitem{Connes97}
{A}. {C}onnes, {M}. {D}ouglas, and {A}. {S}chwarz, \emph{Noncommutative
  geometry and matrix theory: Compactification on tori}, JHEP \textbf{9802}
  (1998), hep-th/9711162.

\bibitem{dandrea13}
{P}.~{M}artinetti {{F}. {D'A}ndrea}, {F}.~{L}izzi, \emph{Spectral geometry with
  a cut-off: topological and metric aspects}, Submitted (2013), 44 pages,
  ArXiv: 1305.2605.

\bibitem{Latremoliere16}
{{F}. {L}atrémolière} and {J}. {P}acker, \emph{Noncommutative solenoids and
  the {G}romov-{H}ausdorff propinquity}, Proceedings of the American
  Mathematical Society \textbf{145} (2017), no.~5, 1179–1195, ArXiv:
  1601.02707.

\bibitem{Grosse97}
H.~Grosse, C.~Klimčík, and P.~Prešnajder, \emph{Field theory on a
  supersymmetric lattice}, Comm. Math. Phys. \textbf{185} (1997), no.~1,
  155–175. \MR{1463037}

\bibitem{kerr02}
D.~{K}err, \emph{Matricial quantum {G}romov-{H}ausdorff distance}, Journal of
  Functional Analysis \textbf{205} (2003), no.~1, 132–167, math.OA/0207282.

\bibitem{kerr09}
{D}. {K}err and {H}. {L}i, \emph{On {G}romov–{H}ausdorff convergence of
  operator metric spaces}, J. Oper. Theory \textbf{1} (2009), no.~1, 83–109.

\bibitem{Latremoliere20a}
{T}. {Landry}, {M}. {L}apidus, and {F}. {L}atrémolière, \emph{Metric
  approximations of the spectral triple on the sierpinky gasket and other
  fractals}, Submitted, 30 pages.

\bibitem{Latremoliere05}
{F}. {L}atrémolière, \emph{Approximation of the quantum tori by finite
  quantum tori for the quantum {G}romov-{H}ausdorff distance}, Journal of
  Functional Analysis \textbf{223} (2005), 365–395, math.OA/0310214.

\bibitem{Latremoliere05b}
\bysame, \emph{Bounded-lipschitz distances on the state space of a
  {C*}-algebra}, Tawainese Journal of Mathematics \textbf{11} (2007), no.~2,
  447–469, math.OA/0510340.

\bibitem{Latremoliere12b}
\bysame, \emph{Quantum locally compact metric spaces}, Journal of Functional
  Analysis \textbf{264} (2013), no.~1, 362–402, ArXiv: 1208.2398.

\bibitem{Latremoliere13c}
\bysame, \emph{Convergence of fuzzy tori and quantum tori for the quantum
  {G}romov–{H}ausdorff propinquity: an explicit approach.}, Münster Journal
  of Mathematics \textbf{8} (2015), no.~1, 57–98, ArXiv: math/1312.0069.

\bibitem{Latremoliere15c}
\bysame, \emph{Curved noncommutative tori as {L}eibniz compact quantum metric
  spaces}, Journal of Mathematical Physics \textbf{56} (2015), no.~12, 123503,
  16 pages, ArXiv: 1507.08771.

\bibitem{Latremoliere13b}
\bysame, \emph{The dual {G}romov–{H}ausdorff propinquity}, Journal de
  Mathématiques Pures et Appliquées \textbf{103} (2015), no.~2, 303–351,
  ArXiv: 1311.0104.

\bibitem{Latremoliere16b}
\bysame, \emph{Equivalence of quantum metrics with a common domain}, Journal of
  Mathematical Analysis and Applications \textbf{443} (2016), 1179–1195,
  ArXiv: 1604.00755.

\bibitem{Latremoliere13}
\bysame, \emph{The quantum {G}romov-{H}ausdorff propinquity}, Transactions of
  the American Mathematical Society \textbf{368} (2016), no.~1, 365–411.

\bibitem{Latremoliere15}
\bysame, \emph{A compactness theorem for the dual {G}romov-{H}ausdorff
  propinquity}, Indiana University Journal of Mathematics \textbf{66} (2017),
  no.~5, 1707–1753, ArXiv: 1501.06121.

\bibitem{Latremoliere14}
\bysame, \emph{The triangle inequality and the dual {G}romov-{H}ausdorff
  propinquity}, Indiana University Journal of Mathematics \textbf{66} (2017),
  no.~1, 297–313, ArXiv: 1404.6633.

\bibitem{Latremoliere18a}
\bysame, \emph{Convergence of {H}eisenberg modules for the modular
  {G}romov-{H}ausdorff propinquity}, Submitted (2018), 34 pages.

\bibitem{Latremoliere18g}
\bysame, \emph{The gromov-hausdorff propinquity for metric spectral triples},
  Submitted, ArXiv: 1811.10843 (2018).

\bibitem{Latremoliere17c}
\bysame, \emph{Actions of categories by lipschitz morphisms on limits for the
  gromov-hausdorff propinquity}, . J. Geom. Phys. \textbf{146} (2019), 103481,
  31 pp., ArXiv: 1708.01973.

\bibitem{Latremoliere18c}
\bysame, \emph{Convergence of {C}auchy sequences for the covariant
  {G}romov-{H}ausdorff propinquity}, Journal of Mathematical Analysis and
  Applications \textbf{469} (2019), no.~1, 378–404, ArXiv: 1806.04721.

\bibitem{Latremoliere18b}
\bysame, \emph{The covariant {G}romov-{H}ausdorff propinquity}, Accepted in
  Studia Mathematica (2019), 29 pages, ArXiv: 1805.11229.

\bibitem{Latremoliere17a}
\bysame, \emph{{H}eisenberg modules over quantum {$2$}-tori are metrized
  quantum vector bundles}, Accepted in Canadian Journal of Mathematics (2019),
  38 pages, ArXiv: 1703.07073.

\bibitem{Latremoliere16c}
\bysame, \emph{The modular {G}romov–{H}ausdorff propinquity}, Dissertationes
  Mathematicae \textbf{544} (2019), 1–70, ArXiv: 1608.04881.

\bibitem{li03}
H.~{L}i, \emph{{$C^\ast$}-algebraic quantum {G}romov-{H}ausdorff distance},
  (2003), ArXiv: math.OA/0312003.

\bibitem{Madore91}
J.~{M}adore, \emph{The commutative limit of a matrix geometry}, Journal of
  Math. Phys. \textbf{32} (1991), no.~2, 332–335.

\bibitem{Madore}
\bysame, \emph{An introduction to noncommutative differential geometry and its
  physical applications}, 2nd ed., London Mathematical Society Lecture Notes,
  vol. 257, Cambridge University Press, 1999.

\bibitem{Junge16}
Qiang~Zeng {Marius Junge}, Sepideh~Rezvani, \emph{Harmonic analysis approach to
  gromov–hausdorff convergence for noncommutative tori}, Comm. Math. Phys.
  (2016), 81 pages, 1612.02735.

\bibitem{Seiberg99}
{N}athan {S}eiberg and {E}dward {W}itten, \emph{String theory and
  noncommutative geometry}, JHEP \textbf{9909} (1999), no.~32, ArXiv:
  hep-th/9908142.

\bibitem{Ozawa05}
{N}. {O}zawa and M.~A. {R}ieffel, \emph{Hyperbolic group {$C\sp\ast$}-algebras
  and free product {$C\sp\ast$}-algebras as compact quantum metric spaces},
  Canadian Journal of Mathematics \textbf{57} (2005), 1056–1079, ArXiv:
  math/0302310.

\bibitem{Rieffel17a}
{M}. {R}ieffel, \emph{Vector bundles for "matrix algebras converge to the
  sphere"}, Submitted (2017), 46 pages, arXiv:1711.04054.

\bibitem{Rieffel74}
M.~A. {R}ieffel, \emph{Induced representations of {C*}-algebras}, Advances in
  Math. \textbf{13} (1974), 176–257.

\bibitem{Rieffel98a}
\bysame, \emph{Metrics on states from actions of compact groups}, Documenta
  Mathematica \textbf{3} (1998), 215–229, math.OA/9807084.

\bibitem{Rieffel99}
\bysame, \emph{Metrics on state spaces}, Documenta Math. \textbf{4} (1999),
  559–600, math.OA/9906151.

\bibitem{Rieffel02}
\bysame, \emph{Group {$C\sp\ast$}-algebras as compact quantum metric spaces},
  Documenta Mathematica \textbf{7} (2002), 605–651, ArXiv: math/0205195.

\bibitem{Rieffel00}
\bysame, \emph{{G}romov-{H}ausdorff distance for quantum metric spaces},
  Memoirs of the American Mathematical Society \textbf{168} (2004), no.~796,
  1–65, math.OA/0011063.

\bibitem{Rieffel10c}
\bysame, \emph{{L}eibniz seminorms for "matrix algebras converge to the
  sphere"}, Clay Mathematics Proceedings \textbf{11} (2010), 543–578, ArXiv:
  0707.3229.

\bibitem{Rieffel15}
\bysame, \emph{Matricial bridges for "matrix algebras converge to the sphere"},
  Submitted (2015), 31 pages, ArXiv: 1502.00329.

\bibitem{Sauvageot91}
{J.-L.} {S}auvageot, \emph{Le {problème} de dirichlet dans les
  {$C^\ast$}-{algèbres}.}, Journal of Funct. Anal. \textbf{101} (1991), no.~1,
  50–73.

\bibitem{Schwarz98}
Albert Schwarz, \emph{Morita equivalence and duality}, Nuclear Phys. B
  \textbf{534} (1998), no.~3, 720–738. \MR{1663471}

\bibitem{Wallet12}
{J}. {W}allet, \emph{Connes distance by examples: Homothetic spectral metric
  spaces}, Rev. Math. Phys. \textbf{24} (2012), no.~9, ArXiv: 1112.3285.

\bibitem{Wu06b}
{W}. {W}u, \emph{Quantized {G}romov-{H}ausdorff distance}, J. Funct. Anal.
  \textbf{238} (2006), no.~1, 58–98, ArXiv: math.OA/0503344.

\end{thebibliography}
\vfill

\end{document}